\documentclass[1 [leqno,11pt]{amsart}
\usepackage{amssymb}
\usepackage{amssymb, amsmath}
\def\inte#1{
\displaystyle\mathop{#1\kern0pt}^\circ }

\let\pa=\partial
\let\al=\alpha
\let\b=\beta

\let\d=\delta
\let\e=\varepsilon

\let\lam=\lambda
\let\r=\rho

\let\f=\frac

\let\p=\psi

\let\D=\Delta

\let\wt=\widetilde
\let\wh=\widehat

\def\cA{{\mathcal A}}
\def\cB{{\mathcal B}}
\def\cC{{\mathcal C}}

\def\cF{{\mathcal F}}

\def\cR{{\mathcal R}}
\def\cS{{\mathcal S}}

\def\pa{\partial}
\def\grad{\nabla}

\def\dH{\dot{H}}

\def\virgp{\raise 2pt\hbox{,}}
\def\cdotpv{\raise 2pt\hbox{;}}

\def\eqdefa{\buildrel\hbox{\footnotesize def}\over =}

\def\C{\mathop{\mathbb C\kern 0pt}\nolimits}
\def\DD{\mathop{\mathbb D\kern 0pt}\nolimits}
\def\EE{\mathop{{\mathbb E \kern 0pt}}\nolimits}
\def\K{\mathop{\mathbb K\kern 0pt}\nolimits}
\def\N{\mathop{\mathbb N\kern 0pt}\nolimits}
\def\Q{\mathop{\mathbb Q\kern 0pt}\nolimits}
\def\R{\mathop{\mathbb R\kern 0pt}\nolimits}
\def\SS{\mathop{\mathbb S\kern 0pt}\nolimits}
\def\ZZ{\mathop{\mathbb Z\kern 0pt}\nolimits}
\def\TT{\mathop{\mathbb T\kern 0pt}\nolimits}
\def\P{\mathop{\mathbb P\kern 0pt}\nolimits}

\newcommand{\la}{\lambda}

\newcommand{\Z}{{\ZZ}}

\def\dv{\mbox{div}}

\def\dive{\mathop{\rm div}\nolimits}

\def\Supp{\mathop{\rm Supp}\nolimits\ }

\def\no{\noindent}
\def\na{\nabla}
\def\p{\partial}

\newcommand{\beq}{\begin{equation}}
\newcommand{\eeq}{\end{equation}}
\newcommand{\ben}{\begin{eqnarray}}
\newcommand{\een}{\end{eqnarray}}
\newcommand{\beno}{\begin{eqnarray*}}
\newcommand{\eeno}{\end{eqnarray*}}
\newcommand{\andf}{\quad\hbox{and}\quad}

\newtheorem{defi}{Definition}[section]
\newtheorem{thm}{Theorem}[section]
\newtheorem{lem}{Lemma}[section]
\newtheorem{rmk}{Remark}[section]

\newtheorem{prop}{Proposition}[section]

\newcommand{\vv}[1]{\boldsymbol{#1}}

\begin{document}

\title[Global Small solutions of 2-D MHD system ]
{ Global small solutions of 2-D incompressible MHD  system }
\author[F. Lin]{Fanghua Lin}
\address [F. Lin]%
{Courant Institute, New York University, New York, NY 10012}
\email{linf@cims.nyu.edu}
\author[L. Xu]{Li Xu}
\address[L. Xu]%
{LSEC, Institute of Computational Mathematics, Academy of Mathematics and Systems Science, CAS\\
Beijing 100190, CHINA}
\email{ xuliice@lsec.cc.ac.cn}
\author[P. Zhang]{Ping Zhang}%
\address[P. Zhang]
 {Academy of Mathematics and Systems Science and  Hua Loo-Keng Key Laboratory of Mathematics,
  Chinese Academy of Sciences, Beijing 100190, CHINA} \email{zp@amss.ac.cn}
\date{5/May/2013}
\maketitle
\begin{abstract} In this paper, we consider the global wellposedness of
 2-D incompressible magneto-hydrodynamical  system with smooth initial
data which is close to some non-trivial steady state. It is a coupled
system between the Navier-Stokes equations and
a free transport equation with an universal nonlinear coupling
structure. The main difficulty of the proof lies in exploring the
dissipative mechanism of the system. To achieve this and to avoid
the difficulty of propagating anisotropic regularity for the free
transport equation, we first reformulate our system \eqref{1.1} in
the Lagrangian coordinates \eqref{a14}. Then we employ anisotropic
Littlewood-Paley analysis to establish the key {\it a priori}
$L^1(\R^+; Lip(\R^2))$ estimate for the Lagrangian velocity field
$Y_t$.  With this estimate, we can prove the global wellposedness of
\eqref{a14} with smooth and small initial data by using the energy
method. We emphasize that the algebraic structure of \eqref{a14} is
crucial for the proofs to work. The global wellposedness of the
original system \eqref{1.1} then follows by a suitable change of
variables.
\end{abstract}

\noindent {\sl Keywords:} Inviscid MHD system, Anisotropic
Littlewood-Paley Theory, Dissipative

 \qquad\qquad
estimates, Lagrangian coordinates\

\vskip 0.2cm

\noindent {\sl AMS Subject Classification (2000):} 35Q30, 76D03  \

\setcounter{equation}{0}
\section{Introduction}

In this paper, we investigate the  global wellposedness of the
following 2-D incompressible   magneto-hydrodynamical  system:
\begin{equation}\label{1.1}
 \left\{\begin{array}{l}
\displaystyle \pa_t \phi+\vv u\cdot\na\phi=0,\qquad (t,x)\in\R^+\times\R^2, \\
\displaystyle \pa_t \vv u +\vv u\cdot\na\vv u -\D\vv u+\na p=-\dv\bigl[\na\phi\otimes\na\phi\bigr], \\
\displaystyle \dv\,\vv u = 0, \\
\displaystyle \phi|_{t=0}=\phi_0,\quad  \vv u|_{t=0}=\vv u_0,
\end{array}\right.
\end{equation}
with initial data $(\phi_0,\vv u_0)$ smooth and close enough to the
equilibrium state $(x_2,\vv 0).$ Here $\phi$ denotes the magnetic
potential and $\vv u=(u^1, u^2)^T,$ $ p$ the velocity and scalar
pressure of the fluid respectively.

Recall that the general MHD system in $\R^d$ reads
\begin{equation}
 \left\{\begin{array}{l}
\displaystyle \pa_t\vv b+\vv u\cdot\na\vv b=\vv b\cdot\na\vv u,\qquad (t,x)\in\R^+\times\R^d, \\
\displaystyle \pa_t\vv u +\vv u\cdot\na\vv u -\D\vv u+\na p=-\f12\na|\vv b|^2+\vv b\cdot\na\vv b, \\
\displaystyle \dv\,\vv u =\dv\,\vv b=0, \\
\displaystyle \vv b|_{t=0}=\vv b_0,\quad \vv u|_{t=0}=\vv u_0,
\end{array}\right. \label{1.1a}
\end{equation}
where $\vv b=(b^1,\cdots,b^d)^T$ denotes the magnetic field, and
$\vv u=(u^1,\cdots,u^d)^T,$ $ p$ the velocity and scalar pressure of
the fluid respectively.  This MHD system \eqref{1.1a} with zero
diffusivity in the equation for the magnetic field can be applied to
model plasmas when the plasmas are strongly collisional, or the
resistivity due to these collisions are extremely small. It often
applies to the case when one is interested in the k-length scales
that are much longer than the ion skin depth and the Larmor radius
perpendicular to the field, long enough along the field to ignore
the Landau damping, and time scales much longer than the ion
gyration time \cite{CP, LL, Ca}. In the particular case when $d=2$
in \eqref{1.1a}, $\dv\,\vv b=0$ implies the existence of a scalar
function $\phi$ so that $\vv b=(\partial_2\phi, -\partial_1\phi)^T,$
and the corresponding system becomes \eqref{1.1}.

It is a long standing open problem that  whether or not classical
solutions of \eqref{1.1a} can develop finite time singularities even
in the 2-D case. Except with full magnetic diffusion in
\eqref{1.1a}, the corresponding 2-D system possesses  a unique
global smooth solution (see \cite{DL,ST} and \cite{AP} for initial
data in the critical spaces). With mixed partial dissipation and
additional (artificial) magnetic diffusion in the 2-D MHD system,
Cao and Wu \cite{CW} (see also \cite{CRW}) proved its global
wellposedness for any data in $H^2(\R^2).$ In \cite{LZ}, we proved
the global wellposedness of a three dimensional version of
\eqref{1.1} with smooth initial data which is close to a non-trivial
steady state. The aim of this paper is to establish the global existence
and uniqueness of solutions to the MHD equation \eqref{1.1} in the
2-D case with the same class of the initial data.

We note that the system \eqref{1.1} has appeared in many problems,
see the recent survey article \cite{Lin}. For the inviscid,
incompressible MHD equations \eqref{1.1a}, it is an important
problem that if it possesses a dissipation mechanism even though the
magnetic diffusivity is close to zero. The heating of high
temperature plasmas by MHD waves is one of the most interesting and
challenging problems of plasma physics especially when the energy is
injected into the system at the length scales much larger than the
dissipative ones. Indeed it has been conjectured that in the MHD
systems, energy is dissipated at a rate that is independent of the
ohmic resistivity \cite{ChCa}. In other words, the diffusivity for
the magnetic field equation can be zero yet the whole system may
still be dissipative. We shall justify this conjecture for initial
data sufficiently close to a non-trivial equilibrium state in the
two-dimensional case. Here the dissipation property is closely
related to a partial dissipative property (in spatial directions) of
magnetic waves due to non-trivial background magnetic fields. Our
global existence results are solely based on the latter property.
For this reason we also conjecture that such global well posedness
results would not be possible without non-trivial background
magnetic fields.

Notice that, after substituting $\phi=x_2+\psi$ into \eqref{1.1},
one obtains the following system for $(\psi,\vv u):$
\begin{equation}\label{1.2}
 \quad\left\{\begin{array}{l}
\displaystyle \pa_t \psi +\vv u \cdot \grad \psi+u^2=0,\qquad (t,x)\in \R^+\times\R^2,\\
\displaystyle \pa_t u^1 + \vv u \cdot \grad u^1-\D u^1+\p_1\p_2\psi=-\p_1p-\dv\bigl[\p_1\psi\na\psi\bigr]\eqdefa f^1, \\
\displaystyle \pa_t u^2 +\vv u \cdot \grad u^2-\D
u^2+(\D+\p_2^2)\psi=-\p_2p
-\dv\bigl[\p_2\psi\na\psi\bigr]\eqdefa f^2, \\
\displaystyle \dv\,\vv u = 0, \\
\displaystyle (\psi,\vv u)|_{t=0}=(\psi_0,\vv u_{0}).
\end{array}\right.
\end{equation}

Starting from \eqref{1.2}, standard energy estimate gives rise to
\beq\label{1.3} \f12\f{d}{dt}\bigl(\|\na\psi(t)\|_{L^2}^2+\|\vv
u(t)\|_{L^2}^2\bigr)+\|\na \vv u(t)\|_{L^2}^2=0 \eeq for smooth
enough solutions $(\psi,\vv u)$ of \eqref{1.2}. The main difficulty
to prove the global existence of small smooth solutions to
\eqref{1.2} is thus to find a  dissipative mechanism for $\psi.$
Motivated by the heuristic analysis in Subsection 2.1, we shall
first write the system \eqref{1.1} in the Lagrangian formulation
\eqref{a14}. There is, however, a subtle technical difficult for MHD
equations in this formulation unlike many other fluid dynamic problems.
We need to introduce a notion of admissible initial data (see Definition
1.1 below).  Next, we employ anisotropic Littlewood-Paley theory to
capture the delicate dissipative property of $Y_t$  in Section 3. It
turns out that  $\p_{y_1}Y$ decays faster
  than $\pa_{y_2}Y.$  This, in some
sense, also justifies the necessity of using anisotropic
Littlewood-Paley theory in Section 3. With the key {\it a priori}
$L^1(\R^+;Lip(\R^2))$ estimate for $Y_t,$ we shall prove the global
wellposedness of \eqref{a14} in Section 4.

To describe the initial data $\phi_0$ in \eqref{1.1}, we need the
following definition:

\begin{defi}\label{def1.1ad}
Let $\vv b=(b^1,b^2)^T$ be a smooth enough vector field. We define its
trajectory $X(t,x)$ by \beq\label{d3} \left\{\begin{array}{l}
\displaystyle \f{d X(t,x)}{dt}=\vv b(X(t,x)), \\
\displaystyle X(t,x)|_{t=0}=x.
\end{array}\right. \eeq
We call that $f$ and $\vv b$ are admissible on a domain $D$ of
$\R^2$ if there holds \beno \int_{\R}f(X(t,x))\,dt=0\quad\mbox{for
\, all }\quad x\in D. \eeno
\end{defi}

\begin{rmk}
The condition that $f$ and $\vv b$ are admissible on $\R^2$ (or some
subsets of $\R^2$) is to guarantee that \beq\label{d1} \vv
b\cdot\na\psi=f\quad\mbox{on}\quad \R^2, \eeq has a solution $\psi$
so that $\lim_{|x|\to \infty}\psi(x)=0.$ Let us take $\vv b=(1,0)^T$
for example. In this case, \eqref{d1} becomes $\p_{x_1}\psi=f,$
which together with the condition $\lim_{|x_1|\to\infty}\psi(x)=0$
ensures that \beno \psi(x_1,x_2)=-\int_{x_1}^\infty
f(s,x_2)\,ds=\int_{-\infty}^{x_1}f(s,x_2)\,ds. \eeno We thus obtain
that $\int_{\R}f(s,x_2)\,ds=0,$ that is, $f$ and $(1,0)^T$ are
admissible on $\{0\}\times\R.$
\end{rmk}

In what follows, for $X_1, X_2$ being two Banach spaces,  we always
denote the norms $\|\cdot\|_{X_1\cap
X_2}\eqdefa\|\cdot\|_{X_1}+\|\cdot\|_{X_2}$ and
$\|\cdot\|_{L^p(\R^+;X_1\cap
X_2)}\eqdefa\|\cdot\|_{L^p(\R^+;X_1)}+\|\cdot\|_{L^p(\R^+;X_2)}$ for
$p\in[1,\infty]$.

We now present our main result in this paper:

\begin{thm}\label{th2}
{\sl Let $s_1>1$, $s_2\in (-1,-\f12)$ and $s\geq s_1+2$. Given
$(\psi_0,\vv u_0)$ satisfying $\na\psi_0\in H^{s}\cap\dH^{s_2}(\R^2)$, $\vv
u_0\in H^{s}\cap\dH^{s_2}(\R^2)$ and \beq \label{1.5a}
 \|\na\psi_0\|_{\dH^{s_1+2}}\leq 1,\quad
\|\na\psi_0\|_{\dH^{s_1+1}\cap \dH^{s_2}}+
\|\p_{x_2}\psi_0\|_{H^{s_1+2}}+\|\vv u_0\|_{\dH^{s_1+1}\cap
\dH^{s_2}}\leq c_0 \eeq for some $c_0$ sufficiently small. We assume
moreover that $\p_{x_2}\psi_0$ and
$\bigl(1+\p_{x_2}\psi_0,-\p_{x_1}\psi_0\bigr)^T$ are admissible on
$\{0\}\times\R$ and $\Supp\p_{x_2}\psi_0(\cdot,x_2)\subset [-K,K]$ for
some positive number $K.$ Then \eqref{1.2} has a unique global
solution $(\psi,
 \vv u, p)$ (up to
a constant for $ p$) so that \beq\label{th1wqa}
\begin{split}
&\na\psi\in  C([0,\infty); H^s\cap\dH^{s_2}(\R^2))\cap L^2(\R^+;\dH^{s_1+1}\cap\dH^{s_2+1}(\R^2)),\\
& \vv u\in  C([0,\infty);
H^{s}\cap\dH^{s_2}(\R^2))\cap L^2(\R^+; \dH^{s_1+2}\cap\dH^{s_2+1}(\R^2))\cap L^1(\R^+;Lip(\R^2)),\\
&\na\vv u\in L_{\mbox{loc}}^2(\R^+;H^{s}(\R^2)),\quad\na p\in
C([0,\infty); H^{s-1}(\R^2))\cap L^2(\R^+;
\dH^{s_1}\cap\dH^{s_2}(\R^2))
\end{split}
\eeq Furthermore, there holds
 \beq\label{th1wra}
\begin{split}
&\|\vv u\|_{L^\infty(\R^+; \dH^{s_1+1}\cap\dH^{s_2})}
+\|\na\psi\|_{L^\infty(\R^+;\dH^{s_1+1}\cap\dH^{s_2})}+\| \vv
u\|_{L^2(\R^+;\dH^{s_1+2}\cap\dH^{s_2+1})}\\
&\quad+\|\na\psi\|_{L^2(\R^+;\dH^{s_1+1}\cap\dH^{s_2+1})}
+\|\na\vv u\|_{L^1(\R^+;L^\infty)}+\|\na p\|_{L^2(\R^+;\dH^{s_1}\cap\dH^{s_2})}\\
&\leq C\bigl( \|\na\psi_0\|_{\dH^{s_1+1}\cap \dH^{s_2}}+
\|\p_{x_2}\psi_0\|_{H^{s_1+2}}+\|\vv
u_0\|_{\dH^{s_1+1}\cap \dH^{s_2}}\bigr).
\end{split}
\eeq }
\end{thm}

\begin{rmk}
We can replace the condition that:
$\Supp\p_{x_2}\psi_0(\cdot,x_2)\subset [-K,K]$ for some positive
number $K$, in Theorem \ref{th2} by assuming appropriate decay of
$\p_{x_2}\psi_0(x)$ with respect to $x_1$ variable. To make the
presentation more transparent, we would not emphasize this technical
point here.
\end{rmk}

\begin{rmk}
 As $\psi$ is a scalar function in \eqref{1.2}, we can not apply the
 ideas and analysis developed  in \cite{LLZ, LZ7, Ch-Zh, LLZhen} for
the viscoelastic fluid system to solve \eqref{1.2}. Though it may sound
to be a technical reason, there is, in fact, a fundamental difference
between the viscoelastic fluid system and \eqref{1.2} which one can see
from \eqref{a14} and \eqref{a12}. We would like also to point out that
it is more tricky to find a mechanism for dissipation in \eqref{1.2} than
the case of the classical isentropic compressible Navier-Stokes system
(CNS), see for example \cite{Da}.
\end{rmk}

 \begin{rmk} For the three-dimensional version of the  system
\eqref{1.1}, we \cite{LZ} introduced the function spaces $\cB^{s_1,s_2}$
(see Definition \ref{def2} below).
 Due the difficulty in understanding the propagating of anisotropic
regularity for the transport equations, where we first
established  the {\it a priori} $L^1(\R^+;\cB^{\f52-\d,\d})$
($\d\in (\f12,1)$) estimate for $u^3,$ the third component of the
velocity field,  before proving the key $L^1(\R^+; Lip(\R^3))$ estimate
of $u^3.$ In doing so, it was essential that
$f^v\eqdefa-\sum_{i,j=1}^3\p_3(-\D)^{-1}
\bigl[\p_iu^j\p_ju^i+\p_i\p_j(\p_i\psi\p_j\psi)\bigr]-\sum_{j=1}^3\p_j(\p_3\psi\p_j\psi)$
 belongs to ${L}^1(\R^+,\cB^{\f12-\d,\d}).$ This latter requirement is
essentially equivalent to that $f^v\in
L^1(\R^+;\cB^{\f12-\d}_{2,1}(\R_v)(\cB^{\d}_{2,1}(\R^2_h)))$ for
$\d\in (\f12,1).$ In the 2-D case, this would require
$f^2$ given by \eqref{1.2} belonging to ${L}^1(\R^+;\cB^{-\d}_{2,1}(\R_v)(\cB^{\d}_{2,1}(\R_h)))$ for $\d\in
(\f12,1).$ The latter, however, is impossible due to the product
laws in Besov spaces for the vertical variable. This gives another good
reason why we will use the Lagrangian formulation of
\eqref{1.1} in this paper.
\end{rmk}

Let us complete this section by the notations we shall use in this paper.\\

\no{\bf Notations.} For any $s\in\R$, we denote by $H^s(\R^2)$ the
classical  $L^2$ based Sobolev spaces with the norm
$\|a\|_{H^s}\eqdefa\bigl(\int_{\R^2}(1+|\xi|^2)^s|\hat{a}(\xi)|^2\,d\xi\bigr)^{\f12},$
while $\dot{H}^s(\R^2)$ the classical homogenous Sobolev spaces with
the norm
$\|a\|_{\dot{H}^s}\eqdefa\bigl(\int_{\R^2}|\xi|^{2s}|\hat{a}(\xi)|^2\,d\xi\bigr)^{\f12}$.
Let $A, B$ be two operators, we denote $[A;B]=AB-BA,$ the commutator
between $A$ and $B$. For $a\lesssim b$, we mean that there is a
uniform constant $C,$ which may be different on different lines,
such that  $a\leq Cb,$ and $a\sim b$ means that both $a\lesssim b$
and $b\lesssim a$. We shall denote by $(a\ |\ b)$  the $L^2(\R^2)$
inner product of $a$ and $b.$ $(d_{j,k})_{j,k\in\Z}$ (resp.
$(c_j)_{j\in\Z}$) will be a generic element of $\ell^1(\Z^2)$ (resp.
$\ell^2(\Z))$ so that $\sum_{j,k\in\Z}d_{j,k}=1$ (resp.
$\sum_{j\in\Z}c_j^2=1).$ Finally, we denote by $L^p_T(L^q_h(L^r_v))$
the space $L^p([0,T]; L^q(\R_{x_1};L^r(\R_{x_2})))$.

\medskip

\setcounter{equation}{0}
\section{Lagrangain formulation of the  system \eqref{1.1}
} \label{sect2}

As in \cite{C}, \cite{CS}, \cite{CS2}, \cite{GMS}, \cite{OC},
\cite{SZ},
 \cite{XZZ} and various earlier references therein,  we
shall use the Lagrangian formulation. At first, we solve the couple
system between \eqref{1.1} and the following additional  transport
equation: \beq\label{a3} \p_t\tilde\phi+ \vv
u\cdot\na_x\tilde\phi=0,\quad\tilde\phi|_{t=0}=\tilde\phi_0, \eeq
where $\tilde\phi_0=-x_1+\tilde\psi_0,$ and $\tilde\psi_0$ is
determined by \beq\label{a2} \det\,U_0=1\quad \mbox{for}\quad
U_0=\begin{pmatrix} 1+\p_{x_2}\psi_0& \p_{x_2}\tilde\psi_0 \\
-\p_{x_1}\psi_0&1-\p_{x_1}\tilde\psi_0
\end{pmatrix}. \eeq
The existence of $\tilde\psi_0$ will be a consequence of Lemma
\ref{lemf1}.

Setting $\wt{\phi}=-x_1+\tilde\psi$ in \eqref{a3} yields
\begin{equation}\label{1.4ad}
 \pa_t \wt{\psi} +\vv u \cdot \grad \wt{\psi}-u^1=0,\quad \mbox{and}\quad
 \wt{\psi}|_{t=0}=\tilde{\psi}_0.
\end{equation}

The main result concerning the global small solutions to the coupled
system \eqref{1.1} and \eqref{a3} can be stated as follows:

\begin{thm}\label{th1}
{\sl Let $s_1>1,$ $s_2\in (-1,-\f12).$ Given
$(\phi_0,\tilde\phi_0,\vv u_0)\eqdefa(x_2+\psi_0,
-x_1+\tilde{\psi}_0, \vv u_0)$ satisfying $(\na \psi_0,\vv u_0)\in
\dH^{s_1+1}(\R^2)\cap \dH^{s_2}(\R^2)$,
$\na\tilde{\psi}_0\in\dH^{s_1+1}(\R^2)\cap \dH^{s_2+1}(\R^2)$ and
\eqref{a2}, the coupled system \eqref{1.1} and \eqref{a3} has a
unique global solution $(\phi, \tilde\phi, \vv u, p)=(x_2+\psi,
-x_1+\tilde\psi, \vv u, p)$ (up to a constant for $\psi, \tilde\psi,
p$) so that \beq\label{th1wq}
\begin{split}
&\na\psi\in  C([0,\infty); \dH^{s_1+1}(\R^2)\cap
\dH^{s_2}(\R^2))\cap
L^2(\R^+;\dH^{s_1+1}(\R^2)\cap \dH^{s_2+1}(\R^2)),\\
& \na\wt{\psi}\in C([0,\infty); \dH^{s_1+1}(\R^2)\cap
\dH^{s_2+1}(\R^2)),\\
&\vv u\in  C([0,\infty); \dH^{s_1+1}(\R^2)\cap \dH^{s_2}(\R^2))\cap
L^2(\R^+;\dH^{s_1+2}(\R^2)\cap \dH^{s_2+1}(\R^2))\\
&\qquad\cap
L^1(\R^+;Lip(\R^2)),\\
&\na p\in  L^2(\R^+;\dH^{s_1}(\R^2)\cap \dH^{s_2}(\R^2)),
\end{split}
\eeq provided that \beq \label{1.5} \|\na\psi_0\|_{\dH^{s_1+1}\cap
\dH^{s_2}}+ \|\na\tilde\psi_0\|_{\dH^{s_1+1}\cap \dH^{s_2+1}}+\|\vv
u_0\|_{\dH^{s_1+1}\cap \dH^{s_2}}\leq c_0 \eeq for some $c_0$
sufficiently small. Furthermore, there holds
 \beq\label{th1wr}
\begin{split}
&\|\vv u\|_{L^\infty(\R^+; \dH^{s_1+1}\cap\dH^{s_2})}
+\|\na\psi\|_{L^\infty(\R^+;\dH^{s_1+1}\cap\dH^{s_2})}+\|\na\tilde\psi\|_{L^\infty(\R^+;\dH^{s_1+1}\cap\dH^{s_2+1})}\\
&\quad+\| \vv
u\|_{L^2(\R^+;\dH^{s_1+2}\cap\dH^{s_2+1})}+\|\na\psi\|_{L^2(\R^+;\dH^{s_1+1}\cap\dH^{s_2+1})}
\\
&\quad+\|\na\vv u\|_{L^1(\R^+;L^\infty)}+\|\na p\|_{L^2(\R^+;\dH^{s_1}\cap\dH^{s_2})}\\
&\leq C\bigl( \|\na\psi_0\|_{\dH^{s_1+1}\cap \dH^{s_2}}+
\|\na\tilde\psi_0\|_{\dH^{s_1+1}\cap \dH^{s_2+1}}+\|\vv
u_0\|_{\dH^{s_1+1}\cap \dH^{s_2}}\bigr).
\end{split}
\eeq }
\end{thm}

In order to avoid the difficulty of propagating anisotropic
regularity for the free transport equation in the coupled system
\eqref{1.1} and \eqref{a3}, we shall first write them
in the Lagrangian coordinates. Indeed let $U\eqdefa\begin{pmatrix} \p_{x_2}\phi& \p_{x_2}\tilde\phi \\
-\p_{x_1}\phi&-\p_{x_1}\tilde\phi
\end{pmatrix},$ we deduce from \eqref{1.1} and \eqref{a3} that $U$
solves\beq\label{a5} \p_tU+ \vv u\cdot\na_x U=\na_x\vv u\ U,\quad
U|_{t=0}=U_0,\eeq with $U_0$ being given by \eqref{a2}.

To write the nonlinear term,
$\dive\bigl[\na\phi\otimes\na\phi\bigr],$ in the momentum equation
of \eqref{1.1},  into a clear form in the Lagrangian formulation, we
need the following lemma concerning the structure of $U_0.$

\begin{lem}\label{initial}
{\sl Let
$\phi_0(x)=x_2+\psi_0(x),\,\tilde\phi_0(x)=-x_1+\tilde\psi_0,$ with
$\psi_0\in\dH^{\tau_1+1}(\R^2)\cap \dH^{\tau_2+1}(\R^2)$,
$\tilde\psi_0\in\dH^{\tau_1+1}(\R^2)\cap \dH^{\tau_2+2}(\R^2)$ for
$\tau_1\in (2,\infty),$ $\tau_2\in (-1,0),$ and
$\|\na\psi_0\|_{L^\infty}+\|\na\tilde\psi_0\|_{L^\infty}\leq\e_0$
for some $\e_0$ sufficiently small. Let $U_0$ satisfy \eqref{a2}.
 Then there exists  $Y_0=(Y_0^1,Y_0^2)^T\in \dH^{\tau_1+1}(\R^2)\cap
\dH^{\tau_2+2}(\R^2)$ with $\p_1Y_0\in\dH^{\tau_2}(\R^2)$ so that
$X_0(y)\eqdefa y+Y_0(y)$ satisfies \beq\label{app0} U_0\circ
X_0(y)=\na_y X_0(y)=I+\na_y Y_0(y). \eeq Moreover, there holds
\beq\label{app1}\begin{aligned}
&\|Y_0\|_{\dH^{\tau_1+1}\cap\dH^{\tau_2+2}}+\|\p_1Y_0\|_{\dH^{\tau_2}}\\
&\leq
C(\|\psi_0\|_{\dH^{\tau_1+1}\cap\dH^{\tau_2+1}},\|\tilde\psi_0\|_{\dH^{\tau_1+1}\cap\dH^{\tau_2+2}}
)\bigl(\|\psi_0\|_{\dH^{\tau_1+1}\cap\dH^{\tau_2+1}}+\|\tilde\psi_0\|_{\dH^{\tau_1+1}\cap\dH^{\tau_2+2}}\bigr),
\end{aligned}\eeq
where $C(\lam_1,\lam_2)$ is a constant depending on $\lam_1$ and
$\lam_2$ non-decreasingly.}
\end{lem}

\begin{proof} Let $Y=(Y^1, Y^2)^T,$ we denote
\beq\label{app3}\left\{\begin{aligned}
&F(y,Y)\eqdefa\psi_0(y+Y)+Y^2,\\
&G(y,Y)\eqdefa\tilde\psi_0(y+Y)-Y^1.
\end{aligned}\right.\eeq
Notice from the assumption  $\mathrm{det}\, U_0=1$ that \beno
\mathrm{det}\f{\p(F,G)}{\p(Y^1,Y^2)}=\mathrm{det}\begin{pmatrix}
\p_{x_1}\psi_0&1+\p_{x_2}\psi_0\\
\p_{x_1}\tilde\psi_0-1&\p_{x_2}\tilde\psi_0
\end{pmatrix}\Big|_{x=y+Y}=\mathrm{det}\,U_0(y+Y)=1,
\eeno from which and the classical
 implicit function
theorem, we deduce that around every point $y,$ the functions
$F(y,Y)=0$ and $G(y,Y)=0$ determines a unique function
$Y_0(y)=(Y_0^1(y),Y_0^2(y))^T$ so that \beno F(y,
Y_0(y))=0=G(y,Y_0(y)),\eeno or equivalently \beq \label{app3ad}
Y^1_0(y)=\tilde\psi_0(y+Y_0(y))\quad\mbox{and}\quad
Y^2_0(y)=-\psi_0(y+Y_0(y)).
 \eeq
Moreover, there holds \beq \label{app3cd} \left\{\begin{aligned}
&\p_{y_1}Y^1_0(y)=\p_{x_1}\tilde{\psi}_0\circ (y+Y_0(y))(1+\p_{y_1}Y^1_0(y))+\p_{x_2}\tilde{\psi}_0\circ (y+Y_0(y))\p_{y_1}Y^2_0(y),\\
&\p_{y_2}Y^1_0(y)=\p_{x_1}\tilde{\psi}_0\circ (y+Y_0(y))\p_{y_2}Y^1_0(y)+\p_{x_2}\tilde{\psi}_0\circ (y+Y_0(y))(1+\p_{y_2}Y^2_0(y)),\\
&\p_{y_1}Y^2_0(y)=-\p_{x_1}\psi_0\circ (y+Y_0(y))(1+\p_{y_1}Y^1_0(y))-\p_{x_2}\psi_0\circ (y+Y_0(y))\p_{y_1}Y^2_0(y),\\
&\p_{y_2}Y^2_0(y)=-\p_{x_1}\psi_0\circ
(y+Y_0(y))\p_{y_2}Y^1_0(y)-\p_{x_2}\psi_0\circ
(y+Y_0(y))(1+\p_{y_2}Y^2_0(y)).
\end{aligned}\right.\eeq
Thanks to \eqref{app3cd}  and $\mathrm{det}\,U_0=1,$ we infer \beq
\label{app3bd} \left\{\begin{aligned}
&\p_{y_1}Y^1_0(y)=\p_{x_2}\psi_0\circ (y+Y_0(y)),\\
&\p_{y_2}Y^1_0(y)=\p_{x_2}\tilde{\psi}_0\circ (y+Y_0(y)),\\
&\p_{y_1}Y^2_0(y)=-\p_{x_1}\psi_0\circ (y+Y_0(y)),\\
&\p_{y_2}Y^2_0(y)=-\p_{x_1}\tilde\psi_0\circ (y+Y_0(y)),
\end{aligned}\right.\eeq
which implies \eqref{app0}. Moreover, it follows from \eqref{app3bd}
that \beno\na_x(X_0^{-1}(x)-x)=(I+\na_yY_0)^{-1}\circ
X_0^{-1}(x)-I=\begin{pmatrix}
-\p_{x_1}\tilde\psi_0&-\p_{x_2}\tilde\psi_0\\
\p_{x_1}\psi_0&\p_{x_2}\psi_0
\end{pmatrix}.\eeno
Then using Lemma \ref{funct} with $\Phi=X_0^{-1}(x)$, we deduce from
\eqref{app3bd} that for $\tau_1\in(2,\infty)$ and $\tau_2\in(-1,0)$,
\beno\begin{aligned} \|\na_yY_0\|_{\dH^{\tau_1}}&\leq
C(\|\na_x\psi_0\|_{L^\infty},\|\na_x\tilde\psi_0\|_{L^\infty})(1+\|\D_x\psi_0\|_{H^{\tau_1-2}}
+\|\D_x\tilde\psi_0\|_{H^{\tau_1-2}})\\
&\quad\times\bigl(\|\D_x\psi_0\|_{H^{\tau_1-1}}+\|\D_x\tilde\psi_0\|_{H^{\tau_1-1}}\bigr),\\
\|\na_yY_0\|_{\dH^{\tau_2+1}}&\leq
C(\|\na_x\psi_0\|_{L^\infty},\|\na_x\tilde\psi_0\|_{L^\infty})
\bigl(\|\na_x\psi_0\|_{\dH^{\tau_2+1}}+\|\na_x\tilde\psi_0\|_{\dH^{\tau_2+1}}\bigr),\\
\|\p_{y_1}Y_0\|_{\dH^{\tau_2}}&\leq
C(\|\na_x\psi_0\|_{L^\infty},\|\na_x\tilde\psi_0\|_{L^\infty})
\bigl(\|\na_x\psi_0\|_{\dH^{\tau_2}}\\
&\qquad+(\|\na_x\psi_0\|_{\dH^{\tau_2+1}}+\|\na_x\tilde\psi_0\|_{\dH^{\tau_2+1}})
\|\na_x\psi_0\|_{L^2}\bigr),
\end{aligned}\eeno
which along with Sobolev imbedding theorem  ensures \eqref{app1}.
This concludes the proof of Lemma \ref{initial}.
\end{proof}

With Lemma \ref{initial}, for $(\psi_0,\tilde\psi_0)$ given by
Theorem \ref{th1} and $U_0$ by \eqref{a2}, there exists
$Y_0=(Y_0^1,Y_0^2)^T$ so that $\na Y_0\in \dot{H}^{s_1+1}(\R^2)\cap
\dot{H}^{s_2+1}(\R^2),$  $\p_1Y_0\in\dot{H}^{s_2}(\R^2),$ and there
hold \eqref{app0} and \eqref{app1}. With $X_0(y)=y+Y_0(y)$ thus
obtained, we define the flow map $X(t,y)$ by
\beno\left\{\begin{aligned}
&\f{dX(t,y)}{dt}=\vv u(t,X(t,y)),\\
&X(t,y)|_{t=0}=X_0(y),
\end{aligned}\right.\eeno
and $Y(t,y)$ through \beq\label{a6} X(t,y)=X_0(y)+\int_0^t\vv
u(s,X(s,y))\,ds\eqdefa y+Y(t,y). \eeq Then thanks to Lemma 1.4
of \cite{Majda} and \eqref{app0}, we deduce from \eqref{a5} that
\beq\label{a8} U(t, X(t,y))=\na_y X(t,y)=I+\na_y
Y(t,y)\quad\mbox{and}\quad\mathrm{det}\, (I+\na_y Y(t,y))=1. \eeq
Denoting $ U(t, X(t,y))\eqdefa(a_{ij})_{i,j=1,2}$ and $U^{-1}(t,
X(t,y))=(I+\na_y Y(t,y))^{-1}\eqdefa(b_{ij})_{i,j=1,2}$. It is easy
to check that, as $\mathrm{det}\, U=1,$ $(b_{ij})_{i,j=1,2}=\cA_Y$
with
 \beq\label{a14a}
\mathcal{A}_Y\eqdefa\begin{pmatrix}
1+\p_{y_2}Y^2&-\p_{y_2}Y^1\\
-\p_{y_1}Y^2&1+\p_{y_1}Y^1
\end{pmatrix},
\eeq is the adjoint matrix  of $(I+\na_y Y),$ and $\sum_{i=1}^2\f{\p
b_{ij}}{\p y_i}=0$. Furthermore, let $\vv
b\eqdefa(\p_{x_2}\phi,-\p_{x_1}\phi)^T,$  it follows from \eqref{a8}
that \beq\label{a9} \vv b\circ
X(t,y)=(1+\p_{y_1}Y^1,\p_{y_1}Y^2)^T\quad\text{and}\quad\cA_Y(\vv
b\circ X)=(1,0)^T, \eeq and consequently one has \beno
\begin{split}
\bigl[-\dv_x\,(\na_x\phi&\otimes\na_x\phi)+\na_x(|\na_x\phi|^2)\bigr]\circ
X(t,y)\\
&=[\dv_x(\vv b\otimes\vv b)]\circ X(t,y)\\
&=\na_y\cdot[\cA_Y(\vv b\circ X)\otimes(\vv b\circ X)]=\p_{y_1}(\vv
b\circ X)=\p_{y_1}^2Y(t,y),
\end{split}\eeno
that is
 \beq\label{a10}
\bigl[-\mathrm{div}_x(\na_x\phi\otimes\na_x\phi)+\na_x(|\na\phi|^2)\bigr](t,
X(t,y))=\p_{y_1}^2 Y(t,y). \eeq With \eqref{a6} and \eqref{a10}, we
can  reformulate \eqref{1.1} and \eqref{a3} as
\beq\label{a14}\left\{\begin{aligned}
&Y_{tt}-\na_Y\cdot\na_Y Y_t-\p_{y_1}^2Y+\na_Yq=\vv 0,\\
&\na_Y\cdot Y_t=0,\\
&Y|_{t=0}=Y_0,\quad Y_t|_{t=0}=\vv u_0\circ X_0(y)\eqdefa Y_1,
\end{aligned}\right.\eeq
where $q(t,y)=(p+|\na\phi|^2)\circ X(t,y)$ and
$\na_Y=\mathcal{A}_Y^T\na_y$  with $\cA_Y$ being given by
\eqref{a14a}.
 Here and in
what follows, we always assume that $\|\na_y
Y\|_{L^\infty}\leq\f{1}{2}$. Under this assumption, we  rewrite
\eqref{a14} as \beq\label{a12}\left\{\begin{aligned}
&Y_{tt}-\Delta_y Y_t-\p_{y_1}^2 Y=\vv f(Y,q),\\
&\na_y\cdot Y=\r(Y),\\
&Y|_{t=0}=Y_0,\quad Y_t|_{t=0}=Y_1,
\end{aligned}\right.\eeq
where \beq\label{a13}\begin{aligned}
&\vv f(Y,q)=(\na_Y\cdot\na_Y-\Delta_y)Y_t-\na_Yq,\\
&\r(Y)=\na_y\cdot Y_0-\int_0^t(\na_Y-\na_y)\cdot Y_s
ds=\p_{y_1}Y^2\p_{y_2}Y^1-\p_{y_1}Y^1\p_{y_2}Y^2.
\end{aligned}\eeq
Here we have used the fact that $\det\,(I+\na_y Y_0)=\det\, U_0=1$
and \eqref{a14a} to derive the second equality of \eqref{a13}.
Indeed, thanks to \eqref{a14a} and $\det\,(I+\na_y Y_0)=1$, one has
\beq\label{a13qw}
\begin{split}
(\na_Y-\na_y)\cdot
Y_t=&\f{d}{dt}\bigl(\p_{y_1}Y^1\p_{y_2}Y^2-\p_{y_1}Y^2\p_{y_2}Y^1\bigr)\\
=&\f{d}{dt}\bigl(\mathrm{det}\,(I+\na_y Y)-1-\na_y\cdot
Y\bigr).
\end{split} \eeq Furthermore, \eqref{a13qw} ensures that the
equation $\na_y\cdot Y=\r(Y)$ would imply  $\mathrm{det}\,(I+\na_y
Y)=1$ and $\na_Y\cdot Y_t=0$.

 For notational
convenience, we shall neglect the subscripts $x$ or $y$ in $\p, \na$
and $\D$ in what follows. We make the convention that whenever $\na$
acts on $(\psi, \vv u, p),$ which is a solution to \eqref{1.2}, we
understand $(\na\psi, \na\vv u, \na p)$ as $(\na_x\psi, \na_x\vv u, \na_x
p).$ While $\na$ acts on $(Y,q),$ the solution to \eqref{a12}, we
understand $(\na Y, \na q)$ as $(\na_y Y, \na_y q).$ Similar
conventions for $\p$ and $\D.$

For  \eqref{a12}-\eqref{a13}, we have  the following global
wellposedness result:

\begin{thm}\label{T} {\sl Let $s_1>1$, $s_2\in(-1,-\f{1}{2})$. Let $(Y_0, Y_1)$ satisfy $(\p_1Y_0,\,\D Y_0)
\in\dot{H}^{s_1}(\R^2)\cap\dot{H}^{s_2}(\R^2)$, $Y_1\in\dot{H}^{s_1+1}(\R^2)\cap\dot{H}^{s_2}(\R^2)$ and
\beq\label{A1a}
 \mathrm{det}\,(I+\na Y_0)=1,\quad\na_{Y_0}\cdot Y_1=0, \quad\text{and}
\eeq \beq\label{A1}
\|\p_1Y_0\|_{\dot{H}^{s_2}}+\|\D
Y_0\|_{\dot{H}^{s_1}\cap\dot{H}^{s_2}}+\|Y_1\|_{\dot{H}^{s_1+1}\cap\dot{H}^{s_2}}\leq
\e_0 \eeq for some $\e_0$ sufficiently small.
 Then  \eqref{a12}-\eqref{a13} has a unique global solution $(Y,q)$ (up to a constant for $q$)
 so that \beq\label{A1wq}\begin{split}
 & Y\in C([0,\infty);\dot{H}^{s_1+2}\cap\dot{H}^{s_2+2}(\R^2))\quad\mbox{and}\quad Y^2\in C([0,\infty);\dot{H}^{s_2+1}(\R^2)),\\
 &\p_1Y\in C([0,\infty);\dot{H}^{s_2}(\R^2))\cap L^2(\R^+; \dot{H}^{s_1+1}
 \cap \dot{H}^{s_2+1}(\R^2)),\\
&Y_t\in C([0,\infty);\dot{H}^{s_1+1}\cap\dot{H}^{s_2}(\R^2))\cap
L^2(\R^+;\dot{H}^{s_1+2}\cap\dot{H}^{s_2+1}(\R^2))\cap L^1(\R^+;
Lip(\R^2)),\\
& \na q\in L^2(\R^+;\dot{H}^{s_1}\cap\dot{H}^{s_2}(\R^2))\cap
L^1(\R^+;\dot{H}^{s_1}\cap\dot{H}^{s_2}(\R^2)). \end{split} \eeq
Moreover, there hold $\mathrm{det}\,(I+\na Y)=1,\, \na_Y\cdot
Y_t=0$, and \beq\label{M1}\begin{aligned}
&\|Y\|_{L^\infty(\R^+;\dot{H}^{s_1+2}\cap\dot{H}^{s_2+2})}^2
+\|\p_1Y\|_{L^\infty(\R^+;\dot{H}^{s_2})}^2+\|Y^2\|_{L^\infty(\R^+;\dot{H}^{s_2+1})}^2\\
 &\quad+\|Y_t\|_{L^\infty(\R^+;\dot{H}^{s_1+1}\cap\dot{H}^{s_2})}^2+\|\p_1Y\|_{L^2(\R^+;\dot{H}^{s_1+1}\cap\dot{H}^{s_2+1})}^2
 +\|Y_t\|_{L^2(\R^+;\dot{H}^{s_1+2}\cap\dot{H}^{s_2+1})}^2
\\
&\quad +\|\na Y_t\|_{L^1(\R^+;L^\infty)}^2+\|\na q\|_{L^2(\R^+;\dot{H}^{s_1}\cap\dot{H}^{s_2})}^2+\|\na q\|_{L^1(\R^+;\dot{H}^{s_1}\cap\dot{H}^{s_2})}^2\\
&\leq C\bigl(\|\p_1Y_0\|_{\dot{H}^{s_2}}^2+\|\D
Y_0\|_{\dot{H}^{s_1}\cap\dot{H}^{s_2}}^2+\|Y_1\|_{\dot{H}^{s_1+1}\cap\dot{H}^{s_2}}^2\bigr).
\end{aligned}\eeq
}
\end{thm}

\begin{rmk} As for the Eulerian formulation \eqref{1.2},
 it is also technical to explore the delicate
mechanism of  partial dissipations in \eqref{a12}. We overcome this
difficulty by applying the anisotropic Littlewood-Paley theory.
Since we use the Lagrangian formulation \eqref{a12} instead of
Eulerian one \eqref{1.2}, we can avoid the difficulty
concerning the propagation anisotropic regularity for the transport
equation, which we encountered in \cite{LZ}.  In general,
 it is interesting to study how the anisotropic
 Littlewood-Paley theory can be applied to these evolution equations
 with degenerations of certain ellipticity (parabolicity) in phase
 variables.
\end{rmk}

\begin{rmk}
We note once again that the equation $\na\cdot
Y=\r(Y)=\p_1Y^2\p_2Y^1-\p_1Y^1\p_2Y^2$, which implies the
incompressibility condition $\det\,(I+\na Y)=1$, plays a key role in
the proof of Theorem \ref{T}. In particular, this equation ensures
the global in time $L^1$-estimates of $\na q$ and $\na Y_t$, which
are crucial in order
 to close the energy estimates for \eqref{a12}-\eqref{a13}. Furthremore,
 this equation provides  better estimates for the component $Y^2$ than that of
$Y^1.$
\end{rmk}

\no{\bf Scheme of the proofs.}\\

To avoid the difficulty caused by propagating  anisotropic
regularity for the free transport equation, we shall first prove
Theorem \ref{T}, which concerns the global wellposedness in the
Lagrangian formulation \eqref{a12}-\eqref{a13} for the coupled
system \eqref{1.1} and \eqref{a3}.

Indeed let $(Y,q)$ be a  smooth enough solution of \eqref{a12},
applying standard energy estimate to \eqref{a12} gives rise to \beq
\label{1.12ab}
\begin{split}
&\f{d}{dt}\Bigl\{\f{1}{2}\bigl(\|Y_t\|_{\dot{H}^s}^2+\|Y_t\|_{\dot{H}^{s+1}}^2+\|\p_1Y\|_{\dot{H}^s}^2
+\|\p_1Y\|_{\dot{H}^{s+1}}^2+\f{1}{4}\|Y\|_{\dot{H}^{s+2}}^2\bigr)
-\f{1}{4}(Y_t\ |\ \Delta Y)_{\dot{H}^s}\Bigr\}\\
&\quad+\f{3}{4}\|Y_t\|_{\dot{H}^{s+1}}^2+\|Y_t\|_{\dot{H}^{s+2}}^2+\f{1}{4}\|\p_1Y\|_{\dot{H}^{s+1}}^2\\
&=\bigl(\vv f\ |\ Y_t-\f{1}{4}\Delta Y-\D Y_t\bigr)_{\dot{H}^s}.
\end{split}
\eeq where $\vv f$ is given by \eqref{a13} and $(a\ |\
b)_{\dot{H}^s}$ denotes the standard $\dot{H}^s$ inner product of
$a$ and $b.$ \eqref{1.12ab} shows that $\p_1 Y$ belongs to
$L^2(\R^+;\dot{H}^{s+1}(\R^2)).$ After a careful check, to close the
energy estimate \eqref{1.12ab},  we  need also the
$L^1(\R^+;Lip(\R^2))$ estimate of $Y_t.$ Toward this, we investigate
first the spectrum properties to the following linearized system of
\eqref{a12}-\eqref{a13}:
 \beq\label{a12ag}\left\{\begin{aligned}
&Y_{tt}-\Delta Y_t-\p_1^2 Y=\vv f,\\
&Y|_{t=0}=Y_0,\quad Y_t|_{t=0}=Y_1.
\end{aligned}\right.\eeq
Simple calculation shows that the symbolic equation of \eqref{a12ag}
has eigenvalues $\la_\pm(\xi)$ given by \eqref{5.2} and they satisfy
\eqref{5.2ad}. This shows that smooth solution of \eqref{a12ag}
decays in a very subtle way. In order to capture this delicate decay
property for the solutions of \eqref{a12ag}, we will have to
decompose our frequency analysis into two parts: $\bigl\{
\xi=(\xi_1,\xi_2):\ |\xi|^2\leq 2|\xi_1|\ \bigr\}$ and $\bigl\{
\xi=(\xi_1,\xi_2):\ |\xi|^2> 2|\xi_1|\ \bigr\}.$ It suggests us to
use anisotropic Littlewood-Paley analysis  to obtain the
$L^1(\R^+,Lip(\R^2))$ estimate of $Y_t.$

With Theorem \ref{T}, we can prove Theorem \ref{th1} through
coordinate transformation, namely from Lagrangian coordinates to
Eulerian ones. Finally thanks to Lemma \ref{lemf1}, for $s>2$ and
 $\na\psi_0\in H^s(\R^2),$  there exists
 $\wt{\psi}_0$ so that there hold \eqref{a2} and \eqref{f2qw}. We
 thus obtained $(\psi_0,\wt{\psi}_0)$ and the initial velocity field
 $\vv u_0$
 given by Theorem \ref{th2}, we infer from Theorem \ref{th1} that
 the coupled system between \eqref{1.1} and \eqref{a3} has a unique
 solution $(\phi, \wt{\phi},\vv u, p),$ which satisfies
 \eqref{th1wq} and \eqref{th1wr}. In particular, $(\phi, \vv u, p)$
 solves \eqref{1.1} and we complete the proof of Theorem \ref{th2}.

The rest of the paper is organized as follows. In the first part of
Section \ref{sect3}, we shall present a heuristic analysis to the
linearized system of \eqref{a12}-\eqref{a13}, which motivates us to
use anisotropic Littlewood-Paley theory below,  then we shall
collect some basic facts on Littlewood-Paley analysis in Subsection
3.2. In Section 4, we apply anisotropic Littlewood-Paley theory to
explore the dissipative mechanism for  a linearized model of
\eqref{a12}-\eqref{a13}. In Section 5, we present the proof of
Theorem \ref{T} and then  Theorems \ref{th1} and \ref{th2} in
Section 6. Finally, we present the proofs of some technical lemmas
in the Appendices.

\medskip
 \setcounter{equation}{0}
\section{Preliminary}\label{sect3}
\subsection{Spectral analysis to the linearized system of \eqref{a12} }\label{subsect3.1} Before dealing with the full system \eqref{a12}-\eqref{a13},
 we shall make some heuristic
analysis to the  linearized system \eqref{a12ag}. Observe that the
symbolic equation of \eqref{a12ag} reads \beno
\la^2+|\xi|^2\la+\xi_1^2=0\quad\mbox{for}\quad \xi=(\xi_1,\xi_2).
\eeno It is easy to calculate that this equation has two different
eigenvalues \beq  \la_\pm =-\f{|\xi|^2\pm
\sqrt{|\xi|^4-4\xi_1^2}}{2}. \label{5.2} \eeq The Fourier modes
corresponding to $\la_+$ decays like $e^{-t|\xi|^2}$.  Whereas the
decay property of the Fourier modes corresponding to $\la_-$  varies
with directions of $\xi$ as \beq\label{5.2ad}
\la_-(\xi)=-\f{2\xi_1^2}{|\xi|^2\bigl(1+\sqrt{1-\f{4\xi_1^2}{|\xi|^4}}\bigr)}
\to -1\quad \mbox{as}\quad |\xi|\to \infty \eeq only in the $\xi_1$
direction. This simple  analysis shows that the dissipative
properties of the solutions to \eqref{a12ag} may be more complicated
than that for the linearized system of isentropic compressible
Navier-Stokes system in \cite{Da}. It  also suggests us to employ
the tool of anisotropic Littlewood-Paley theory, which has been used
in the study of the global wellposedness to 3-D anisotropic
incompressible Navier-Stokes equations \cite{CDGG,CZ,GZ2, DI,
Pa02,PZ1,Zhangt2}, and in \cite{LZ} to explore the dissipative
properties to the three-dimensional case of
 \eqref{1.2}. One may check Section \ref{sect4} below for the detailed rigorous analysis
corresponding to this scenario.

\subsection{Littlewood-Paley theory} The
proof  of Theorem \ref{th1} requires a dyadic decomposition of the
Fourier variables, or the Littlewood-Paley decomposition. For the
convenience of the readers, we recall some basic facts on
Littlewood-Paley theory from \cite{bcd}.  Let $\varphi$ and $\chi$
be smooth functions supported in $\mathcal{C}\eqdefa \{
\tau\in\R^+,\ \frac{3}{4}\leq\tau\leq\frac{8}{3}\}$ and $\cB\eqdefa
\{ \tau\in\R^+,\ \tau\leq\frac{4}{3}\}$ such that
\begin{equation*}
 \sum_{j\in\Z}\varphi(2^{-j}\tau)=1 \quad\hbox{for}\quad \tau>0\quad\mbox{and}\quad  \chi(\tau)+ \sum_{j\geq
0}\varphi(2^{-j}\tau)=1.
\end{equation*}
For $a\in{\mathcal S}'(\R^2),$ we set \beq
\begin{split}
&\Delta_k^ha\eqdefa\cF^{-1}(\varphi(2^{-k}|\xi_1|)\widehat{a}),\qquad
S^h_ka\eqdefa\cF^{-1}(\chi(2^{-k}|\xi_1|)\widehat{a}),
\\
& \Delta_\ell^va
\eqdefa\cF^{-1}(\varphi(2^{-\ell}|\xi_2|)\widehat{a}),\qquad \
S^v_\ell a \eqdefa \cF^{-1}(\chi(2^{-\ell}|\xi_2|)\widehat{a}),
 \quad\mbox{and}\\
&\Delta_ja\eqdefa\cF^{-1}(\varphi(2^{-j}|\xi|)\widehat{a}),
 \qquad\ \ \
S_ja\eqdefa \cF^{-1}(\chi(2^{-j}|\xi|)\widehat{a}), \end{split}
\label{1.0}\eeq where $\cF a$ and $\widehat{a}$ denote the Fourier
transform of the distribution  $a.$ The dyadic operators defined in
\eqref{1.0} satisfy the property of almost orthogonality:
\begin{equation}\label{Pres_orth}
\Delta_k\Delta_j a\equiv 0 \quad\mbox{if}\quad| k-j|\geq 2
\quad\mbox{and}\quad \Delta_k( S_{j-1}a \Delta_j b) \equiv
0\quad\mbox{if}\quad| k-j|\geq 5.
\end{equation}
Similar properties hold for $\D_k^h$ and $\D_\ell^v.$

In what follows, we shall frequently use the following anisotropic
type Bernstein inequalities:

\begin{lem}\label{le2.1} {\sl Let $\cB_{h}$ (resp.~$\cB_{v}$) a ball
of~$\R_{h}$ (resp.~$\R_{v}$), and~$\cC_{h}$ (resp.~$\cC_{v}$) a ring
of~$\R_{h}$ (resp.~$\R_{v}$); let~$1\leq p_2\leq p_1\leq \infty$ and
~$1\leq q_2\leq q_1\leq \infty.$ Then there holds:
\smallbreak\noindent If the support of~$\wh a$ is included
in~$2^k\cB_{h}$, then
\[
\|\partial_{x_1}^\alpha a\|_{L^{p_1}_h(L^{q_1}_v)} \lesssim
2^{k\left(\al+\left(\frac1{p_2}-\frac1{p_1}\right)\right)}
\|a\|_{L^{p_2}_h(L^{q_1}_v)}.
\]
If the support of~$\wh a$ is included in~$2^\ell\cB_{v}$, then
\[
\|\partial_{x_2}^\beta a\|_{L^{p_1}_h(L^{q_1}_v)} \lesssim
2^{\ell(\beta+(\frac1{q_2}-\frac1{q_1}))} \|
a\|_{L^{p_1}_h(L^{q_2}_v)}.
\]
If the support of~$\wh a$ is included in~$2^k\cC_{h}$, then
\[
\|a\|_{L^{p_1}_h(L^{q_1}_v)} \lesssim 2^{-kN} \|\partial_{x_1}^N
a\|_{L^{p_1}_h(L^{q_1}_v)}.
\]
If the support of~$\wh a$ is included in~$2^\ell\cC_{v}$, then
\[
\|a\|_{L^{p_1}_h(L^{q_1}_v)} \lesssim 2^{-\ell N} \|\partial_{x_2}^N
a\|_{L^{p_1}_h(L^{q_1}_v)}.
\]}
\end{lem}

\begin{defi}\label{def1}
{\sl  Let  $(p,r)\in[1,+\infty]^2$, $s\in\R$ and $u\in{\mathcal S}'(\R^2),$ we set
$$
\|u\|_{\dot{B}^s_{p,r}}\eqdefa\bigl(2^{qs}\|\Delta_qu\|_{L^p}\bigr)_{\ell^r}.
$$
\begin{itemize}
\item For $s<\f{2}{p}$ (or $s=\f{2}{p}$ if $r=1$), we define $\dot{B}_{p,r}^s(\R^2)\eqdefa\{u\in\cS'(\R^2)\,|\, \|u\|_{\dot{B}^s_{p,r}}<\infty \}$.
\item If $k\in\N$ and $\f{2}{p}+k-1\leq s<\f{2}{p}+k$ (or $s=\f{2}{p}+k$ if $r=1$), then $\dot{B}_{p,r}^s(\R^2)$
is defined as the subset of distribution $u\in\cS'(\R^2)$ such that $\p^\b u\in\dot{B}_{p,r}^{s-k}(\R^2)$ whenever $|\b|=k$.
\end{itemize}
}
\end{defi}

\begin{rmk}\label{rmk1}
(1) It is easy to observe that $\dot{B}_{2,2}^s(\R^2)=\dot{H}^s(\R^2)$.

(2) Let  $(p,r)\in[1,+\infty]^2$, $s\in\R$ and $u\in{\mathcal S}'(\R^2)$. Then $u\in\dot{B}_{p,r}^s(\R^2)$  if and only if there exists $\{c_{j,r}\}_{j\in\Z}$ such that $\|c_{j,r}\|_{\ell^r}=1$ and
\beno
\|\D_ju\|_{L^p}\leq Cc_{j,r}2^{-js}\|u\|_{\dot{B}^s_{p,r}} \quad\text{for all}\quad j\in\Z.
\eeno

(3) Let $s, s_1, s_2\in\R$ with $s_1<s<s_2$ and
$u\in\dot{H}^{s_1}(\R^2)\cap\dot{H}^{s_2}(\R^2)$. Then $u\in
\dot{B}^s_{2,1}(\R^2)$, and  there holds \beq\label{bb0}
\|u\|_{\dot{B}^s_{2,1}}\lesssim
\|u\|_{\dot{H}^{s_1}}^{\f{s_2-s}{s_2-s_1}}\|u\|_{\dot{H}^{s_2}}^{\f{s-s_1}{s_2-s_1}}\lesssim\|u\|_{\dot{H}^{s_1}}+\|u\|_{\dot{H}^{s_2}}.
\eeq
\end{rmk}

To derive the $L^1(\R^+; Lip(\R^2))$ estimate of $Y_t$ determined by
\eqref{a12ag}, we need the following two-dimensional version of the
anisotropic Besov type space introduced in \cite{LZ}:

\begin{defi}\label{def2}
{\sl  Let  $s_1,s_2\in\R$ and $u\in{\mathcal S}'(\R^2),$ we define
the norm
$$
\|u\|_{\cB^{s_1,s_2}}\eqdefa\sum_{j,k\in\Z^2}2^{js_1}2^{ks_2}\|\Delta_j\D_k^h
u\|_{L^{2}}.
$$
}
\end{defi}

Then motivated by the proof of Lemma 2.1 in \cite{LZ}, we have the
following improved version:

\begin{lem}\label{L1}
Let $s_1,s_2,\tau_1,\tau_2\in\R,$  which satisfy
$s_1<\tau_1+\tau_2<s_2$ and $\tau_2>0.$ Then for
$a\in\dot{H}^{s_1}(\R^2)\cap\dot{H}^{s_2}(\R^2)$,
$a\in\cB^{\tau_1,\tau_2}(\R^2)$ and there holds \beno
\|a\|_{\cB^{\tau_1,\tau_2}}\lesssim\|a\|_{\dot{B}^{\tau_1+\tau_2}_{2,1}}\lesssim\|a\|_{\dot{H}^{s_1}}+\|a\|_{\dot{H}^{s_2}}.
\eeno
\end{lem}

\begin{proof}
Indeed thanks to Definition \ref{def2} and the fact: $j\geq k-N_0$
for some fixed positive integer $N_0$ in dyadic operator
$\D_j\D^h_k$, we have \beno
\begin{split}
\|a\|_{\cB^{\tau_1,\tau_2}}&=\sum_{k\leq
j+N_0}2^{j\tau_1}2^{k\tau_2}\|\D_j\D_k^ha\|_{L^2}\\
&\lesssim
\sum_{j\in\Z}2^{j\tau_1}\|\D_ja\|_{L^2}\sum_{k\leq
j+N_0}2^{k\tau_2}\\
&\lesssim\sum_{j\in\Z}2^{j(\tau_1+\tau_2)}\|\D_ja\|_{L^2}\lesssim
\|a\|_{\dot{B}^{\tau_1+\tau_2}_{2,1}},
\end{split}
\eeno which together with \eqref{bb0} completes the proof of the
lemma.
\end{proof}

In  order to obtain a better description of the regularizing effect
of the transport-diffusion equation, we will use Chemin-Lerner type
spaces $\widetilde{L}^{\lambda}_T(B^s_{p,r}(\R^2))$ (see \cite{bcd}
for instance).

\begin{defi}\label{def3}
Let  $(r,\lam,p)\in[1,+\infty]^3$ and $T\in(0,+\infty]$.
 We define the $\wt{L}^\lam_T(\dot{B}^s_{p,r}(\R^2))$  norm by
\beno
\|u\|_{\wt{L}^\lam_T(\dot{B}^s_{p,r})}\eqdefa\Bigl(\sum_{j\in\Z}2^{jrs}
\bigl(\int_0^T\|\D_ju(t)\|_{L^p}^\lam\,dt\bigr)^{\f{r}{\lam}}\Bigr)^{\f{1}{r}}<\infty,
\eeno with the usual change if $r=\infty$. For short, we just denote
this space by  $\wt{L}^\lam_T(\dot{B}^s_{p,r})$.
\end{defi}

\begin{rmk}\label{rmk3}
Corresponding to Definitions \ref{def2} and \ref{def3}, we define
the norm \beno
\|u\|_{\wt{L}^2_T(\cB^{s_1,s_2})}\eqdefa\sum_{j,k\in\Z}2^{js_1}2^{ks_2}\|\Delta_j\D_k^h
u\|_{L^2_T(L^{2})}. \eeno Then it follows from the proof of Lemma
\ref{L1} that \beq\label{bb4}
\|u\|_{\wt{L}^2_T(\cB^{\tau_1,\tau_2})}\lesssim\|u\|_{\wt{L}^2_T(\dot{B}^{\tau_1+\tau_2}_{2,1})}\lesssim\|u\|_{L^2_T(\dot{H}^{s_1})}+\|u\|_{L^2(\dot{H}^{s_2})},
\eeq with $\tau_1,\tau_2$ and $s_1,s_2$ satisfying the assumptions
of Lemma \ref{L1}.
\end{rmk}

We also need both the isotropic  and anisotropic versions of
para-differential decomposition of  Bony \cite{Bo}.  We first recall
the isotropic para-differential decomposition from \cite{Bo}: let
$a, b\in \cS'(\R^2),$  \beq \label{bb5}\begin{split}
&ab=T(a,b)+\cR(a,b), \quad\mbox{or}\quad
ab=T(a,b)+\bar{T}(a,b)+ R(a,b), \quad\hbox{where}\\
& T(a,b)=\sum_{j\in\Z}S_{j-1}a\Delta_jb, \quad
\bar{T}(a,b)=T(b,a),\quad
 \cR(a,b)=\sum_{j\in\Z}\Delta_jaS_{j+2}b, \andf\\
&R(a,b)=\sum_{j\in\Z}\Delta_ja\tilde{\Delta}_{j}b,\quad\hbox{with}\quad
\tilde{\Delta}_{j}b=\sum_{\ell=j-1}^{j+1}\D_\ell b. \end{split} \eeq
We shall also use the following anisotropic version of Bony's
decomposition for the horizontal variables: \beq
\label{bb6}\begin{split} &ab=T^h(a,b)+\cR^h(a,b),
\quad\mbox{or}\quad
ab=T^h(a,b)+\bar{T}^h(a,b)+ R^h(a,b), \quad\hbox{where}\\
& T^h(a,b)=\sum_{k\in\Z}S_{k-1}^ha\Delta_k^h b, \quad
\bar{T}^h(a,b)=T^h(b,a),\quad
 \cR^h(a,b)=\sum_{k\in\Z}\Delta_k^haS_{k+2}^hb, \andf\\
&R^h(a,b)=\sum_{k\in\Z}\Delta_k^ha\tilde{\Delta}_{k}^hb,\quad\hbox{with}\quad
\tilde{\Delta}_{k}^hb=\sum_{\ell=k-1}^{k+1}\D_\ell^h b.
\end{split} \eeq
Considering the special structure of the functions in
$\cB^{s_1,s_2}(\R^2),$ we sometime use both \eqref{bb5} and
\eqref{bb6} simultaneously.

As an application of the above basic facts on Littlewood-Paley
theory, we prove the following product law in space
$\cB^{s_1,s_2}(\R^2)$ given by Definition \ref{def2}.

\begin{lem}\label{L2}
{\sl Let $s_1,s_2,\tau_1,\tau_2\in\R,$ which satisfy  $s_1,
s_2\leq\f{1}{2}$, $\tau_1,\tau_2\leq\f{1}{2}$ and $s_1+s_2>0$,
$\tau_1+\tau_2>0$. Then for $a\in\cB^{s_1,\tau_1}(\R^2)$ and
$b\in\cB^{s_2,\tau_2}(\R^2)$,
$ab\in\cB^{s_1+s_2-\f{1}{2},\tau_1+\tau_2-\f{1}{2}}(\R^2)$ and there
holds \beq\label{bb7}
\|ab\|_{\cB^{s_1+s_2-\f{1}{2},\tau_1+\tau_2-\f{1}{2}}}\lesssim\|a\|_{\cB^{s_1,\tau_1}}\|b\|_{\cB^{s_2,\tau_2}}.
\eeq}
\end{lem}
\begin{proof}
We first get by using Bony's decompositions \eqref{bb5} and
\eqref{bb6} that \beq\label{bb24}\begin{aligned}
ab=&\bigl(TT^h+T\bar{T}^h+TR^h+\bar{T}T^h+\bar{T}\bar{T}^h+\bar{T}R^h
+RT^h+R\bar{T}^h+RR^h\bigr)(a,b).
\end{aligned}\eeq
We shall present the detailed estimates to typical terms above.
Indeed applying Lemma \ref{le2.1} gives \beno
\begin{split}
\|\D_j\D_k^h(TR^h(a,b))\|_{L^2}&\lesssim
2^{\f{k}2}\sum_{{|j'-j|\leq4}\atop{k'\geq
k-N_0}}\|S_{j'-1}\D_{k'}^ha\|_{L^\infty_h(L^2_v)}
\|\D_{j'}\wt{\D}_{k'}^hb\|_{L^2}\\
&\lesssim2^{\f{k}2}\sum_{{|j'-j|\leq4}\atop{k'\geq
k-N_0}}d_{j',k'}2^{j'(\f{1}{2}-s_1-s_2)}2^{-k'(\tau_1+\tau_2)}\|a\|_{\cB^{s_1,\tau_1}}
\|b\|_{\cB^{s_2,\tau_2}}\\
&\lesssim
d_{j,k}2^{-j(s_1+s_2-\f{1}{2})}2^{-k(\tau_1+\tau_2-\f{1}{2})}
\|a\|_{\cB^{s_1,\tau_1}}\|b\|_{\cB^{s_2,\tau_2}}, \end{split}\eeno
as $\tau_1+\tau_2>0$ and $s_1\leq\f12$ so that
$\|S_{j'-1}\D_{k'}^ha\|_{L^\infty_h(L^2_v)}\lesssim
2^{j'(\f12-s_1)}2^{-k'\tau_1}\|a\|_{\cB^{s_1,\tau_1}}.$ The same
estimate holds for $\D_j\D_k^h(\bar{T}R^h(a,b)).$

Along the same line, we have \beno
\begin{split}
\|\D_j\D_k^h(RR^h(a,b))\|_{L^2}&\lesssim
2^{\f{j}{2}}2^{\f{k}{2}}\sum_{{j'\geq j-N_0}\atop{k'\geq k-N_0}}
\|\D_{j'}\D_{k'}^ha\|_{L^2}\|\wt{\D}_{j'}\wt{\D}_{k'}^hb\|_{L^2}\\
&\lesssim 2^{\f{j}{2}}2^{\f{k}{2}}\sum_{{j'\geq j-N_0}\atop{k'\geq k-N_0}}d_{j',k'}2^{-j'(s_1+s_2)}2^{-k'(\tau_1+\tau_2)}\|a\|_{\cB^{s_1,\tau_1}}\|b\|_{\cB^{s_2,\tau_2}}\\
&\lesssim
d_{j,k}2^{-j(s_1+s_2-\f{1}{2})}2^{-k(\tau_1+\tau_2-\f{1}{2})}
\|a\|_{\cB^{s_1,\tau_1}}\|b\|_{\cB^{s_2,\tau_2}} \end{split} \eeno
due to the fact:  $s_1+s_2>0$ and $\tau_1+\tau_2>0$. The estimate to
the remaining terms in \eqref{bb24} is identical, and we omit the
details here.

Whence thanks to \eqref{bb24}, we arrive at \beno
\|\D_j\D_k^h(ab)\|_{L^2}\lesssim
d_{j,k}2^{-j(s_1+s_2-\f{1}{2})}2^{-k(\tau_1+\tau_2-\f{1}{2})}
\|a\|_{\cB^{s_1,\tau_1}}\|b\|_{\cB^{s_2,\tau_2}}, \eeno which
implies \eqref{bb7}. This concludes the proof of  Lemma \ref{L2}.
\end{proof}

We finish this section by some product laws in  $\dH^s(\R^2)$ and
$\dot{B}^s_{2,1}(\R^2)$ (see \cite{bcd} for instance):

\begin{lem}\label{L5}
{\sl For any $s>-1$, there hold
\begin{itemize}
\item[] (i)\quad \ $\|ab\|_{\dH^s}\lesssim\|a\|_{L^\infty}\|b\|_{\dH^s}+\|a\|_{\dH^s}\|b\|_{L^\infty},\quad\text{for}\quad
s>0;$
\item[] (ii)\quad $\|ab\|_{\dH^s}\lesssim\|a\|_{\dot{B}^1_{2,1}}\|b\|_{\dH^s}+\langle\|a\|_{\dH^s}\|b\|_{\dot{B}^1_{2,1}}\rangle_{s>1}$;
\item[] (iii)\quad $\|ab\|_{\dH^s}\lesssim\|a\|_{L^\infty}\|b\|_{\dH^s}+\|a\|_{\dH^{s+1}}\|b\|_{L^2}$,
\end{itemize}
where the notation $A_s=B_s+\langle C_s\rangle_{s>s_0}$ means
$A_s=B_s$ if $s\leq s_0$ and $A_s=B_s+C_s$ if $s>s_0$.}
\end{lem}

\begin{lem}\label{L6}
For any $s>-1$, there hold
\begin{itemize}
\item[] (i)\quad\ $\|ab\|_{\dot{B}^s_{2,1}}\lesssim\|a\|_{\dot{B}^1_{2,1}}\|b\|_{\dot{B}^s_{2,1}}+\langle\|a\|_{\dot{B}^s_{2,1}}
    \|b\|_{\dot{B}^1_{2,1}}\rangle_{s>1}$;
\item[] (ii)\quad $\|ab\|_{\dot{B}^s_{2,1}}\lesssim\|a\|_{\dot{B}^1_{2,1}}\|b\|_{\dot{B}^s_{2,1}}+\|a\|_{\dot{B}^{s+1}_{2,1}}\|b\|_{L^2}$.
\end{itemize}
\end{lem}

\medskip

\setcounter{equation}{0}
\section{$L^1_T(Lip)$ estimate of $Y_t$}\label{sect4}

 As it is well-known, the existence of
solutions to a nonlinear partial differential equation follows
essentially
 from the uniform estimates for its appropriate approximate
solutions. In Subsection \ref{subsect4.3}, we shall present the
uniform global estimates to the approximate solutions of
\eqref{a12}-\eqref{a13} provided that the initial data $(Y_0,Y_1)$
satisfies \eqref{A1a} and \eqref{A1}, and hence the proof of Theorem
\ref{T}. Toward this, a key ingredient used in Section
\ref{subsect4.3} will be the $L^1(\R^+; Lip(\R^2))$ estimate of
$Y_t.$ And this is the purpose of this section.

\subsection{The  $\|\na Y_t\|_{L^1_T(L^\infty)}$ estimate for
solutions of \eqref{a12ag}.} We first investigate the
 $\|\na Y_t\|_{L^1_T(L^\infty)}$ estimate on  solutions of the  linear equation
 \eqref{a12ag}.

\begin{prop}\label{p1}{\sl  Let $Y$ be a sufficiently smooth  solution of \eqref{a12ag} on $[0,T]$. Then there holds
\beq\label{2.9} \|\na Y_t\|_{L^1_T(L^\infty)}\lesssim
\|Y_t\|_{L_T^1(\cB^{\f{3}{2},\f{1}{2}})}\lesssim\|Y_1\|_{\cB^{-\f{1}{2},\f{1}{2}}}+\|\p_1
Y_0\|_{\cB^{-\f{1}{2},\f{1}{2}}}+\|\D
Y_0\|_{\cB^{-\f{1}{2},\f{1}{2}}}+\|\vv f\|_{L^1_T(\cB^{0,0})}. \eeq}
\end{prop}
\begin{proof}
Applying operator $\D_j\D_k^h$ to \eqref{a12ag}, we first get  that
\beq
\D_j\D_k^hY_{tt}-\D\D_j\D_k^hY_t-\p_1^2\D_j\D_k^hY=\D_j\D_k^h\vv f.
\label{2.1} \eeq Taking the $L^2$ inner product of \eqref{2.1} with
$\D_j\D_k^hY_t$ gives \beq
\f12\f{d}{dt}\Bigl(\|\D_j\D_k^hY_t\|_{L^2}^2+\|\p_1\D_j\D_k^hY\|_{L^2}^2\Bigr)+\|\na\D_j\D_k^hY_t\|_{L^2}^2
=(\D_j\D_k^h \vv f\ |\ \D_j\D_k^hY_t). \label{2.2} \eeq While taking
the $L^2$ inner product of \eqref{2.1} with $\D\D_j\D_k^hY$ leads to
\beno (\D_j\D_k^hY_{tt}\ |\
\D\D_j\D_k^hY)-\f12\f{d}{dt}\|\D\D_j\D_k^h
Y\|_{L^2}^2-\|\p_1\na\D_j\D_k^hY\|_{L^2}^2=(\D_j\D_k^h \vv f\ |\
\D\D_j\D_k^hY). \eeno Notice that \beno (\D_j\D_k^hY_{tt}\ |\
\D\D_j\D_k^hY)=\f{d}{dt}(\D_j\D_k^hY_{t}\ |\
\D\D_j\D_k^hY)-(\D_j\D_k^hY_{t}\ |\ \D\D_j\D_k^hY_t), \eeno hence
there holds that \beq\label{2.3}
\begin{split}
&\f{d}{dt}\Bigl(\f12\|\D\D_j\D_k^h Y\|_{L^2}^2-(\D_j\D_k^hY_{t}\ |\
\D\D_j\D_k^hY)\Bigr)-\|\na\D_j\D_k^hY_t\|_{L^2}^2+\|\p_1\na\D_j\D_k^hY\|_{L^2}^2\\
&=-(\D_j\D_k^h\vv f\ |\ \D\D_j\D_k^hY).
\end{split}
\eeq Summing up \eqref{2.2} with $\f14\times$\eqref{2.3} gives
\beq\label{2.4}
\begin{split}
\f{d}{dt} &g_{j,k}(t)^2+\f34\|\na\D_j\D_k^hY_t\|_{L^2}^2+\f14\|\p_1\na\D_j\D_k^hY\|_{L^2}^2\\
&=\bigl(\D_j\D_k^h\vv f\ |\ \D_j\D_k^hY_t-\f14\D\D_j\D_k^hY\bigr),
\end{split}
\eeq
where
 \beno\begin{split}
g_{j,k}(t)^2\eqdefa&\f12\bigl(\|\D_j\D_k^hY_t(t)\|_{L^2}^2+\|\p_1\D_j\D_k^hY(t)\|_{L^2}^2
+\f14\|\D\D_j\D_k^hY(t)\|_{L^2}^2\bigr)\\
&-\f14\bigl(\D_j\D_k^hY_t(t)\ |\ \D\D_j\D_k^hY(t)\bigr).
\end{split}
\eeno

It is easy to check that \beno \f14\bigl|\bigl(\D_j\D_k^hY_t\ |\
\D\D_j\D_k^hY\bigr)\bigr|\leq\f1{16}\|\D\D_j\D_k^hY\|_{L^2}^2+\f14\|\D_j\D_k^hY_t\|_{L^2}^2,
\eeno which implies \beno
\begin{split}
\f14\|\D_j&\D_k^hY_t(t)\|_{L^2}^2+\f12\|\p_1\D_j\D_k^hY(t)\|_{L^2}^2+\f1{16}\|\D\D_j\D_k^hY(t)\|_{L^2}^2\\
&\leq
g_{j,k}(t)^2\leq\f34\|\D_j\D_k^hY_t(t)\|_{L^2}^2+\f12\|\p_1\D_j\D_k^hY(t)\|_{L^2}^2+\f3{16}\|\D\D_j\D_k^hY(t)\|_{L^2}^2,
\end{split}
\eeno or equivalently
 \beq \label{2.5}
g_{j,k}(t)^2\sim\|\D_j\D_k^hY_t(t)\|_{L^2}^2+\|\p_1\D_j\D_k^hY(t)\|_{L^2}^2+\|\D\D_j\D_k^hY(t)\|_{L^2}^2.
\eeq With \eqref{2.4} and \eqref{2.5}, and according to the
heuristic analysis in Section \ref{sect2}, we shall divide the proof
of \eqref{2.9} into two cases: one is when  $j\leq \f{k+1}2,$  and
the other one is when  $j> \f{k+1}2.$

\no {\bf Case (1)}: when $j\leq \f{k+1}2.$ In this case, we infer
from Lemma \ref{le2.1} and \eqref{2.5} that \beno g_{j,k}(t)^2\sim
\|\D_j\D_k^hY_t(t)\|_{L^2}^2+\|\p_1\D_j\D_k^hY(t)\|_{L^2}^2, \eeno
and \beno
\begin{split}
\|\na\D_j\D_k^hY_t(t)&\|_{L^2}^2+\|\p_1\na\D_j\D_k^hY(t)\|_{L^2}^2\\
&\geq
c2^{2j}\bigl(\|\D_j\D_k^hY_t(t)\|_{L^2}^2+\|\p_1\D_j\D_k^hY(t)\|_{L^2}^2\bigr)
\geq c2^{2j}g_{j,k}(t)^2, \end{split}\eeno from which, for any
$\e>0,$ dividing \eqref{2.4} by $g_{j,k}(t)+\e,$ then taking $\e\to
0$ and integrating the resulting equation over $[0,T],$ we deduce
\beq \label{2.6}
\begin{split}
\|&\D_j\D_k^hY_t\|_{L^\infty_T(L^2)}+\|\p_1\D_j\D_k^hY\|_{L^\infty_T(L^2)}\\
&\quad+
c2^{2j}\bigl(\|\D_j\D_k^hY_t\|_{L^1_T(L^2)}+\|\p_1\D_j\D_k^hY\|_{L^1_T(L^2)}\bigr)\\
&\lesssim
\|\D_j\D_k^hY_1\|_{L^2}+\|\p_1\D_j\D_k^hY_0\|_{L^2}+\|\D_j\D_k^h\vv
f\|_{L^1_T(L^2)}.
\end{split}
\eeq

\no {\bf Case (2)}: when $j> \f{k+1}2.$ In this case, we  apply
Lemma \ref{le2.1} to obtain that \beno
\begin{split}
&g_{j,k}(t)^2 \sim
\|\D_j\D_k^hY_t(t)\|_{L^2}^2+\|\D\D_j\D_k^hY(t)\|_{L^2}^2,\end{split}
\eeno  and
 \beno
\begin{split}
&\|\na\D_j\D_k^hY_t(t)\|_{L^2}^2+\|\p_1\na\D_j\D_k^hY(t)\|_{L^2}^2\\
&\qquad\geq
c\frac{2^{2k}}{2^{2j}}\bigl(\|\D_j\D_k^hY_t(t)\|_{L^2}^2+\|\D\D_j\D_k^hY(t)\|_{L^2}^2\bigr)
\geq c\frac{2^{2k}}{2^{2j}} g_{j,k}(t)^2,
\end{split}
\eeno from which and \eqref{2.4},  we deduce by a similar argument
for \eqref{2.6}
 that \beq \label{2.7}
\begin{split}
\|&\D_j\D_k^hY_t\|_{L^\infty_T(L^2)}+\|\D\D_j\D_k^hY\|_{L^\infty_T(L^2)}\\
&\qquad+
c\f{2^{2k}}{2^{2j}}\bigl(\|\D_j\D_k^hY_t\|_{L^1_T(L^2)}+\|\D\D_j\D_k^hY\|_{L^1_T(L^2)}\bigr)\\
&\lesssim \|\D_j\D_k^hY_1\|_{L^2}+\|\D\D_j\D_k^hY_0\|_{L^2}+
\|\D_j\D_k^h\vv f\|_{L^1_T(L^2)}.
\end{split}
\eeq

On the other hand, via \eqref{2.2}, one has \beno
\f12\f{d}{dt}\|\D_j\D_k^hY_t(t)\|_{L^2}^2+\|\na\D_j\D_k^hY_t(t)\|_{L^2}^2=\bigl(\p_1^2\D_j\D_k^hY+\D_j\D_k^h\vv
f\ |\ \D_j\D_k^hY_t\bigr), \eeno from which, Lemma \ref{le2.1} and
\eqref{2.7}, we conclude \beq\label{2.8}
\begin{split}
\|\D_j\D_k^h&Y_t\|_{L^\infty_T(L^2)}+c2^{2j}\|\D_j\D_k^hY_t\|_{L^1_T(L^2)}\\
&\leq \|\D_j\D_k^hY_1\|_{L^2}+C\bigl(2^{2k}\|\D_j\D_k^hY\|_{L^1_T(L^2)}+\|\D_j\D_k^h\vv f\|_{L^1_T(L^2)}\bigr)\\
&\lesssim
\|\D_j\D_k^hY_1\|_{L^2}+\|\D\D_j\D_k^hY_0\|_{L^2}+\|\D_j\D_k^h\vv
f\|_{L^1_T(L^2)}\quad\mbox{for}\quad j>\f{k+1}2.
\end{split}
\eeq

By Definition \ref{def2}, we get, by  combining \eqref{2.6} with
\eqref{2.8}, that \beno
\|Y_t\|_{L_T^1(\cB^{\f{3}{2},\f{1}{2}})}\lesssim\|Y_1\|_{\cB^{-\f{1}{2},\f{1}{2}}}+\|\p_1
Y_0\|_{\cB^{-\f{1}{2},\f{1}{2}}}+\|\D
Y_0\|_{\cB^{-\f{1}{2},\f{1}{2}}}+\|\vv f\|_{L^1_T(\cB^{0,0})}, \eeno
which together with the fact \beno \begin{split} \|\na
Y_t\|_{L_T^1(L^\infty)}&\lesssim \sum_{j,k\in\Z}\|\D_j\D^h_k\na
Y_t\|_{L_T^1(L^\infty)}\\
&\lesssim\sum_{\substack{j,k,\ell\in\Z\\
j\geq\ell-N_0}}2^{j}2^{\f{k}{2}}2^{\f{\ell}{2}}\|\D_j\D^h_k\D^v_\ell
Y_t\|_{L_T^1(L^2)}\\
&\lesssim\sum_{j,k\in\Z}2^{\f{3j}2}2^{\f{k}{2}}\|\D_j\D^h_k
Y_t\|_{L_T^1(L^2)}\lesssim \|Y_t\|_{L_T^1(\cB^{\f{3}{2},\f{1}{2}})},
\end{split}\eeno implies \eqref{2.9}. This finishes the proof of
Proposition \ref{p1}.
\end{proof}

\subsection{$L^1_T(\cB^{0,0})$ estimate of  $\vv f(Y,q)$ given by \eqref{a13}.}

With \eqref{2.9}, to derive the $L^1_T(Lip(\R^2))$ estimate of $Y_t$
for smooth enough solutions of \eqref{a12}, we  need to estimate the
$L^1_T(\cB^{0,0})$ norm of
 $\vv f(Y,q)$ given by \eqref{a13}. This will be done in  this subsection.

\begin{prop}\label{p2} {\sl Let $(Y, q)$ be a smooth enough solution of \eqref{a14} (or equivalently \eqref{a12}-\eqref{a13})
on $[0,T]$, with the initial data satisfying \eqref{A1a}. We assume
further that \beq\label{bb14} \|\na
Y\|_{L^\infty_T(\dot{B}^1_{2,1})}\leq c_0 \eeq for some $c_0$
sufficiently small. Then one has \beq\label{bb15}
\|\na q\|_{L^1_T(\cB^{0,0})}\leq C\Bigl\{\|Y_t\|_{L_T^2(\cB^{\f{1}{4},\f{1}{4}})}\|\na Y_t\|_{L_T^2(\cB^{\f{1}{4},\f{1}{4}})}\\
 +\|\p_1 Y\|_{L_T^2(\cB^{\f{1}{4},\f{1}{4}})}\|\p_1\na
Y\|_{L_T^2(\cB^{\f{1}{4},\f{1}{4}})}\Bigr\}.
\eeq}
\end{prop}
\begin{proof}
Thanks to \eqref{a14}, we get by taking $\p_t$ to $\na_Y\cdot Y_t=0$
that \beq\label{bb15ad} \na_Y\cdot Y_{tt}=-\p_t\cA_Y^T\na\cdot Y_t.
\eeq While for $c_0$ sufficiently small so that \beno \|\na
Y\|_{L^\infty_T(L^\infty)}\leq C\|\na
Y\|_{L^\infty_T(\dot{B}^1_{2,1})}\leq Cc_0\leq\f{1}{2}, \eeno
$X(t,y)$ determined by \eqref{a6} has a smooth inverse map
$X^{-1}(t,x)$ with $X(t,X^{-1}(t,x))=x$ and $ X^{-1}(t,X(t,y))=y.$
Then it follows from $\na_Y\cdot Y_t=0$ that \beno
\na_Y\cdot(\na_Y\cdot\na_Y Y_t)=[\na_x\cdot\D_x(Y_t\circ
X^{-1}(t,x))]\circ X(t,y)=\na_Y\cdot\na_Y(\na_Y\cdot Y_t)=0,\eeno
from the latter  and \eqref{bb15ad}, we obtain by  taking
$\na_Y\cdot$ to the first equation of \eqref{a14} that \beno
\na_Y\cdot\na_Y q=\p_t\cA_Y^T\na\cdot Y_t+\na_Y\cdot\p_1^2Y, \eeno
or equivalently \beq\label{bb8}
\begin{split}
\D q=&-(\na_Y-\na)\cdot\na_Y q-\na\cdot(\na_Y-\na)q+\p_t\cA_Y^T\na\cdot Y_t+\na_Y\cdot\p_1^2Y\\
=&-\na\cdot\bigl((\cA_Y-I)\cA_Y^T\na
q\bigr)-\na\cdot\bigl((\cA_Y^T-I)\na q\bigr)+\na\cdot(\p_t\cA_Y
Y_t)+\na_Y\cdot\p_1^2Y. \end{split} \eeq  The above, Lemma
\ref{le2.1} and Definition \ref{def2} lead to
\beq\label{bb9}\begin{aligned}
\|\na q\|_{L^1_T(\cB^{0,0})}\leq&\|(\cA_Y-I)\cA_Y^T\na q\|_{L^1_T(\cB^{0,0})}+\|(\cA_Y^T-I)\na q\|_{L^1_T(\cB^{0,0})}\\
&+\|\p_t\cA_Y
Y_t\|_{L^1_T(\cB^{0,0})}+\|\na(-\D)^{-1}(\na_Y\cdot\p_1^2Y)\|_{L^1_T(\cB^{0,0})}.
\end{aligned}\eeq
Applying Lemma \ref{L2} and \eqref{a14a}, one has  \beq\label{bb10}
\begin{split}
&\|(\cA_Y^T-I)\na q\|_{L^1_T(\cB^{0,0})}\lesssim\|\na
Y\|_{L^\infty_T(\cB^{\f{1}{2},\f{1}{2}})}\|\na
q\|_{L^1_T(\cB^{0,0})},\\
&\|(\cA_Y-I)\cA_Y^T\na q\|_{L^1_T(\cB^{0,0})} \lesssim(1+\|\na
Y\|_{L^\infty_T(\cB^{\f{1}{2},\f{1}{2}})})\|\na
Y\|_{L^\infty_T(\cB^{\f{1}{2},\f{1}{2}})}\|\na
q\|_{L^1_T(\cB^{0,0})}\quad\mbox{and}\\
& \|\p_t\cA_Y Y_t\|_{L^1_T(\cB^{0,0})}\lesssim\|\na
Y_t\|_{L_T^2(\cB^{\f{1}{4},\f{1}{4}})}\|Y_t\|_{L_T^2(\cB^{\f{1}{4},\f{1}{4}})}.
\end{split}
\eeq While by \eqref{a12} and \eqref{a13}, we have $\na\cdot
Y=\r(Y)=\p_1Y^2\p_2Y^1-\p_1Y^1\p_2Y^2,$ from which, we deduce
\beq\label{bb13a}
\begin{split}
\na_Y\cdot\p_1^2Y=&\na\cdot\bigl((\cA_Y-I)\p_1^2Y\bigr)+\p_1^2\r(Y)\\
=&\p_1\bigl(\p_1Y^2\p_1\p_2Y^1-\p_1Y^1\p_1\p_2Y^2\bigr)+\p_2(-\p_1Y^2\p_1^2Y^1+\p_1Y^1\p_1^2Y^2). \end{split} \eeq
Applying Lemma \ref{le2.1} and Lemma \ref{L2} leads to
\beq\label{bb13}\begin{aligned}
\|\na(-\D)^{-1}(\na_Y&\cdot\p_1^2Y)\|_{L^1_T(\cB^{0,0})}
\lesssim\|\p_1 Y\|_{L_T^2(\cB^{\f{1}{4},\f{1}{4}})}\|\p_1\na
Y\|_{L_T^2(\cB^{\f{1}{4},\f{1}{4}})}.
\end{aligned}\eeq
Substituting \eqref{bb10} and \eqref{bb13} into \eqref{bb9} and
using Lemma \ref{L1}, we conclude that \beno
\begin{split}
\|\na q\|_{L^1_T(\cB^{0,0})}\lesssim &(1+\|\na Y\|_{L^\infty_T(\dot{B}^1_{2,1})})\|\na Y\|_{L^\infty_T(\dot{B}^1_{2,1})}\|\na q\|_{L^1_T(\cB^{0,0})}\\
&+\|\na Y_t\|_{L_T^2(\cB^{\f{1}{4},\f{1}{4}})}\|Y_t\|_{L_T^2(\cB^{\f{1}{4},\f{1}{4}})}+\|\p_1 Y\|_{L_T^2(\cB^{\f{1}{4},\f{1}{4}})}\|\p_1\na
Y\|_{L_T^2(\cB^{\f{1}{4},\f{1}{4}})}, \end{split} \eeno which
together with \eqref{bb14} ensures \eqref{bb15}. This completes the
proof of Proposition \ref{p2}.
\end{proof}

We now turn to the estimate of $\vv f(Y,q)$ given by \eqref{a13}.
The main result is as follows:

\begin{prop}\label{p3}
{\sl Let $\vv f(Y,q)$ be given by \eqref{a13}. Under the assumptions
of Proposition \ref{p2}, and , in addition,  if \beq\label{A4} \|\na
Y\|_{L^\infty_T(\dot{B}^1_{2,1})}+\|\na
Y\|_{L^\infty_T(\dot{B}^2_{2,1})}\leq 1, \eeq   we have
\beq\label{bb17}\begin{aligned} \|\vv
f&(Y,q)\|_{L^1_T(\cB^{0,0})}\lesssim\|\na
Y\|_{L^\infty_T(\dot{B}^1_{2,1})}\|Y_t\|_{L^1_T(\cB^{\f{3}{2},\f{1}{2}})}+
\|Y_t\|_{L_T^2(\cB^{\f{1}{4},\f{1}{4}})}\|\na Y_t\|_{L_T^2(\cB^{\f{1}{4},\f{1}{4}})}\\
& +\|\p_1 Y\|_{L_T^2(\cB^{\f{1}{4},\f{1}{4}})}\|\p_1\na
Y\|_{L_T^2(\cB^{\f{1}{4},\f{1}{4}})}\\
&+\bigl(\|\p_1 Y\|_{L_T^2(\cB^{\f{1}{4},\f{1}{4}})}
+\|\p_1\na Y\|_{L_T^2(\cB^{\f{1}{4},\f{1}{4}})}\bigr)\bigl(\|\na Y_t\|_{L_T^2(\cB^{\f{1}{4},\f{1}{4}})}+\|\na^2 Y_t\|_{L_T^2(\cB^{\f{1}{4},\f{1}{4}})}\bigr)\\
&+\bigl(\|\p_1 Y\|_{L_T^2(\dot{B}^1_{2,1})}+\|\p_1
Y\|_{L_T^2(\dot{B}^2_{2,1})}\bigr)\|Y_t\|_{L_T^2(\cB^{\f{1}{2},\f{1}{2}})}.
\end{aligned}\eeq}
\end{prop}
\begin{proof} Via \eqref{a13},
 we  split $\vv f(Y,q)$ as \beq \label{bb17a}
\begin{split}
& \qquad\qquad\vv f(Y,q)=\bar{\vv f}(Y)+\tilde{\vv f}(Y,q)\quad \mbox{ with}\\
& \bar{\vv
f}(Y)\eqdefa(\na_Y\cdot\na_Y-\Delta)Y_t,\quad\mbox{and}\quad
\tilde{\vv f}(Y,q)\eqdefa-\na_Y q. \end{split}\eeq Notice from the
definition of $\cA_Y=(b_{ij})_{i,j=1,2}$ given by \eqref{a14a} that
\beq\label{b15}\begin{split} \bar{\vv
f}(Y)=\sum_{\ell,m,j=1}^2\p_\ell(b_{\ell j}b_{mj}\p_mY_t)-\Delta
Y_t\eqdefa\p_1\vv F_1+\p_2\vv F_2, \end{split} \eeq with
\beq\label{b16}\begin{aligned}
&\vv F_1=(2\p_2Y^2+|\p_2Y|^2)\p_1Y_t-(1+\p_2Y^2)\p_1Y^2\p_2Y_t-(1+\p_1Y^1)\p_2Y^1\p_2Y_t,\\
&\vv
F_2=(2\p_1Y^1+|\p_1Y|^2)\p_2Y_t-(1+\p_2Y^2)\p_1Y^2\p_1Y_t-(1+\p_1Y^1)\p_2Y^1\p_1Y_t,
\end{aligned}\eeq
and \beq\label{b17} \tilde{\vv f}(Y,q)=-(\mathcal{A}_Y^T-I)\na q-\na
q. \eeq We first observe  from \eqref{bb10} and \eqref{b17} that
\beq\label{bb16} \|\tilde{\vv
f}(Y,q)\|_{L^1_T(\cB^{0,0})}\lesssim(1+\|\na
Y\|_{L^\infty_T(\dot{B}^1_{2,1})})\|\na q\|_{L^1_T(\cB^{0,0})}. \eeq

To deal with $\bar{\vv f}(Y)$, we need the following two lemmas,
their proofs  will be presented in the Appendix \ref{appendb}.

\begin{lem}\label{L3} {\sl Under the assumptions of Proposition \ref{p3}, one has
\beno
\|\p_2(\p_2Y^1\p_1Y_t)\|_{L^1_T(\cB^{0,0})}\lesssim\|\na
Y\|_{L^\infty_T(\dot{B}^1_{2,1})}\|Y_t\|_{L^1_T(\cB^{\f{3}{2},\f{1}{2}})}+\|\p_1
Y^1\|_{L^2_T(\dot{B}^2_{2,1})}\|Y_t\|_{L^2_T(\cB^{\f{1}{2},\f{1}{2}})}.
\eeno}
\end{lem}

\begin{lem}\label{L4} {\sl Under the assumptions of Proposition \ref{p3}, one has
\beno
\|\p_1(\p_2Y^1\p_2Y_t)\|_{L^1_T(\cB^{0,0})}\lesssim\|\p_2
Y^1\|_{L^\infty_T(\dot{B}^1_{2,1})}\|Y_t\|_{L^1_T(\cB^{\f{3}{2},\f{1}{2}})}.
\eeno}
\end{lem}

Now, let us go back to the proof of Proposition \ref{p3}. We first,
 by applying  Lemma \ref{L3}, to get that \beq\label{bb18}\begin{aligned}
&\|\p_2(\p_2Y^1\p_1Y_t)\|_{L^1_T(\cB^{0,0})}+\|\p_2(\p_1Y^2\p_1Y_t)\|_{L^1_T(\cB^{0,0})}\\
&\lesssim\|\na
Y\|_{L^\infty_T(\dot{B}^1_{2,1})}\|Y_t\|_{L^1_T(\cB^{\f{3}{2},\f{1}{2}})}+\|\p_1
Y\|_{L^2_T(\dot{B}^2_{2,1})}\|Y_t\|_{L^2_T(\cB^{\f{1}{2},\f{1}{2}})}.
\end{aligned}\eeq
By the same argument and using the product laws in Lemma \ref{L6},
one has \beq\label{bb23}\begin{aligned}
&\|\p_2(\p_1Y^1\p_2Y^1\p_1Y_t)\|_{L^1_T(\cB^{0,0})}+\|\p_2(\p_2Y^2\p_1Y^2\p_1Y_t)\|_{L^1_T(\cB^{0,0})}\\
&\lesssim\|\p_1Y\p_2
Y\|_{L^\infty_T(\dot{B}^1_{2,1})}\|Y_t\|_{L^1_T(\cB^{\f{3}{2},\f{1}{2}})}+\|\p_1Y\p_2Y\|_{L^2_T(\dot
B^2_{2,1})}
\|Y_t\|_{L^2_T(\cB^{\f12,\f12})}\\
& \lesssim\|\na
Y\|_{L^\infty_T(\dot{B}^1_{2,1})}^2\|Y_t\|_{L^1_T(\cB^{\f{3}{2},\f{1}{2}})}+\bigl(\|\p_2Y\|_{L^\infty_T(\dot
B^2_{2,1})}\|\p_1Y\|_{L^2_T(\dot
B^1_{2,1})}\\
&\qquad+\|\p_2Y\|_{L^\infty_T(\dot B^1_{2,1})}\|\p_1Y\|_{L^2_T(\dot
B^2_{2,1})}\bigr)\|Y_t\|_{L^2_T(\cB^{\f12,\f12})}.
\end{aligned}\eeq
Whereas applying \eqref{bb7} gives  \beq\label{bb22}\begin{aligned}
\|\p_2(\p_1Y^1\p_2Y_t)\|_{L^1_T(\cB^{0,0})}\lesssim&\|\p_1 Y^1\|_{L^2_T(\cB^{\f{1}{4},\f{1}{4}})}\|\p_2^2Y_t\|_{L^2_T(\cB^{\f{1}{4},\f{1}{4}})}\\
&+\|\p_1\p_2
Y^1\|_{L^2_T(\cB^{\f{1}{4},\f{1}{4}})}\|\p_2Y_t\|_{L^2_T(\cB^{\f{1}{4},\f{1}{4}})},
\end{aligned}\eeq
and applying \eqref{bb7} twice leads to
\beq\label{bb25}\begin{aligned}
&\|\p_2(|\p_1Y|^2\p_2Y_t)\|_{L^1_T(\cB^{0,0})}\lesssim\||\p_1Y|^2\|_{L^2_T(\cB^{\f{1}{4},\f{1}{4}})}
\|\p_2^2Y_t\|_{L^2_T(\cB^{\f{1}{4},\f{1}{4}})}\\
&\quad+\|\p_2(|\p_1Y|^2)\|_{L^2_T(\cB^{\f{1}{4},\f{1}{4}})}\|\p_2Y_t\|_{L^2_T(\cB^{\f{1}{4},\f{1}{4}})}\\
&\lesssim\|\na Y\|_{L^\infty_T(\cB^{\f{1}{2},\f{1}{2}})}\bigl(\|\p_1Y\|_{L^2_T(\cB^{\f{1}{4},\f{1}{4}})}+\|\p_1\na Y\|_{L^2_T(\cB^{\f{1}{4},\f{1}{4}})}\bigr)\\
&\qquad\times\bigl(\|\p_2^2
Y_t\|_{L^2_T(\cB^{\f{1}{4},\f{1}{4}})}+\|\p_2Y_t\|_{L^2_T(\cB^{\f{1}{4},\f{1}{4}})}\bigr).
\end{aligned}\eeq

Combining \eqref{bb18}-\eqref{bb25} and using Lemma \ref{L1} and
\eqref{b16}, we obtain \beq\label{bb25ad}\begin{split} &\|\p_2\vv
F_2\|_{L^1_T(\cB^{0,0})}\lesssim(1+\|\na
Y\|_{L_T^\infty(\dot{B}^1_{2,1})})\bigl\{\|\na
Y\|_{L^\infty_T(\dot{B}^1_{2,1})}\|Y_t\|_{L^1_T(\cB^{\f{3}{2},\f{1}{2}})}\\
 &\qquad+\bigl(\|\p_1 Y\|_{L_T^2(\cB^{\f{1}{4},\f{1}{4}})}
+\|\p_1\na Y\|_{L_T^2(\cB^{\f{1}{4},\f{1}{4}})}\bigr)\bigl(\|\p_2^2 Y_t\|_{L_T^2(\cB^{\f{1}{4},\f{1}{4}})}+\|\p_2 Y_t\|_{L_T^2(\cB^{\f{1}{4},\f{1}{4}})}\bigr)\bigr\}\\
&\qquad+\bigl(1+\|\p_2Y\|_{L^\infty_T(\dot{B}^1_{2,1})}+\|\p_2Y\|_{L^\infty_T(\dot{B}^2_{2,1})}\bigr)\\
&\qquad\quad\times\bigl(\|\p_1 Y\|_{L_T^2(\dot{B}^1_{2,1})}+\|\p_1
Y\|_{L_T^2(\dot{B}^2_{2,1})}\bigr)\|Y_t\|_{L_T^2(\cB^{\f{1}{2},\f{1}{2}})}.
\end{split}\eeq

On the other hand, via \eqref{b16}, we deduce from  Lemma \ref{L4}
along with its proof that \beq\label{bb21b} \|\p_1\vv
F_1\|_{L^1_T(\cB^{0,0})}\lesssim(1+\|\na
Y\|_{L^\infty_T(\dot{B}^1_{2,1})})\|\na
Y\|_{L^\infty_T(\dot{B}^1_{2,1})}\|Y_t\|_{L^1_T(\cB^{\f{3}{2},\f{1}{2}})}.
\eeq

Therefore, under the assumptions \eqref{bb14} and \eqref{A4}, we
arrive at \eqref{bb17} by summing up \eqref{bb15}, \eqref{b15},
\eqref{bb16}, \eqref{bb25ad}  and \eqref{bb21b}. This concludes the
proof of Proposition \ref{p3}.
\end{proof}

\subsection{$L^1_T(Lip(\R^2))$ estimate of $Y_t$.}

We now conclude this section with

\begin{prop}\label{p4}
Under the assumption of Proposition \ref{p3}, one has
\beq\label{bb26}\begin{aligned} \|\na
Y_t\|_{L^1_T(L^\infty)}\lesssim&\|
Y_1\|_{\dot{B}^{0}_{2,1}}+\|\p_1Y_0\|_{\dot{B}^0_{2,1}}
+\|Y_0\|_{\dot{B}^2_{2,1}}+\|Y_t\|_{L_T^2(\dot{B}^{\f12}_{2,1})}^2\\
&+\|Y_t\|_{L_T^2(\dot{B}^{\f52}_{2,1})}^2
+\|\p_1Y\|_{L_T^2(\dot{B}^{\f12}_{2,1})}^2+\|\p_1Y\|_{L_T^2(\dot{B}^{2}_{2,1})}^2.
\end{aligned}\eeq
\end{prop}
\begin{proof}
Thanks to Propositions \ref{p1} and \ref{p3}, and the assumption
\eqref{bb14}, we infer \beq\label{bb27}\begin{aligned} \|\na
&Y_t\|_{L^1_T(L^\infty)}\lesssim \|
Y_t\|_{L^1_T(\cB^{\f32,\f12})}\lesssim
\|Y_1\|_{\cB^{-\f{1}{2},\f{1}{2}}}+\|\p_1
Y_0\|_{\cB^{-\f{1}{2},\f{1}{2}}}+\|\D
Y_0\|_{\cB^{-\f{1}{2},\f{1}{2}}}
\\
&+
\|Y_t\|_{L_T^2(\cB^{\f{1}{4},\f{1}{4}})}\|\na Y_t\|_{L_T^2(\cB^{\f{1}{4},\f{1}{4}})}+\bigl(\|\p_1 Y\|_{L_T^2(\dot{B}^1_{2,1})}+\|\p_1 Y\|_{L_T^2(\dot{B}^2_{2,1})}\bigr)\|Y_t\|_{L_T^2(\cB^{\f{1}{2},\f{1}{2}})} \\
& +\bigl(\|\p_1 Y\|_{L_T^2(\cB^{\f{1}{4},\f{1}{4}})}+\|\p_1\na
Y\|_{L_T^2(\cB^{\f{1}{4},\f{1}{4}})}
\bigr)\bigl(\|\na Y_t\|_{L_T^2(\cB^{\f{1}{4},\f{1}{4}})}+\|\na^2Y_t\|_{L_T^2(\cB^{\f{1}{4},\f{1}{4}})}\bigr)\\
&+\|\p_1 Y\|_{L_T^2(\cB^{\f{1}{4},\f{1}{4}})}\|\p_1\na
Y\|_{L_T^2(\cB^{\f{1}{4},\f{1}{4}})},
\end{aligned}\eeq
which along with Lemma \ref{L1} and \eqref{bb4} lead to
\eqref{bb26}. This completes the proof of Proposition \ref{p4}.
\end{proof}

\medskip
\section{The proof of Theorem \ref{T}}\label{sect5}
\setcounter{equation}{0}

\subsection{{\it A priori} estimate of \eqref{a14}} With $L^1_T(Lip(\R^2))$ estimate of $Y_t,$
we can now proceed with the energy method.

\begin{prop}\label{prop5.1}
{\sl Let  $Y$ be a sufficiently smooth  solution of \eqref{a12} on
$[0,T].$ Then one has
 \beq \label{b14a}
\begin{split}
&\|\Delta_jY_t\|_{L^\infty_T(L^2)}^2+\|\na\Delta_jY_t\|_{L^\infty_T(L^2)}^2+\|\p_1\Delta_jY\|_{L^\infty_T(L^2)}^2+\|\na\Delta_jY^2\|_{L^\infty_T(L^2)}^2
\\
&\quad +\|\Delta\Delta_jY\|_{L^\infty_T(L^2)}^2
+\|\na\Delta_jY_t\|_{L^2_T(L^2)}^2+\|\D\Delta_jY_t\|_{L^2_T(L^2)}^2\\
&\quad+\|\p_1\na\Delta_jY\|_{L^2_T(L^2)}^2+\|\Delta\Delta_jY^2\|_{L^2_T(L^2)}^2
\\
&\lesssim\|\Delta_jY_1\|_{L^2}^2+\|\na\Delta_jY_1\|_{L^2}^2+\|\p_1\Delta_jY_0\|_{L^2}^2+\|\Delta\Delta_jY_0\|_{L^2}^2+\|\Delta_j\r\|_{L^\infty_T(L^2)}^2\\
&\quad
+\|\na\Delta_j\r\|_{L^2_T(L^2)}^2+\bigl|\int_0^T\bigl(\Delta_j\vv f\
|\ \Delta_jY_t-\f{1}{4}\Delta\Delta_jY-\D\D_jY_t\bigr)\,dt\bigr|.
\end{split}\eeq
}\end{prop}
\begin{proof}
Applying $\Delta_j$ to \eqref{a12} gives \beq\label{b4} \Delta_j
Y_{tt}-\Delta \Delta_j Y_t-\p_1^2 \Delta_j Y=\Delta_j\vv f, \eeq
from which, and along the same line of the proof of \eqref{2.4}, we
get by taking the $L^2$ inner product of \eqref{b4} with $\Delta_j
Y_t-\f{1}{4}\Delta\Delta_jY-\Delta\Delta_jY_t$ that
\beq\label{b5}\begin{aligned}
&\f{d}{dt}\Bigl\{\f{1}{2}\bigl(\|\Delta_jY_t\|_{L^2}^2+\|\na\Delta_jY_t\|_{L^2}^2+\|\p_1\Delta_jY\|_{L^2}^2+\|\p_1\na\Delta_jY\|_{L^2}^2
+\f{1}{4}\|\Delta\Delta_jY\|_{L^2}^2\bigr)
\\
&\quad-\f{1}{4}(\Delta_jY_t\ |\ \Delta\Delta_jY)\Bigr\}+\f{3}{4}\|\na\Delta_jY_t\|_{L^2}^2+\|\D\Delta_jY_t\|_{L^2}^2+\f{1}{4}\|\p_1\na\Delta_jY\|_{L^2}^2\\
&=\bigl(\Delta_j\vv
f\ |\ \Delta_jY_t-\f{1}{4}\Delta\Delta_jY-\Delta\Delta_jY_t\bigr).
\end{aligned}\eeq

 While it deduces from $\dv Y=\r(Y)$ the followings
\beq\label{b5.c}
\begin{split}
&\|\na\Delta_jY^2\|_{L^\infty_T(L^2)}\lesssim\|\p_1\Delta_jY\|_{L^\infty_T(L^2)}+\|\Delta_j\r\|_{L^\infty_T(L^2)},\\
&\|\Delta\Delta_jY^2\|_{L^2_T(L^2)}\lesssim\|\p_1\na\Delta_jY\|_{L^2_T(L^2)}+\|\Delta_j\p_2\r\|_{L^2_T(L^2)}.
\end{split}\eeq
And similar to \eqref{2.5}, here we have
 \beno
 \begin{split}
 &\f{1}{2}\bigl(\|\Delta_jY_t\|_{L^2}^2+\|\na\Delta_jY_t\|_{L^2}^2+\|\p_1\Delta_jY\|_{L^2}^2+\|\p_1\na\Delta_jY\|_{L^2}^2
+\f{1}{4}\|\Delta\Delta_jY\|_{L^2}^2\bigr)
\\
&\quad-\f{1}{4}(\Delta_jY_t\ |\ \Delta\Delta_jY)
\sim\|\D_jY_t\|_{L^2}^2+\|\na\Delta_jY_t\|_{L^2}^2+\|\p_1\D_jY\|_{L^2}^2+\|\D\D_jY\|_{L^2}^2.
\end{split}
\eeno Hence by integrating \eqref{b5} over $[0,T]$ and using
\eqref{b5.c}, we obtain \eqref{b14a}.
This completes the proof of Proposition \ref{prop5.1}.
\end{proof}

The next proposition is concerned with the {\it a priori} estimate
to the pressure function $q$ in \eqref{a14}.

\begin{prop}\label{p5}
{\sl Let $(Y,q)$ be a smooth enough solution of \eqref{a14} on
$[0,T].$ Then under the assumption  \eqref{bb14}, for any $s>-1$,
we have \beq\label{q2}\begin{aligned}
 \|\na q\|_{L^1_T(\dot{H}^s)}&\lesssim\|Y\|_{L^\infty_T(\dot{H}^{s+2})}\|\na q\|_{L^1_T(L^2)}+\|Y_t\|_{L^2_T(\dot{B}^1_{2,1})}\|Y_t\|_{L^2_T(\dot{H}^{s+1})}\\
&\quad+\|\p_1Y\|_{L^2_T(\dot{B}^1_{2,1})}\|\p_1Y\|_{L^2_T(\dot{H}^{s+1})}
+\|\p_1Y\|_{L^2_T(\dot{H}^{s+1})}\|\p_1\na
Y\|_{L^2_T(L^2)},
\end{aligned}\eeq
and \beq\label{q3}\begin{aligned}
\|\na q\|_{L^2_T(\dot{H}^s)}&\lesssim\|Y\|_{L^\infty_T(\dot{H}^{s+2})}\|\na q\|_{L^2_T(L^2)}+\|Y_t\|_{L^\infty_T(\dot{B}^1_{2,1})}\|Y_t\|_{L^2_T(\dot{H}^{s+1})}\\
&\quad+\|\na
Y\|_{L^\infty_T(\dot{B}^1_{2,1})}\|\p_1Y\|_{L^2_T(\dot{H}^{s+1})}
+\|Y\|_{L^\infty_T(\dot{H}^{s+2})}\|\p_1\na Y\|_{L^2_T(L^2)}.
\end{aligned}\eeq}
\end{prop}
\begin{proof} We first deduce from
\eqref{bb8} that \beq\label{q1}\begin{aligned}
\|\na q(t)\|_{\dot{H}^s}\leq&\|((\cA_Y-I)\cA_Y^T\na q)(t)\|_{\dot{H}^s}+\|((\cA_Y^T-I)\na q)(t)\|_{\dot{H}^s}\\
&+\|(\p_t\cA_Y
Y_t)(t)\|_{\dot{H}^s}+\|\na(-\D)^{-1}(\na_Y\cdot\p_1^2Y)(t)\|_{\dot{H}^s}.
\end{aligned}\eeq
Thanks to \eqref{a14a}, Lemmas \ref{L5} and \ref{L6}, we see that
for any $s>-1$, \beq\label{b19}\begin{split} \|((\cA_Y-I)\cA_Y^T\na
q)(t)\|_{\dot{H}^s}\lesssim&(1+\|\na Y(t)\|_{\dot{B}^1_{2,1}})
\bigl(\|\na Y(t)\|_{\dot{B}^1_{2,1}}\|\na q(t)\|_{\dot{H}^s}\\
&+\|\na Y(t)\|_{\dot{H}^{s+1}}\|\na q(t)\|_{L^2}\bigr).
\end{split}
\eeq By a similar argument, we obtain that for any $s>-1$,
\beq\label{b18}\begin{split} \|((\cA_Y^T-I)\na
q)(t)\|_{\dot{H}^s}\lesssim& \|\na Y(t)\|_{\dot{B}^1_{2,1}}\|\na
q(t)\|_{\dot{H}^s}+\|\na Y(t)\|_{\dot{H}^{s+1}}\|\na q(t)\|_{L^2},
\end{split}
\eeq
and \beq\label{b21}\begin{split} \|(\p_t\cA_Y
Y_t)(t)\|_{\dot{H}^s}\lesssim
\|Y_t(t)\|_{\dot{B}^1_{2,1}}\|Y_t(t)\|_{\dot{H}^{s+1}}.
\end{split}
\eeq

On the other hand, due to \eqref{bb13a}, we deduce from Lemma
\ref{L5} that for any $s>-1$, \beno\begin{aligned}
\|\na(-\D)^{-1}(\na_Y\cdot\p_1^2Y)(t)\|_{\dot{H}^s}&\lesssim\|(\p_1 Y\p_1\na Y)(t)\|_{\dot{H}^s}\\
&\lesssim\|\p_1 Y(t)\|_{\dot{B}^1_{2,1}}\|\p_1\na Y(t)\|_{\dot{H}^s}
+\|\p_1
Y(t)\|_{\dot{H}^{s+1}}\|\p_1\na Y(t)\|_{L^2}.
\end{aligned}\eeno
This combines  with \eqref{q1}, \eqref{b19}, \eqref{b18},
\eqref{b21} and the assumption \eqref{bb14}  ensures that for any
$s>-1$, \beno\begin{aligned}
\|\na q(t)\|_{\dot{H}^s}&\lesssim\|\na Y(t)\|_{\dot{H}^{s+1}}\|\na q(t)\|_{L^2}+\|Y_t(t)\|_{\dot{B}^1_{2,1}}\|Y_t(t)\|_{\dot{H}^{s+1}}\\
&\quad+\|\p_1Y(t)\|_{\dot{B}^1_{2,1}}\|\p_1\na Y(t)\|_{\dot{H}^s}+\|\p_1Y(t)\|_{\dot{H}^{s+1}}\|\p_1\na
Y(t)\|_{L^2},\quad\text{for}\quad t\in(0,T).
\end{aligned}\eeno
Integrating the above inequality over $(0,T)$ leads to \eqref{q2},
whereas taking its $L^2$ norm with respect to time on $(0,T)$ gives
rise to \eqref{q3}. This proves Proposition \ref{p5}.
\end{proof}

 \subsection{The Estimate to the source terms in \eqref{b14a}} With Propositions \ref{prop5.1} and \ref{p5}, to close the {\it a priori}  estimate for
smooth enough solutions of \eqref{a12}-\eqref{a13}, we need to deal
with the estimate of $ \r$ and $\int_0^T\bigl(\Delta_j\vv f\ |\
\Delta_jY_t-\f{1}{4}\Delta\Delta_jY-\D\D_jY_t\bigr)\,dt$ for $\r,
\vv f$ given by \eqref{a13}. We first present the related estimates
for $\r,$ which is a direct consequence of Lemma \ref{L5} and
\eqref{a13}.

\begin{lem}\label{lem5.2}
{\sl Let $Y$ be a sufficiently  smooth function on
$[0,T]\times\R^2.$ Then for any $s>-1,$
 $\r(Y)=\p_1Y^2\p_2Y^1-\p_1Y^1\p_2Y^2$ satisfies
 \beno
\begin{split}
&\|\Delta_j\r\|_{L^\infty_T(L^2)} \lesssim c_j2^{-js}\bigl(\|\p_2
Y\|_{L^\infty_T(\dot B^1_{2,1})}\|\p_1
Y\|_{L^\infty_T(\dot H^{s})}+\|\p_2 Y\|_{L^\infty_T(\dot
H^{s+1})}\|\p_1Y\|_{L^\infty_T(L^2)}\bigr),\\
&\|\na\Delta_j\r\|_{L^2_T(L^2)} \lesssim c_j2^{-js}\bigl(\|\p_2
Y\|_{L^\infty_T(\dot B^1_{2,1})}\|\p_1 Y\|_{L^2_T(\dot
H^{s+1})}+\|\p_2 Y\|_{L^\infty_T(\dot
H^{s+1})}\|\p_1Y\|_{L^2_T(\dot{B}^1_{2,1})}\bigr).
\end{split}
\eeno}
\end{lem}

\begin{lem}\label{lem5.3}
{\sl Let $(Y,q)$ be a smooth enough solution of
\eqref{a12}-\eqref{a13}. For $s>-1$, we assume, in addition,  that
\beq\label{A5} \|\na
Y\|_{L^\infty_T(\dot{B}^1_{2,1})}+\|Y\|_{L^\infty_T(\dot{H}^{s+2})}\leq
1, \eeq then  one has
 \beq\label{b20}\begin{aligned}
\bigl|&\int_0^T\bigl(\Delta_j\vv f\ |\
\Delta_jY_t\bigr)\,dt\bigr|\lesssim c_j^22^{-2js} \Bigl\{\bigl(\|\na
q\|_{L^1_T(\dot{H}^s)}+\|\na
q\|_{L^1_T(L^2)}\bigr)\|Y_t\|_{\wt{L}^\infty_T(\dot{H}^s)}\\
&\qquad +\bigl(\|\na
Y\|_{L^\infty_T(\dot{B}^1_{2,1})}+\|Y_t\|_{L^2_T(\dH^1)}\bigr)\bigl(\|Y\|_{L^\infty_T(\dot{H}^{s+2})}^2+
\|Y_t\|_{L^2_T(\dot{H}^{s+1})}^2\bigr)\Bigr\}.
\end{aligned}\eeq}
\end{lem}
\begin{proof} For $\tilde{\vv f}$ given by
 \eqref{b17}, it follows from Lemma \ref{L5} that \beq\label{b20ad}
\|\tilde{\vv f}\|_{L^1_T(\dot{H}^s)}\lesssim(1+\|\na
Y\|_{L^\infty_T(\dot{B}^1_{2,1})})\|\na q\|_{L^1_T(\dot{H}^s)}+\|\na
Y\|_{L^\infty_T(\dot{H}^{s+1})}\|\na q\|_{L^1_T(L^2)},\eeq which
along with \eqref{A5} implies  for any $s>-1$ \beq\label{c3}
\bigl|\int_0^T\bigl(\Delta_j\tilde{\vv f}\ |\
\Delta_jY_t\bigr)\,dt\bigr|\lesssim c_j^22^{-2js} \bigl(\|\na
q\|_{L^1_T(\dot{H}^s)}+\|\na
q\|_{L^1_T(L^2)}\bigr)\|Y_t\|_{\wt{L}^\infty_T(\dot{H}^s)}. \eeq

 While via \eqref{b15}, we get by using
integration by parts that \beno \int_0^T(\Delta_j\bar{\vv f}\ |\
\Delta_jY_t)\,dt=-\sum_{i=1}^2\int_0^T(\Delta_j\vv F_i\ |\
\Delta_j\p_iY_t)_{L^2}\,dt, \eeno with $\vv F_1, \vv F_2$  given by
\eqref{b16}. For any $s>-1,$ applying Lemma \ref{L5}, one has \beno
\|\p_2Y^1\p_1Y_t\|_{L^2_T(\dot{H}^s)}\lesssim
\|\p_2Y^1\|_{L^\infty_T(\dot{B}^1_{2,1})}\|\p_1Y_t\|_{L^2_T(\dot{H}^s)}
+\|\p_2Y^1\|_{L^\infty_T(\dot{H}^{s+1})}\|\p_1Y_t\|_{L^2_T(L^2)}.
\eeno Applying Lemma \ref{L5} twice and Lemma \ref{L6}, it leads to
 \beno
 \begin{split}
 \|\p_2&Y^2\p_1Y^2\p_2Y_t\|_{L^2_T(\dot{H}^s)}\\
 \lesssim & \|\p_2Y^2\p_1Y^2\|_{L^\infty_T(\dot{B}^1_{2,1})}\|\p_2Y_t\|_{L^2_T(\dot{H}^s)}
 +\|\p_2Y^2\p_1Y^2\|_{L^\infty_T(\dot{H}^{s+1})}\|\p_2Y_t\|_{L^2_T(L^2)}\\
 \lesssim &
 \|\na Y^2\|_{L^\infty_T(\dot{B}^1_{2,1})}\bigl(\|\na Y^2\|_{L^\infty_T(\dot{B}^1_{2,1})}\|\p_2Y_t\|_{L^2_T(\dot{H}^s)}+
 \|\na Y^2\|_{L^\infty_T(\dot{H}^{s+1})}
 \|\p_2Y_t\|_{L^2_T(L^2)}\bigr).
 \end{split}
 \eeno
Similar estimates hold for the other terms in $\vv F_1,\vv F_2$
given by \eqref{b16}. Therefore, under the assumption \eqref{A5}, we
obtain \beno
\begin{split}
\bigl|\int_0^T(\Delta_j\bar{\vv f}\ |\
\Delta_jY_t)\,dt\bigr|\lesssim & c_j^22^{-2js} \bigl\{\bigl(\|\na
Y\|_{L^\infty_T(\dot{B}^1_{2,1})}\|
Y_t\|_{L^2_T(\dot{H}^{s+1})}\\
&+\|Y\|_{L^\infty_T(\dot{H}^{s+2})}\|Y_t\|_{L^2_T(\dH^1)}\bigr)
\|Y_t\|_{L^2_T(\dot{H}^{s+1})}\bigr\}.
\end{split}
\eeno This together with \eqref{bb17a} and \eqref{c3} implies
\eqref{b20}. We complete the proof of this Lemma.
\end{proof}

\begin{lem}\label{lem5.5}
{\sl Let $(Y,q)$ be a smooth enough solution of
\eqref{a12}-\eqref{a13}. For $s>-1$, we assume  also that \eqref{A4}
and \eqref{A5},  then one has \beq\label{b31}\begin{aligned}
&\bigl|\int_0^T(\Delta_j\vv f\ |\ \Delta\Delta_j Y_t)\,dt\bigr|
\lesssim
c_j^22^{-2js}\Bigl\{\bigl(\|\na q\|_{L^2_T(\dot{H}^s)}+\|\na q\|_{L^2_T(L^2)}\bigr)\|Y_t\|_{L^2_T(\dot{H}^{s+2})}\\
&\qquad+\bigl(\|\na Y\|_{L^\infty_T(\dot{B}^1_{2,1})}+\|\na
Y_t\|_{L^2_T(\dot{B}^1_{2,1})}\bigr)\bigl(\|Y\|_{L^\infty_T(\dot{H}^{s+2})}^2+
\| Y_t\|_{L^2_T(\dot{H}^{s+2})}^2\bigr)\Bigr\},
\end{aligned}\eeq
and \beq\label{b30}\begin{aligned} &\bigl|\int_0^T(\Delta_j\vv f\ |\
\Delta\Delta_j Y)\,dt\bigr| \lesssim c_j^22^{-2js}\Bigl\{\bigl(\|\na
q\|_{L^1_T(\dot{H}^s)}+\|\na q\|_{L^1_T(L^2)}\bigr)\|Y\|_{\wt{L}^\infty_T(\dot{H}^{s+2})}\\
&\quad +\bigl(\|\p_1 Y_t^1\|_{L^1_T(L^\infty)}+\|\p_1\na
Y\|_{L^2_T(\dot{B}^1_{2,1})}+\|\p_1
Y\|_{L^2_T(\dot{B}^1_{2,1})}+\|\na
Y\|_{L^\infty_T(\dot{B}^1_{2,1})}\\
&\qquad+\|\na
Y_t\|_{L^2_T(\dot{B}^1_{2,1})}\bigr)
\bigl(\|\p_1Y\|_{L^2_T(\dot{B}^1_{2,1})}^2+\|Y^2\|_{L^2_T(\dot{H}^{s+2})}^2+\|Y\|_{\wt{L}^\infty_T(\dot{H}^{s+2})}^2
\\
&\qquad+\|\p_1Y\|_{L^2_T(\dot{H}^{s+1})}^2
+\|Y_t\|_{L^2_T(\dot{H}^{s+1})}^2+\|Y_t\|_{L^2_T(\dot{H}^{s+2})}^2\bigr) \Bigr\}.
\end{aligned}\eeq
}
\end{lem}
\begin{proof}
Thanks to \eqref{b16} and \eqref{b17}, we get by using Lemma
\ref{L5},   \eqref{A4} and \eqref{A5} that for any $s>-1$,
\beq\label{c17}\begin{aligned}
\|\bar{\vv f}\|_{L^2_T(\dot{H}^s)}&\lesssim\sum_{i=1}^2\|\vv F_i\|_{L^2_T(\dot{H}^{s+1})}\\
&\lesssim\|\na Y\|_{L^\infty_T(\dot{B}^1_{2,1})}\|\na
Y_t\|_{L^2_T(\dot{H}^{s+1})} +\|\na
Y\|_{L^\infty_T(\dot{H}^{s+1})}\|\na Y_t\|_{L^2_T(\dot{B}^1_{2,1})},
\end{aligned}\eeq
and
\beq\label{c18}\begin{aligned}
\|\tilde{\vv
f}\|_{L^2_T(\dot{H}^s)}&\lesssim(1+\|\na
Y\|_{L^\infty_T(\dot{B}^1_{2,1})})\|\na
q\|_{L^2_T(\dot{H}^s)}+\|\na Y\|_{L^\infty_T(\dot{H}^{s+1})}\|\na q\|_{L^2_T(L^2)}\\
&\lesssim\|\na q\|_{L^2_T(\dot{H}^s)}+\|\na q\|_{L^2_T(L^2)},
\end{aligned}\eeq
which together with \eqref{bb17a} implies \eqref{b31}.

On the other hand, we infer from \eqref{A4}, \eqref{A5},
\eqref{b20ad}, \eqref{c17} and  \eqref{c18} that for any $s>-1$
\beq\label{b24a}\begin{aligned} &\bigl|\int_0^T(\Delta_j\tilde{\vv
f}\ |\ \Delta\Delta_j Y)\,dt\bigr|\lesssim c_j^22^{-2js}\bigl(\|\na
q\|_{L^1_T(\dot{H}^s)}+\|\na q\|_{L^1_T(L^2)}\bigr)\|\Delta
Y\|_{\wt{L}^\infty_T(\dot{H}^s)},
\end{aligned}\eeq
 \beq\label{b24}\begin{aligned}
&\bigl|\int_0^T(\Delta_j\bar{f}^2\ |\ \Delta\Delta_j
Y^2)\,dt\bigr|\lesssim c_j^22^{-2js}
\bigl(\|\na Y\|_{L^\infty_T(\dot{B}^1_{2,1})}\|\na Y_t\|_{L^2_T(\dot{H}^{s+1})}\\
&\qquad+\|\na Y\|_{L^\infty_T(\dot{H}^{s+1})}\|\na
Y_t\|_{L^2_T(\dot{B}^1_{2,1})}\bigr)\|\Delta Y^2\|_{L^2_T(\dot{H}^s)}.
\end{aligned}\eeq

Whereas via \eqref{c17} and using integration by parts, we can
obtain for $s>-1$ the following
 \beq\label{b25}\begin{aligned}
\bigl|\int_0^T(\Delta_j\p_1F_1^1\ |\ \Delta\Delta_j Y^1)\,dt\bigr|=&\bigl|\int_0^T(\Delta_j\na F_1^1\ |\ \p_1\na\Delta_j Y^1)\,dt\bigr|\\
\lesssim& c_j^22^{-2js}
\bigl(\|\na Y\|_{L^\infty_T(\dot{B}^1_{2,1})}\|\na Y_t\|_{L^2_T(\dot{H}^{s+1})}\\
& +\|\na Y\|_{L^\infty_T(\dot{H}^{s+1})}\|\na
Y_t\|_{L^2_T(\dot{B}^1_{2,1})}\bigr)\|\p_1\na
Y^1\|_{L^2_T(\dot{H}^s)}.
\end{aligned}\eeq
Because of \eqref{b15}, \eqref{b24a} to \eqref{b25}, to complete the
proof of \eqref{b30}, we only need to deal with the term
$\int_0^T(\Delta_j\p_2F_2^1\ |\ \Delta\Delta_j Y^1)\,dt$. For this
purpose, we split $F_2^1$ as $G^1+G^2$ with \beq\label{b25af}
 G^1=(2\p_1Y^1+|\p_1Y|^2)\p_2Y_t^1-(1+\p_2Y^2)\p_1Y^2\p_1Y_t^1,
\quad  G^2=-(1+\p_1Y^1)\p_2Y^1\p_1Y_t^1. \eeq It follows from Lemmas
\ref{L5} and \ref{L6} and \eqref{A4} that \beno
\begin{split}
\|\p_2G^1\|_{L_T^1(\dot{H}^s)} \lesssim&\|\p_1
Y\|_{L^2_T(\dot{B}^1_{2,1})}\|\na
Y_t\|_{L^2_T(\dot{H}^{s+1})}\\
& +(\|\p_1 Y\|_{L^2_T(\dot{H}^{s+1})}+\|\na
Y^2\|_{L^2_T(\dot{H}^{s+1})})\|\na Y_t\|_{L^2_T(\dot{B}^1_{2,1})},
\end{split} \eeno for $s>-1,$
this along with the fact that
$\|\p_2G^1\|_{\wt{L}_T^1(\dot{H}^s)}\lesssim\|\p_2G^1\|_{L_T^1(\dot{H}^s)}$
implies \beq\label{b27}\begin{aligned} |\int_0^T(\Delta_j&\p_2G^1\
|\ \Delta\Delta_j Y^1)_{L^2}\,dt| \lesssim c_j^22^{-2js}\Bigl(\|\p_1
Y\|_{L^2_T(\dot{B}^1_{2,1})}\|\na
Y_t\|_{L^2_T(\dot{H}^{s+1})}\\
& +(\|\p_1 Y\|_{L^2_T(\dot{H}^{s+1})}+\|\na
Y^2\|_{L^2_T(\dot{H}^{s+1})})\|\na
Y_t\|_{L^2_T(\dot{B}^1_{2,1})}\Bigr) \|\Delta
Y^1\|_{\wt{L}^\infty_T(\dot{H}^s)},
\end{aligned}\eeq for any $s>-1$.

To handle $G^2$ in \eqref{b25af}, we first do a Bony's decomposition
\eqref{bb5} so that \beno
\p_2Y^1\p_1Y_t^1=T_{\p_2Y^1}\p_1Y_t^1+T_{\p_1Y_t^1}{\p_2Y^1}+R(\p_2Y^1,\p_1Y_t^1).
\eeno Using integration by parts, we have
 \beno
 \begin{split}
&\int_0^T(\Delta_j(T_{\p_2Y^1}\p_1Y_t^1)\ |\ \Delta\Delta_j\p_2Y^1)\,dt\\
&=-\int_0^T(\Delta_j(T_{\p_1\p_2Y^1}Y_t^1)\ |\
\Delta\Delta_j\p_2Y^1)\,dt-\int_0^T(\Delta_j(T_{\p_2Y^1}Y_t^1)\ |\
\Delta\Delta_j\p_1\p_2Y^1)\,dt, \end{split} \eeno from which and
Lemma \ref{le2.1}, we conclude \beno\begin{aligned}
&\bigl|\int_0^T(\Delta_j(T_{\p_2Y^1}\p_1Y_t^1)\ |\ \Delta\Delta_j\p_2Y^1)\,dt\bigr|\\
&\lesssim\sum_{|j'-j|\leq4}\bigl(\|S_{j'-1}\p_1\p_2Y^1\|_{L^2_T(L^\infty)}\|\Delta_{j'}Y_t^1\|_{L^2_T(L^2)}
\|\Delta\Delta_j\p_2Y^1\|_{L^\infty_T(L^2)}\\
&\qquad+\|S_{j'-1}\p_2Y^1\|_{L^\infty_T(L^\infty)}\|\Delta_{j'}Y_t^1\|_{L^2_T(L^2)}
\|\Delta\Delta_j\p_1\p_2Y^1\|_{L^2_T(L^2)}\bigr)\\
&\lesssim
c_j^22^{-2js}\bigl(\|\p_1\p_2Y^1\|_{L^2_T(\dot{B}^1_{2,1})}\|Y_t^1\|_{L^2_T(\dot{H}^{s+1})}
\|\Delta Y^1\|_{\wt{L}^\infty_T(\dot{H}^s)}\\
&\qquad+\|\p_2Y^1\|_{L^\infty_T(\dot{B}^1_{2,1})}\|Y_t^1\|_{L^2_T(\dot{H}^{s+2})}
\|\p_1\na Y^1\|_{L^2_T(\dot{H}^s)}\bigr).
\end{aligned}\eeno
While it is easy to verify that for $s>-1$, \beno\begin{aligned}
\bigl|&\int_0^T(\Delta_j(R(\p_2Y^1,\p_1Y_t^1))\ |\ \Delta\Delta_j\p_2Y^1)\,dt\bigr|\\
&\lesssim\sum_{j'\geq
j-N_0}\|\Delta_{j'}\p_2Y^1\|_{L^\infty_T(L^2)}\|\wt\Delta_{j'}\p_1
Y_t^1\|_{L^1_T(L^\infty)}
\|\Delta\Delta_j\p_2Y^1\|_{L^\infty_T(L^2)}\\
&\lesssim c_j^22^{-2js}\|\p_1
Y_t^1\|_{L^1_T(L^\infty)}\|\p_2Y^1\|_{\wt{L}^\infty_T(\dot{H}^{s+1})}^2.
\end{aligned}\eeno
The same estimate holds for
$\int_0^T(\Delta_j(T_{\p_1Y_t^1}\p_2Y^1))\ |\
\Delta\Delta_j\p_2Y^1)\,dt$. Consequently we obtain for any $s>-1$
that
 \beq\label{b27ad}
\begin{split}
\bigl|\int_0^T&(\Delta_j(\p_2Y^1\p_1Y_t^1)\ |\
\Delta\Delta_j\p_2Y^1)_{L^2}\,dt\bigr|\lesssim
c_j^22^{-2js}\Bigl\{\|\p_1
Y_t^1\|_{L^1_T(L^\infty)}\|\p_2Y^1\|_{\wt{L}^\infty_T(\dot{H}^{s+1})}^2\\
&\qquad+\|\p_1\p_2Y^1\|_{L^2_T(\dot{B}^1_{2,1})}\|
Y_t^1\|_{L^2_T(\dot{H}^{s+1})} \|
Y^1\|_{\wt{L}^\infty_T(\dot{H}^{s+2})}\\
&\qquad+\|\p_2Y^1\|_{L^\infty_T(\dot{B}^1_{2,1})}\|
Y_t^1\|_{L^2_T(\dot{H}^{s+2})} \|\p_1
Y^1\|_{L^2_T(\dot{H}^{s+1})}\Bigr\}.
\end{split}
\eeq

Finally, under the assumption \eqref{A5}, we deduce from Lemmas
\ref{L5} and \ref{L6} the following
 \beno
\begin{split}
\bigl|\int_0^T&(\Delta_j(\p_1Y^1\p_2Y^1\p_1Y_t^1)\ |\
\Delta\Delta_j\p_2Y^1)\,dt\bigr|\\
\lesssim &c_j^22^{-2js}\Bigl(\|\p_1Y^1\|_{L^2_T(\dot
B^1_{2,1})}\|\p_1Y_t^1\|_{L^2_T(\dot
H^{s+1})}+(\|\p_1Y^1\|_{L^2_T(\dot
H^{s+1})}\\
&\qquad+\|\p_1Y^1\|_{L^2_T(\dot
B^1_{2,1})})\|\p_1Y^1_t\|_{L^2_T(\dot
B^1_{2,1})}\Bigr)\|\p_2Y^1\|_{\wt{L}^\infty_T(\dot
H^{s+1})}\quad\mbox{for}\quad s>-1,
\end{split}
\eeno from which and \eqref{b25af}, \eqref{b27ad}, we arrive at
\beq\label{b29}\begin{aligned} |\int_0^T&(\Delta_j\p_2G^2\ |\
\Delta\Delta_jY^1)\,dt|\lesssim c_j^22^{-2js} \bigl(\|\p_1
Y_t^1\|_{L^1_T(L^\infty)}+\|\p_1\na
Y\|_{L^2_T(\dot{B}^1_{2,1})}\\
&+\|\p_1
Y\|_{L^2_T(\dot{B}^1_{2,1})}+\|\na Y\|_{L^\infty_T(\dot{B}^1_{2,1})}+\|\na
Y_t\|_{L^2_T(\dot{B}^1_{2,1})}\bigr)\bigl(\|\p_1Y\|_{L^2_T(\dot{B}^1_{2,1})}^2\\
&+\|Y\|_{\wt{L}^\infty_T(\dot{H}^{s+2})}^2+\|\p_1Y\|_{L^2_T(\dot{H}^{s+1})}^2
+\|Y_t\|_{L^2_T(\dot{H}^{s+1})}^2+\|Y_t\|_{L^2_T(\dot{H}^{s+2})}^2\bigr),
\end{aligned}\eeq
for any $s>-1$ and with the assumption \eqref{A5}.

Using \eqref{b15}, and by summing up  \eqref{b24a}, \eqref{b24},
\eqref{b25}, \eqref{b27} and \eqref{b29},  we conclude the proof of
\eqref{b30}. This finishes the proof of Lemma \ref{lem5.5}.
\end{proof}

\subsection{The proof of Theorem \ref{T}} \label{subsect4.3} The proof of
Theorem \ref{T} is  based on the following proposition:

\begin{prop}\label{prop5.2}
{\sl Let  $s_1>1$ and $ s_2\in (-1,-\f12).$ Let $(Y,q)$ be a smooth
enough solution of \eqref{a12}-\eqref{a13} on $[0,T].$ We denote
$\mathcal{E}_T^{s_1,s_2}(Y,q)$ as \beq\label{c22ag}
\mathcal{E}_T^{s_1,s_2}(Y,q)\eqdefa
E_T^{s_1}(Y,q)+E_T^{s_2}(Y,q)\quad
\text{and}\quad\mathcal{E}_0^{s_1,s_2}\eqdefa E^{s_1}_0+E^{s_2}_0,
\eeq with  \beno\begin{aligned}
E_T^s(Y,q)\eqdefa&\|Y_t\|_{\wt{L}^\infty_T(\dot{H}^s)}^2+\|Y_t\|_{\wt{L}^\infty_T(\dot{H}^{s+1})}^2+\|\p_1Y\|_{\wt{L}^\infty_T(\dot{H}^s)}^2
+\|Y^2\|_{\wt{L}^\infty_T(\dot{H}^{s+1})}^2\\
&+\|Y\|_{\wt{L}^\infty_T(\dot{H}^{s+2})}^2+\|Y_t\|_{L^2_T(\dot{H}^{s+1})}^2+\|Y_t\|_{L^2_T(\dot{H}^{s+2})}^2+\|\p_1Y\|_{L^2_T(\dot{H}^{s+1})}^2
\\
&+\|Y^2\|_{L^2_T(\dot{H}^{s+2})}^2+\|\na
q\|_{L^2_T(\dot{H}^s)}^2+\|\na q\|_{L^1_T(\dot{H}^s)}^2,
\end{aligned}\eeno
and  \beno
E_0^s\eqdefa\|Y_1\|_{\dot{H}^s}^2+\|Y_1\|_{\dot{H}^{s+1}}^2+\|\p_1Y_0\|_{\dot{H}^s}^2+\|Y_0\|_{\dot{H}^{s+2}}^2.
\eeno Then under the assumption  \eqref{bb14} and
\beq\label{A3}\begin{aligned} \|\na
Y\|_{\wt{L}^\infty_T(\dot{B}^1_{2,1})}+\|\na
Y\|_{\wt{L}^\infty_T(\dot{B}^2_{2,1})}+\|Y\|_{\wt{L}^\infty_T(\dot{H}^{s_1+2})}+\|Y\|_{\wt{L}^\infty_T(\dot{H}^{s_2+2})}
\leq 1,
\end{aligned}\eeq
we have \beq\label{c22} \mathcal{E}_T^{s_1,s_2}(Y,q)\leq
C_1\mathcal{E}_0^{s_1,s_2}+C_1\bigl((\mathcal{E}^{s_1,s_2}_0)^{1/2}+\mathcal{E}_T^{s_1,s_2}(Y,q)^{1/2}+\mathcal{E}_T^{s_1,s_2}(Y,q)\bigr)\mathcal{E}_T^{s_1,s_2}(Y,q),
\eeq for some positive constant $C_1.$}
\end{prop}
\begin{proof} Under the assumptions \eqref{bb14} and
\eqref{A3}, for $s=s_1$ and $s=s_2$, we can deduce from Propositions
\ref{prop5.1} and \ref{p5} that \beno
\begin{split}
E_T^s(&Y,q)\lesssim
E_0^s+\|\r\|_{\wt{L}^\infty_T(\dH^s)}^2+\|\r\|_{L^2_T(\dH^{s+1})}^2+\|Y\|_{L^\infty_T(\dH^{s+2})}^2(\|\na
q\|_{L^1_T(L^2)}^2+\|\na q\|_{L^2_T(L^2)}^2)\\
&+\bigl(\|Y_t\|_{L^2_T(\dot B^1_{2,1})}^2+\|\p_1 Y\|_{L^2_T(\dot B^1_{2,1})}^2+\|\p_1\na
Y\|_{L^2_T(L^2)}^2+\|Y_t\|_{L^\infty_T(\dot
B^1_{2,1})}^2\\
&+\|\na Y\|_{L^\infty_T(\dot
B^1_{2,1})}^2\bigr)E_T^s(Y,q)+\sum_{j\in\Z}2^{2js}\bigl|\int_0^T\bigl(\Delta_j\vv
f\ |\ \Delta_jY_t-\f{1}{4}\Delta\Delta_jY-\D\D_jY_t\bigr)\,dt\bigr|.
\end{split}
\eeno
Thanks to Lemma \ref{lem5.2}, there holds
\beno
\|\r\|_{\wt{L}^\infty_T(\dH^s)}^2+\|\r\|_{L^2_T(\dH^{s+1})}^2\lesssim\bigl(\|\na Y\|_{L^\infty_T(\dot
B^1_{2,1})}^2+\|\p_1 Y\|_{L^\infty_T(L^2)}^2+\|\p_1 Y\|_{L^2_T(\dot
B^1_{2,1})}^2\bigr)E_T^s(Y,q).
\eeno
While it follows from Lemmas \ref{lem5.3} and \ref{lem5.5}
that \beno
\begin{split}
&\sum_{j\in\Z}2^{2js}\bigl|\int_0^T\bigl(\Delta_j\vv f\ |\
\Delta_jY_t-\f{1}{4}\Delta\Delta_jY-\D\D_jY_t\bigr)\,dt\bigr|\\
&\lesssim
\bigl(\|Y\|_{\wt{L}^\infty_T(\dH^{s+2})}+\|Y_t\|_{\wt{L}^\infty_T(\dH^s)}+\|Y_t\|_{L^2_T(\dH^{s+2})}\bigr)\bigl(\|\na
q\|_{L^1_T(\dH^s)}+\|\na q\|_{L^1_T(L^2)}+\|\na
q\|_{L^2_T(\dH^s)}\\
&\quad+\|\na q\|_{L^2_T(L^2)}\bigr)+\bigl(\|\p_1 Y_t^1\|_{L^1_T(L^\infty)}+\|\p_1\na
Y\|_{L^2_T(\dot{B}^1_{2,1})}+\|\p_1
Y\|_{L^2_T(\dot{B}^1_{2,1})}+\|\na
Y\|_{L^\infty_T(\dot{B}^1_{2,1})}\\
&\quad+\|Y_t\|_{L^2_T(\dH^1)}+\|\na
Y_t\|_{L^2_T(\dot{B}^1_{2,1})}\bigr)\bigl(\|\p_1Y\|_{L^2_T(\dot{B}^1_{2,1})}^2+E^s_T(Y,q)\bigr)\quad
\mbox{for}\quad s=s_1, s_2.
\end{split}
\eeno 
As a consequence and with Proposition \ref{p5}, we obtain for $s=s_1, s_2,$ that
\beq\label{c20}\begin{aligned}
E_T^{s}(Y,q)\lesssim& E_0^{s}+\bigl((E_T^{0}(Y,q))^{1/2}+(E_T^{s}(Y,q))^{1/2}+E_T^{0}(Y,q)+E_T^{s}(Y,q)\\
&+\|\p_1Y^1_t\|_{L^1_T(L^\infty)}+\|\na Y\|_{\wt{L}^\infty_T(\dot{B}^1_{2,1})}+\|\p_1 Y\|_{\wt{L}^2_T(\dot{B}^1_{2,1})}+\|\p_1\na Y\|_{\wt{L}^2_T(\dot{B}^1_{2,1})}\\
&+\|\na
Y_t\|_{\wt{L}^2_T(\dot{B}^1_{2,1})}+\|\p_1 Y\|_{\wt{L}^2_T(\dot{B}^1_{2,1})}^2+\|Y_t\|_{\wt{L}^\infty_T(\dot{B}^1_{2,1})}^2+
\|Y_t\|_{\wt{L}^2_T(\dot{B}^1_{2,1})}^2\bigr)\\
&\times\bigl(E_T^{0}(Y,q)+E_T^{s}(Y,q)
+\|\p_1 Y\|_{\wt{L}^2_T(\dot{B}^1_{2,1})}^2\bigr).
\end{aligned}\eeq
Notice from \eqref{bb0}, one can easily deduce  \beno
&&\|\p_1Y\|_{\wt{L}^2_T(\dot{B}^1_{2,1})}\lesssim \|\p_1Y\|_{L^2_T(\dot{H}^{s_2+1})}+\|\p_1 Y\|_{L^2_T(\dot{H}^{s_1+1})},\\
&&\text{and}\quad E_T^{0}(Y,q)\lesssim
E_T^{s_2}(Y,q)+E_T^{s_1}(Y,q), \eeno for any $s_1>1,\,
s_2\in(-1,-\f12)$. Thus by taking $s=s_1, s_2$ in \eqref{c20} and
summing up the resulting inequality yields \beq\label{c20ad}
\mathcal{E}_T^{s_1,s_2}(Y,q)\lesssim\mathcal{E}_0^{s_1,s_2}+\bigl(\|\p_1Y^1_t\|_{L^1_T(L^\infty)}
+\mathcal{E}_T^{s_1,s_2}(Y,q)^{1/2}+\mathcal{E}_T^{s_1,s_2}(Y,q)\bigr)\mathcal{E}_T^{s_1,s_2}(Y,q).
\eeq On the other hand,  it follows from Proposition \ref{p4} and
\eqref{bb0} that \beq\label{c20ag}
\|\p_1Y^1_t\|_{L^1_T(L^\infty)}\lesssim
(\mathcal{E}^{s_1,s_2}_0)^{1/2} +\mathcal{E}_T^{s_1,s_2}(Y,q). \eeq
Substituting \eqref{c20ag} into \eqref{c20ad}, one concludes the
proof of \eqref{c22}.
\end{proof}

\begin{rmk}
We should mention that the restriction for $s_2\in (-1,-\f12)$ in
Theorem \ref{T} is due to the following fact: \beno
\begin{split}
\|Y_t\|_{L^2_T(\dot{B}^{\f12}_{2,1})}+\|\p_1Y\|_{L^2_T(\dot{B}^{\f12}_{2,1})}\lesssim&
\|Y_t\|_{L^2_T(\dot{H}^{s_2+1})}+\|Y_t\|_{L^2_T(\dot{H}^{s_1+1})}\\
&+\|\p_1Y\|_{L^2_T(\dot{H}^{s_2+1})}+\|\p_1Y\|_{L^2_T(\dot{H}^{s_1+1})},
\end{split}
\eeno which has been used in the proof of \eqref{c20ag}.
\end{rmk}

Now we are in position to complete the proof of Theorem \ref{T}.

\begin{proof}[Proof of Theorem \ref{T}] Given initial data  $(Y_0, Y_1)$ satisfying the assumptions
listed in Theorem \ref{T}, we deduce by a standard argument that
\eqref{a12}-\eqref{a13} has a unique solution $(Y,q)$ on $[0,T],$
which satisfies \eqref{A1wq} on $[0,T]$. Let $T^*$ be the largest
possible time so that \eqref{A1wq} holds. Then to completes the
proof of Theorem \ref{T}, we only need to show that $T^\ast=\infty$
and there holds \eqref{M1} provided that \eqref{A1a} and \eqref{A1}
hold. Toward this, we denote
 \beq\label{c26} \begin{split}
&\bar{T}\eqdefa \max\bigl\{\ T<T^\ast:\ \
\mathcal{E}_T^{s_1,s_2}(Y,q)\leq \eta_0^2\ \ \ \bigr\},
\end{split}
\eeq for $\eta_0$ so small that $C_1(2\eta_0+\eta_0^2)\leq\f12,$ and
\beq\label{c26ad}
\begin{split}
&\|\na Y\|_{\wt{L}^\infty_T(\dot{B}^1_{2,1})}+\|\na Y\|_{\wt{L}^\infty_T(\dot{B}^2_{2,1})}+\|Y\|_{\wt{L}^\infty_T(\dot{H}^{s_1+2})}+\|Y\|_{\wt{L}^\infty_T(\dot{H}^{s_2+2})}\\
&
\leq C_2\mathcal{E}_T^{s_1,s_2}(Y,q)^{\f12}\leq C_2\eta_0\leq 1,
\end{split}
\eeq for the same $C_1$  as that in \eqref{c22}.

 We shall prove that $\bar{T}=\infty$ provided that
$\e_0$ is sufficiently small in \eqref{A1}.  Otherwise, by
\eqref{c26ad}, we can apply Proposition \ref{prop5.2} to conclude
that
 \beq\label{c26ah}
\mathcal{E}_{\bar T}^{s_1,s_2}(Y,q)\leq 2C_1\mathcal{E}_0^{s_1,s_2}.
\eeq In particular, if we take $\e_0$ so small that
$2C_1\e_0^2\leq\f12\eta_0^2,$ \eqref{c26ah} contradicts with
\eqref{c26} if $\bar{T}<\infty.$ This in turn shows that
$\bar{T}=T^\ast=\infty,$ and \eqref{M1} is valid. This completes the
proof of Theorem \ref{T}.
\end{proof}

\medskip
\section{The proof of Theorem \ref{th1} and Theorem \ref{th2}}\label{sect6}
\setcounter{equation}{0}

With Theorem \ref{T} and  Lemma \ref{funct} in the Appendix
\ref{appendc} in hand, we can now present the proof of Theorem
\ref{th1}.

\begin{proof}[Proof of Theorem \ref{th1}] Given $\psi_0,
\tilde\psi_0$ satisfying the assumptions of Theorem \ref{th1}, we
get by using Lemma \ref{initial} that there exists a vector-valued
function $Y_0(y)=(Y_0^1(y),Y_0^2(y))^T$ so that \beno \begin{aligned}
&\|Y_0\|_{\dot{H}^{s_1+2}\cap\dot{H}^{s_2+2}}
+\|\p_1Y_0\|_{\dot{H}^{s_2}}\\
& \leq
C(\|\na\psi_0\|_{\dH^{s_1+1}\cap\dH^{s_2}},\|\na\tilde\psi_0\|_{\dH^{s_1+1}\cap\dH^{s_2+1}}
)\bigl(\|\na
\psi_0\|_{\dot{H}^{s_1+1}\cap\dot{H}^{s_2}}+\|\na\tilde\psi_0\|_{\dot{H}^{s_1+1}\cap\dot{H}^{s_2+1}}\bigr),
\end{aligned}\eeno
and
 $X_0(y)\eqdefa
I+Y_0(y)$ satisfies \beno \na_y Y_0(y)=\begin{pmatrix}
\p_{x_2}\psi_0& \p_{x_2}\tilde\psi_0\\
-\p_{x_1}\psi_0&-\p_{x_1}\tilde\psi_0
\end{pmatrix}\circ X_0(y). \eeno
Let $Y_1(y)\eqdefa\vv u_0(X_0(y)).$ Applying Lemma \ref{funct} (vi)
and (ii) with $\Phi=X^{-1}_0(x)$ gives \beno
\begin{split}
&\|Y_1\|_{\dot{H}^{s_1+1}}\leq
C(\|\na_x\Psi_0\|_{L^\infty})(1+\|\D_x\Psi_0\|_{H^{s_1-1}})\|\na_x\vv
u_0\|_{H^{s_1}},\\
&\|Y_1\|_{\dot{H}^{s_2}}\leq
C(\|\na_x\Psi_0\|_{L^\infty})\bigl(\|\vv u_0
\|_{\dot{H}^{s_2}}+\|\na_x\Psi_0\|_{\dH^{s_2+1}}\|\vv
u_0\|_{L^2}\bigr),
\end{split}
 \eeno where
$\Psi_0\eqdefa(\psi_0,\tilde\psi_0)^T.$ Therefore, thanks to
\eqref{1.5}, we conclude that \beq\label{cc7}\begin{aligned}
&\|Y_0\|_{\dot{H}^{s_1+2}\cap\dot{H}^{s_2+2}}
+\|\p_1Y_0\|_{\dot{H}^{s_2}}+\|Y_1\|_{\dot{H}^{s_1+1}\cap\dot{H}^{s_2}}\\
&\qquad\lesssim\|\na
\psi_0\|_{\dot{H}^{s_1+1}\cap\dot{H}^{s_2}}+\|\na\tilde\psi_0\|_{\dot{H}^{s_1+1}\cap\dot{H}^{s_2+1}}+\|\vv
u_0\|_{\dot{H}^{s_1+1}\cap\dot{H}^{s_2}},
\end{aligned}\eeq
from which, \eqref{1.5} and Theorem \ref{T}, we deduce that the
system \eqref{a14} has a unique global solution $(Y,q)$ which
satisfies \eqref{A1wq} and \eqref{M1} provided that $c_0$ in
\eqref{1.5} is sufficiently small.

Let $X(t,y)\eqdefa y+Y(t,y),$  it follows from \eqref{M1} that
$X(t,y)$ is invertible with respect to $y$ variable and  we denote
its inverse mapping by $X^{-1}(t,x).$ Let \beno
(a_{ij}(t,y))_{i,j=1,2}\eqdefa I+\na_yY(t,y)\quad \mbox{and}\quad
(b_{ij}(t,y))_{i,j=1,2}\eqdefa(a_{ij}(t,y))_{i,j=1,2}^{-1}. \eeno
Then as $\det\,(I+\na_yY)=1,$ $(b_{ij})_{i,j=1,2}$ equals to the
adjoint matrix of $(a_{ij})_{i,j=1,2}$  and
$\sum_{i=1}^2\p_ib_{ij}=0$. With the notations above, we can write
\beno
\begin{split}
&\p_{x_2}\bigl(-\p_1Y^2(t,X^{-1}(t,x))\bigr)-\p_{x_1}\bigl(\p_1Y^1(t,X^{-1}(t,x))\bigr)\\
&=-\p_{x_2}\bigl(a_{21}(t,X^{-1}(t,x))\bigr)
-\p_{x_1}\bigl(a_{11}(t,X^{-1}(t,x))\bigr)\\
&=-\sum_{j=1}^2\bigl[\p_j(b_{j2}a_{21}+b_{j1}a_{11})(t,X^{-1}(t,x))\bigr]=\sum_{j=1}^2[\p_j\d_{j1}](t,X^{-1}(t,x))=0.
\end{split}
\eeno By a similar argument, we have \beno
\p_{x_2}\bigl(-\p_2Y^2(t,X^{-1}(t,x))\bigr)-\p_{x_1}\bigl(\p_2Y^1(t,X^{-1}(t,x))\bigr)=0,\eeno
so that we can define  $(\psi(t,x),\tilde\psi(t,x))$ through
 \beq \label{9.1a}
\begin{split} &
\na_x\psi(t,x)=\bigl(-\p_1Y^2(t, X^{-1}(t,x)),\p_1Y^1(t,
X^{-1}(t,x))\bigr)^T,\quad  \psi(t,x)\to 0\ \ \mbox{as}\ |x|\to \infty,\\
&\na_x\tilde\psi(t,x)=\bigl(-\p_2Y^2(t, X^{-1}(t,x)),\p_2Y^1(t,
X^{-1}(t,x))\bigr)^T,\quad  \tilde\psi(t,x)\to 0\ \ \mbox{as}\
|x|\to \infty,
\end{split}
\eeq and  we define $(\vv u(t,x),p(t,x))$ via \beq \label{9.1}
\begin{split} &\vv u(t,x)\eqdefa Y_t(t, X^{-1}(t,x)), \quad p(t,x)\eqdefa
q(t,X^{-1}(t,x))-|\na_x(x_2+\psi(t,x))|^2.
\end{split} \eeq
Then according to Section \ref{sect2}, $(\phi, \tilde\phi, \vv u,
p)=(x_2+\psi, -x_1+\tilde\psi, \vv u, p)$ thus defined satisfies
\eqref{th1wq} and globally solves the coupled system between
\eqref{1.1} and \eqref{a3}. Then to complete the proof of Theorem
\ref{th1}, it suffices to prove \eqref{th1wr}.

For this, we first notice from  \eqref{9.1} that \beno (\na\vv
u)\circ X(t,y)=\na_yY_t(t,y) \bigl(I+\na_yY(t,y)\bigr)^{-1}, \eeno
which leads to \beq\label{cc1} \begin{aligned} \|\na\vv
u\|_{L^1(\R^+;L^\infty)}
&\lesssim(1+\|\na_yY\|_{L^\infty(\R^+;L^\infty)})\|\na_y
Y_t\|_{L^1(\R^+;L^\infty)}.
\end{aligned}\eeq

Again thanks to \eqref{9.1} and \eqref{9.1a}, we get by applying
Lemma \ref{funct} (ii) with $\Phi=X(t,y)$ that \beq\label{cc2}\begin{aligned}
&\|\vv u\|_{L^\infty(\R^+; \dH^{s_2})}+\|\na\psi\|_{L^\infty(\R^+;\dH^{s_2})}\\
& \quad\leq C(\|\na
Y\|_{L^\infty(\R^+;L^\infty)})\bigl\{\|Y_t\|_{L^\infty(\R^+;
\dH^{s_2})}
+\|\p_1 Y\|_{L^\infty(\R^+;\dH^{s_2})}\\
&\qquad+\|\na
Y\|_{L^\infty(\R^+;\dot{H}^{s_2+1})}\bigl(\|Y_t\|_{L^\infty(\R^+;
L^2)} +\|\p_1 Y\|_{L^\infty(\R^+;L^2)}\bigr)\bigr\}.
\end{aligned}\eeq
Whereas applying  Lemma \ref{funct} (iii) yields
\beq\label{cc3}\begin{aligned} &\|\vv
u\|_{L^2(\R^+;\dH^{s_2+1})}+\|\na\tilde\psi\|_{L^\infty(\R^+;\dH^{s_2+1})}
+\|\na\psi\|_{L^2(\R^+;\dH^{s_2+1})}\\
&\leq C(\|\na Y\|_{L^\infty(\R^+;L^\infty)})\bigl(\|Y_t\|_{L^2(\R^+;\dH^{s_2+1})}+\|\p_2 Y\|_{L^\infty(\R^+;\dH^{s_2+1})}+\|\p_1 Y\|_{L^2(\R^+;\dH^{s_2+1})}\bigr).
\end{aligned}\eeq
In the same manner, we get by applying Lemma \ref{funct} (vi) to
\eqref{9.1} and \eqref{9.1a} that \beq\label{cc4}
\begin{split}
&\|\vv u\|_{L^\infty(\R^+; \dH^{s_1+1})}
+\|\na\psi\|_{L^\infty(\R^+;\dH^{s_1+1})}+\|\na\tilde\psi\|_{L^\infty(\R^+;\dH^{s_1+1})}\\
&\quad+\|\vv u\|_{L^2(\R^+;\dH^{s_1+2})}+\|\na\psi\|_{L^2(\R^+;\dH^{s_1+1})}\\
&\leq C(\|\na Y\|_{L^\infty(\R^+;L^\infty)})(1+\|\D
Y\|_{L^\infty(\R^+;H^{s_1})})
\bigl(\|\na Y_t\|_{L^\infty(\R^+; H^{s_1})}\\
&\quad
+\|\D Y\|_{L^\infty(\R^+; H^{s_1})}
+\|\na Y_t\|_{L^2(\R^+; H^{s_1+1})}+\|\p_1\na Y\|_{L^2(\R^+; H^{s_1})}\bigr).
\end{split}
\eeq

Consequently, we deduce from \eqref{M1},  \eqref{bb0}, \eqref{cc7},
and \eqref{cc1} to \eqref{cc4} that
 \beq\label{cc5}
\begin{split}
&\|\vv u\|_{L^\infty(\R^+; \dH^{s_1+1}\cap\dH^{s_2})}
+\|\na\psi\|_{L^\infty(\R^+;\dH^{s_1+1}\cap\dH^{s_2})}+\|\na\tilde\psi\|_{L^\infty(\R^+;\dH^{s_1+1}\cap\dH^{s_2+1})}\\
&\quad+\|\vv u\|_{L^2(\R^+;\dH^{s_1+2}\cap\dH^{s_2+1})}
+\|\na\psi\|_{L^2(\R^+;\dH^{s_1+1}\cap\dH^{s_2+1})}+\|\na\vv u\|_{L^1(\R^+;L^\infty)}\\
&\lesssim\|\na
\psi_0\|_{\dot{H}^{s_1+1}\cap\dot{H}^{s_2}}+\|\na\tilde\psi_0\|_{\dot{H}^{s_1+1}\cap\dot{H}^{s_2+1}}+\|\vv
u_0\|_{\dot{H}^{s_1+1}\cap\dot{H}^{s_2}},
\end{split}
\eeq provided that \eqref{1.5} holds for $c_0$ sufficiently small.

Next, it follows from \eqref{9.1} that \beno \na_xp=((I+\na
Y)^{-T}\na q)\circ X^{-1}-2\na_x\p_{x_2}\psi-\na_x(|\na_x\psi|^2).
\eeno Applying Lemma \ref{funct} (ii) and Lemma \ref{L5} gives rise
to \beno
\begin{split}
\|\na_xp(t)\|_{\dot{H}^{s_2}}\leq& C(\|\na
Y\|_{L^\infty(\R^+;L^\infty)})
\bigl\{\|((I+\na Y)^{-T}\na q))(t)\|_{\dot{H}^{s_2}}+\|\na Y(t)\|_{\dot{H}^{s_2+1}}\\
&\times\|((I+\na Y)^{-T}\na q))(t)\|_{L^2}+\|\na\psi(t)\|_{\dot{H}^{s_2+1}}+\|(\na\psi)^2(t)\|_{\dot{H}^{s_2+1}}\bigr\}\\
\leq & C(\|\na Y\|_{L^\infty(\R^+;L^\infty)})\bigl\{(1+\|\na Y(t)\|_{\dot{B}^1_{2,1}})\bigl(\|\na q(t)\|_{\dot{H}^{s_2}}
\\
&+\|\na Y(t)\|_{\dot{H}^{s_2+1}}\|\na q(t)\|_{L^2}\bigr)
+(1+\|\na\psi(t)\|_{\dot{B}^1_{2,1}})\|\na\psi(t)\|_{\dot{H}^{s_2+1}}\bigr\},
\end{split}
\eeno which along with \eqref{M1}, \eqref{bb0},  \eqref{cc7} and
\eqref{cc5} implies that
 \beno
\|\na_xp\|_{L^2(\R^+;\dot{H}^{s_2})} \lesssim\|\na
\psi_0\|_{\dot{H}^{s_1+1}\cap\dot{H}^{s_2}}+\|\na\tilde\psi_0\|_{\dot{H}^{s_1+1}\cap\dot{H}^{s_2+1}}+\|\vv
u_0\|_{\dot{H}^{s_1+1}\cap\dot{H}^{s_2}}, \eeno provided that
\eqref{1.5} holds for $c_0$ sufficiently small.

Finally, applying Lemma \ref{funct} (v) and   (iv), along with
\eqref{M1}, \eqref{bb0},  \eqref{cc7}, \eqref{cc5} and Lemma
\ref{L5} yields \beno \|\na_xp\|_{L^2(\R^+;\dot{H}^{s_1})}
\lesssim\|\na
\psi_0\|_{\dot{H}^{s_1+1}\cap\dot{H}^{s_2}}+\|\na\tilde\psi_0\|_{\dot{H}^{s_1+1}\cap\dot{H}^{s_2+1}}+\|\vv
u_0\|_{\dot{H}^{s_1+1}\cap\dot{H}^{s_2}}. \eeno This completes the
proof of  \eqref{th1wr} and thus Theorem \ref{th1}.
\end{proof}

Before we present the proof of Theorem \ref{th2}, we shall first
prove the following blow-up criterion for smooth enough solutions of
\eqref{1.2}.

\begin{prop}\label{propq0}
{\sl Given $\na\psi_0\in H^s(\R^2)$ and $\vv u_0\in H^s(\R^2)$ for
$s>1,$ \eqref{1.2} has a unique solution $(\psi, \vv u)$ on $[0,T]$
for  some $T>0$ so that \beq\label{qs1} \begin{split} &\na\psi\in
C([0,T]; H^s(\R^2)),\quad \vv u\in  C([0,T]; H^{s}(\R^2)),\quad\na\vv u\in L^2((0,T);H^s(\R^2)),\\
&\na p\in C([0,T]; H^{s-1}(\R^2)).
\end{split}
\eeq Moreover, if $T^\ast$ is the lifespan to this solution and
$T^\ast<\infty,$ then \beq \int_0^{T^\ast}\bigl(\|\na\vv
u(t)\|_{L^\infty}+\|\na\psi(t)\|_{L^\infty}^2\bigr)\,dt=\infty.
\label{qs2}\eeq}
\end{prop}

\begin{proof} Given initial data $(\psi_0, \vv u_0),$ it is standard
to prove that \eqref{1.2} has a unique solution $(\psi, \vv u)$ on
$[0,T]$ for some $T>0,$ so that there holds the first line of
\eqref{qs1}. While we get by taking $\dive$ to the momentum equation
of \eqref{1.2} that \beno \na
p=-2\na\p_2\psi+\na(-\D)^{-1}\dive\bigl\{\vv u\cdot\na\vv
u+\dive(\na\psi\otimes\na\psi)\bigr\}, \eeno which along with the
first line of \eqref{qs1} implies that $\na p\in C([0,T];
H^{s-1}(\R^2)).$ This proves \eqref{qs1}.

It remains to verify the blow-up criterion \eqref{qs2}. Toward this,
for any $t<T^\ast,$ we get by using a standard energy estimate for
\eqref{1.2} that \beno
\f12\f{d}{dt}&\bigl(\|\na\psi(t)\|_{L^2}^2+\|\vv
u(t)\|_{L^2}^2\bigr)+\|\na\vv
u(t)\|_{L^2}^2=0,
\eeno
which implies that for any $t<T^\ast$
\beq\label{qs3} \|\na\psi(t)\|_{L^2}^2+\|\vv u(t)\|_{L^2}^2+\|\na\vv
u\|_{L^2_t(L^2)}^2\leq \|\na\psi_0\|_{L^2}^2+\|\vv
u_0\|_{L^2}^2.
\eeq While  acting $\D_j$ to $u^1$ equation of \eqref{1.2} and then
taking the $L^2$ inner product of the resulting equation with
$\D_ju^1$, it leads to  \beq\label{qs4}\begin{split}
&\f12\f{d}{dt}\|\D_ju^1(t)\|_{L^2}^2+\|\na\D_ju^1\|_{L^2}^2\\
&=-\bigl(\D_j\p_1(p+\p_2\psi)\ |\ \D_ju^1\bigr)-\bigl(\D_j(\vv u\cdot\na
u^1)\ |\ \D_ju^1\bigr)-\bigl(\dive\D_j(\p_1\psi\na\psi)\ |\
\D_ju^1\bigr).
\end{split}
\eeq Similarly acting $\D_j$ to $u^2$ equation of \eqref{1.2} and
then taking the $L^2$ inner product of the resulting equation with
$\D_ju^2$ leads to \beq\label{qs5}\begin{split}
&\f12\f{d}{dt}\|\D_ju^2(t)\|_{L^2}^2+\|\na\D_ju^2\|_{L^2}^2+(\D_j\D\psi\
|\ \D_ju^2)\\
&=-\bigl(\D_j\p_2(p+\p_2\psi)\ |\ \D_ju^2\bigr)-\bigl(\D_j(\vv u\cdot\na
u^2)\ |\ \D_ju^2\bigr)-\bigl(\dive\D_j(\p_2\psi\na\psi)\ |\
\D_ju^2\bigr).
\end{split}
\eeq However by the transport equation of \eqref{1.2} and using
integration by parts, one has \beno
\begin{split}
(\D_j\D\psi\ |\ \D_ju^2)=&-\bigl(\D_j\D\psi\ |\
\D_j(\p_t\psi+\vv u\cdot\na\psi)\bigr)\\
=&\f12\f{d}{dt}\|\na\D_j\psi(t)\|_{L^2}^2+\bigl(\D_j\na\psi\ |\
\na\D_j(\vv u\cdot\na\psi)\bigr).
\end{split}
\eeno Hence by combining \eqref{qs4} with \eqref{qs5} and using
$\dive\vv u=0,$ we obtain \beq\label{qs6}\begin{split}
&\f12\f{d}{dt}\bigl(\|\D_j\vv u(t)\|_{L^2}^2+\|\na\D_j\psi(t)\|_{L^2}^2\bigr)+\|\na\D_j\vv u\|_{L^2}^2\\
&=-\bigl(\D_j(\vv u\cdot\na\vv u)\ |\ \D_j\vv u\bigr)-\bigl(\D_j\na\psi\ |\
\na\D_j(\vv u\cdot\na\psi)\bigr)-\bigl(\dive\D_j(\na\psi\otimes\na\psi)\
|\ \D_j\vv u\bigr).
\end{split}\eeq

Next for $s>0,$ we claim that \beq\label{qs7}
\begin{split}
\bigl|\bigl(\D_j(\vv u\cdot\na b)\ |\ \D_jb\bigr)\bigr|\lesssim&
c_j(t)^22^{-2js}\bigl(\|\na\vv
u(t)\|_{L^\infty}\|b(t)\|_{\dH^s}^2\\
&\qquad\qquad\quad+\|\na
b(t)\|_{L^\infty}\|\vv u(t)\|_{\dH^s}\|b(t)\|_{\dH^s}\bigr)\quad\mbox{or}\\
\bigl|\bigl(\D_j(\vv u\cdot\na b)\ |\ \D_jb\bigr)\bigr|\lesssim&
c_j(t)^22^{-2js}\bigl(\|\na\vv u(t)\|_{L^\infty}\|b(t)\|_{\dH^s}^2\\
&\qquad\qquad\quad+\| b(t)\|_{L^\infty}\|\na\vv
u(t)\|_{\dH^s}\|b(t)\|_{\dH^s}\bigr).
\end{split}
\eeq Indeed applying Bony's decomposition \eqref{bb5} for $\vv
u\cdot\na b$ and then using a standard commutator's argument, we can
write  \beno
\begin{split}
\bigl(\D_j(\vv u\cdot\na b)\ |\ \D_jb\bigr)=&\sum_{|j-\ell|\leq
5}\bigl([\D_j; S_{\ell-1}\vv u]\cdot\na\D_\ell
b+(S_{\ell-1}\vv u-S_{j-1}\vv u)\cdot\na\D_\ell \D_jb\ |\ \D_jb\bigr)
\\
&+\bigl(S_{j-1}\vv u\cdot\na\D_jb\ |\ \D_jb\bigr)-\bigl(\D_j(\cR(\vv u,\na
b))\ |\ \D_jb\bigr).
\end{split}
\eeno It follows from the  classical commutator's estimate (see
\cite{bcd} for instance) that \beno\begin{split} \sum_{|j-\ell|\leq
5}\bigl|\bigl([\D_j; S_{\ell-1}\vv u]\cdot\na\D_\ell b\ |\
\D_jb\bigr)\bigr|\lesssim &\sum_{|j-\ell|\leq 5}\|\na
S_{\ell-1}\vv u\|_{L^\infty}\|\D_\ell b\|_{L^2}\|\D_j b\|_{L^2}\\
\lesssim &c_j(t)^22^{-2js}\|\na\vv u(t)\|_{L^\infty}\|b(t)\|_{\dH^s}^2.
\end{split}
\eeno The same estimate holds for $\sum_{|j-\ell|\leq
5}\bigl((S_{\ell-1}\vv u-S_{j-1}\vv u)\cdot\na\D_\ell \D_jb\ |\ \D_jb\bigr)$
and $\bigl(S_{j-1}\vv u\cdot\na\D_jb\ |\ \D_jb\bigr).$

Whereas applying Lemma \ref{le2.1} gives \beno \|\D_j\cR(\vv u,\na
b)\|_{L^2}\lesssim \sum_{\ell\geq j-N_0}\|\D_\ell\vv
u\|_{L^2}\|S_{\ell+2}\na b\|_{L^\infty}, \eeno which can be
controlled by $c_j(t)2^{-js}\|\na b(t)\|_{L^\infty}\|\vv u(t)\|_{\dH^s}$
or $c_j(t)2^{-js}\| b(t)\|_{L^\infty}\|\na\vv u(t)\|_{\dH^s}$ as long
as $s>0.$ This completes the proof of \eqref{qs7}.

Now we go back to \eqref{qs6}. In fact, applying Lemma \ref{L5} (i)
and \eqref{qs7} to \eqref{qs6} gives  \beno
\begin{split}
&\f12\f{d}{dt}\bigl(\|\D_j\vv u(t)\|_{L^2}^2+\|\na\D_j\psi(t)\|_{L^2}^2\bigr)+\|\na\D_j\vv u\|_{L^2}^2\\
&\lesssim c_j(t)^22^{-2js}\bigl\{\|\na\vv u\|_{L^\infty}(\|\vv
u\|_{\dH^s}^2+\|\na\psi\|_{\dH^s}^2)+\|\na\psi\|_{L^\infty}\|\na\vv
u\|_{\dH^s}\|\na\psi\|_{\dH^s}\bigr\}\quad\mbox{for}\ s>0.
\end{split}
\eeno The above implies for any $s>0$ \beno
\begin{split}
&\|\vv u(t)\|_{\dH^s}^2+\|\na\psi(t)\|_{\dH^s}^2+\|\na\vv
u\|_{L^2_t(\dH^s)}^2\\
&\leq \|\vv u_0\|_{\dH^s}^2+\|\na\psi_0\|_{\dH^s}^2+C\int_0^t(\|\na\vv
u(t')\|_{L^\infty}+\|\na\psi(t')\|_{L^\infty}^2)(\|\vv u(t')\|_{\dH^s}^2+\|\na\psi(t')\|_{\dH^s}^2)\,dt'.
\end{split}
\eeno Applying Gronwall's inequality yields \beno
\begin{split}
&\|\vv u(t)\|_{\dH^s}^2+\|\na\psi(t)\|_{\dH^s}^2+\|\na\vv
u\|_{L^2_t(\dH^s)}^2\\
&\leq
(\|\vv u_0\|_{\dH^s}^2+\|\na\psi_0\|_{\dH^s}^2)\exp\Bigl\{C\int_0^t(\|\na\vv
u(t')\|_{L^\infty}+\|\na\psi(t')\|_{L^\infty}^2)\,dt'\Bigr\}\quad\mbox{for}\quad
t<T^\ast,
\end{split}
\eeno which together with \eqref{qs3} implies \eqref{qs2}. This
completes the proof of Proposition \ref{propq0}.
\end{proof}

In order to apply Theorem \ref{th1} to prove Theorem \ref{th2}, we
also need the following lemma concerning the existence of $
\tilde{\psi}_0$ so that there holds \eqref{a2}.

\begin{lem}\label{lemf1}
{\sl Under the assumptions of Theorem \ref{th2}, \eqref{a2} has a
solution  $\tilde{\psi}_0\in H^{s_1+2}(\R^2)$ so that there holds
\beq\label{f2qw} \|\tilde\psi_0\|_{H^{s_1+2}} \leq
C(K,\|\na\psi_0\|_{H^{s_1+2}})\|\p_{x_2}\psi_0\|_{H^{s_1+2}}.\eeq}
\end{lem}

The proof of Lemma \ref{lemf1} will be postponed in the Appendix
\ref{appenda}. We now turn to the proof of Theorem \ref{th2}.

\begin{proof}[Proof of Theorem \ref{th2}]
Under the assumption of Theorem \ref{th2}, we deduce from Lemma
\ref{lemf1} that there exists a $\tilde\psi_0$ so that there holds
\eqref{f2qw} and \eqref{a2}. Notice that for $s_2\in (-1,-\f12)$,
then it is easy to observe that \beno
\|\na\tilde\psi_0\|_{\dH^{s_1+1}\cap\dH^{s_2+1}}\lesssim\|\tilde\psi_0\|_{H^{s_1+2}}\leq
C(\|\na\psi_0\|_{H^{s_1+2}}) \|\p_{x_2}\psi_0\|_{H^{s_1+2}} \eeno
Therefore, under the assumption of \eqref{1.5a},  we infer from
Theorem \ref{th1} that the coupled system \eqref{1.1} and \eqref{a3}
has a unique global solution $(\phi, \tilde\phi, \vv u,
p)=(x_2+\psi, -x_1+\tilde\psi, \vv u, p)$ so that there holds
\eqref{th1wq} and \eqref{th1wr}. Then according to the discussions
at the beginning of Section \ref{sect2}, $(\psi, \vv u, p)$ thus
obtained solves \eqref{1.2}, which is in fact the unique solution of
\eqref{1.2} with initial data $(\psi_0, \vv u_0),$ and there holds
\eqref{th1wra}.

On the other hand, thanks to Proposition \ref{propq0}, given initial
data $(\psi_0,\vv u_0)$ with $\na \psi_0\in H^s(\R^2),\, \vv u_0\in
H^{s}(\R^2)$, \eqref{1.2} has a unique solution $(\psi, \vv u, p)$
with $\na\psi\in C([0,T]; H^s(\R^2)),\, \vv u\in  C([0,T];
H^{s}(\R^2)),\, \na\vv u\in L^2((0,T);H^{s}(\R^2)),\,\na p\in
C([0,T]; H^{s-1}(\R^2))$ for any given $T<T^*$. Moreover, if
$T^*<\infty$, there holds \eqref{qs2}. Due to the uniqueness, this
solution must coincide with the one obtained in the last paragraph.
By virtue of  \eqref{th1wra}, \eqref{qs2} can not be true for any
finite $T^*$. Therefore $T^*=\infty$ and there holds \eqref{th1wqa}.
This completes the proof of Theorem \ref{th2}.
\end{proof}

\smallskip

\appendix
\setcounter{equation}{0}
\section{The Sobolev estimates to a function composed with a measure preservation mapping }\label{appendc}

\begin{lem}\label{funct}
{\sl Let $\Phi(y)=y+\Psi(y)$ be a diffeomorphism from $\R^2$ to
$\R^2$ with $\det\,({\na_y\Phi})=\det\,(I+\na_y\Psi)=1$, and
$\Phi^{-1}(x)$ be its inverse mapping. Then for any smooth function
$u, v,$ there hold

\no (i) if $s=0$,
 \beno
 \|u\circ\Phi\|_{\dot{H}^0}=\|u\|_{\dot{H}^0}\quad\text{and}\quad \|v\circ\Phi^{-1}\|_{\dot{H}^0}=\|v\|_{\dot{H}^0};
 \eeno
 (ii) if $-1<s<0$,
 \beno
&&\|u\circ\Phi\|_{\dot{H}^s}\lesssim
(1+\|\na_y\Psi\|_{L^\infty})^{s+3}\|u\|_{\dot{H}^s}
+(1+\|\na_y\Psi\|_{L^\infty})\|\na_y\Psi\|_{\dot{H}^{s+1}}\|u\|_{L^2}\quad\text{and}\\
&&\|v\circ\Phi^{-1}\|_{\dot{H}^s}\lesssim
(1+\|\na_y\Psi\|_{L^\infty})^{s+3}\|v\|_{\dot{H}^s}
+(1+\|\na_y\Psi\|_{L^\infty})\|\na_y\Psi\|_{\dot{H}^{s+1}}\|v\|_{L^2};
\eeno (iii) if $0<s<1$,
 \beno &&\|u\circ\Phi\|_{\dot{H}^s}\lesssim
(1+\|\na_y\Psi\|_{L^\infty})^{s+1}\|u\|_{\dot{H}^s}\quad\text{and}
\\
&& \|v\circ\Phi^{-1}\|_{\dot{H}^s}\lesssim
(1+\|\na_y\Psi\|_{L^\infty})^{s+1}\|v\|_{\dot{H}^s}; \eeno (iv) if
$s=1$, \beno &&\|u\circ\Phi\|_{\dot{H}^1}\lesssim
(1+\|\na_y\Psi\|_{L^\infty})\|u\|_{\dot{H}^1}\quad\text{and}
\\
&& \|v\circ\Phi^{-1}\|_{\dot{H}^1}\lesssim
(1+\|\na_y\Psi\|_{L^\infty})\|v\|_{\dot{H}^1}; \eeno (v) if
$1<s\leq2$,  \beno &&\|u\circ\Phi\|_{\dot{H}^s}\lesssim
(1+\|\na_y\Psi\|_{L^\infty})^{s+1}\|u\|_{\dot{H}^s}+\|\na_y\Psi\|_{\dot{H}^s}\|\na_xu\|_{L^2}\quad\text{and}
\\
&& \|v\circ\Phi^{-1}\|_{\dot{H}^s}\lesssim
(1+\|\na_y\Psi\|_{L^\infty})^{s+1}\|v\|_{\dot{H}^s}+\|\na_y\Psi\|_{\dot{H}^s}\|\na_yv\|_{L^2};
\eeno (vi) if $s>2$, \beno &&\|u\circ\Phi\|_{\dot{H}^s}\lesssim
(1+\|\na_y\Psi\|_{L^\infty})^{s+1}(1+\|\D_y\Psi\|_{H^{s-2}})\|\na_xu\|_{H^{s-1}}\quad\text{and}
\\
&& \|v\circ\Phi^{-1}\|_{\dot{H}^s}\lesssim
(1+\|\na_y\Psi\|_{L^\infty})^{s+1}(1+\|\D_y\Psi\|_{H^{s-2}})\|\na_yv\|_{H^{s-1}},
\eeno }
\end{lem}

\begin{proof} We denote
\beno \cA=(a_{ij})_{i,j=1,2}\eqdefa I+\na_y\Psi,\quad
\cB=(b_{ij})_{i,j=1,2}\eqdefa (I+\na_y\Psi)^{-1}. \eeno Due to
$\det\cA=1,$ the matrix $\cB$ equals to the adjoint matrix of $\cA$.
This leads to \beq\label{app5}
(\p_{x_i}u)\circ\Phi=\sum_{j=1}^2b_{ji}\p_{y_j}(u\circ\Phi)\quad\text{and}
\quad(\p_{y_i}v)\circ\Phi^{-1}=\sum_{j=1}^2a_{ji}\circ\Phi^{-1}\p_{x_j}(v\circ\Phi^{-1}).
\eeq

In what follows, we shall only present  the  proof of the related
estimates involving $u\circ\Phi$, and the ones involving
$v\circ\Phi^{-1}$ is identical. Firstly it follows from
$\det\,(\na_y\Phi)=1$ that  \beno
\|u\circ\Phi\|_{\dot{H}^0}=\|u\|_{\dot{H}^0}. \eeno When $s\in
(0,1),$ we obtain from \beno
\|f\|_{\dot{H}^s}^2\sim\int_{\R^2\times\R^2}\f{|f(x)-f(y)|^2}{|x-y|^{2+2s}}\,dx\,dy,
\eeno  that \beq\label{app6}
\|u\circ\Phi\|_{\dot{H}^s}\lesssim(1+\|\na_y\Psi\|_{L^\infty})^{s+1}\|u\|_{\dot{H}^s}.
\eeq For the case  $s\in (-1,0),$ we get, by using \eqref{app5},
that\beno
u\circ\Phi=-(\na_x\cdot\na_x(-\D_x)^{-1}u)\circ\Phi=-\cB^T\na_y\cdot((\na_x(-\D_x)^{-1}u)\circ\Phi),
\eeno which combining  with (iii) of Lemma \ref{L5} leads to
\beno\begin{aligned}
\|u\circ\Phi\|_{\dot{H}^s}&\lesssim\|\na_y\cdot((\na_x(-\D_x)^{-1}u)\circ\Phi)\|_{\dot{H}^s}
+\|(\cB^T-I)\na_y\cdot((\na_x(-\D_x)^{-1}
u)\circ\Phi)\|_{\dot{H}^s}\\
&\lesssim(1+\|\na_y\Psi\|_{L^\infty})\|(\na_x(-\D_x)^{-1}u)\circ\Phi\|_{\dot{H}^{s+1}}\\
&\quad+\|\na_y\Psi\|_{\dot{H}^{s+1}}\|\na_y(\na_x(-\D_x)^{-1}u)\circ\Phi)\|_{L^2}.
\end{aligned}\eeno
Applying \eqref{app5} and \eqref{app6}, one thus obtains  for $s\in
(-1,0)$ that \beq\label{app10}
\|u\circ\Phi\|_{\dot{H}^s}\lesssim(1+\|\na_y\Psi\|_{L^\infty})^{s+3}\|u\|_{\dot{H}^s}+
(1+\|\na_y\Psi\|_{L^\infty})\|\na_y\Psi\|_{\dot{H}^{s+1}}\|u\|_{L^2}.
\eeq Whereas we deduce from \eqref{app5} that \beq\label{app7}
\|u\circ\Phi\|_{\dot{H}^1}\lesssim\|(\na_xu)\circ\Phi(I+\na_y\Psi)\|_{L^2}
\lesssim(1+\|\na_y\Psi\|_{L^\infty})\|u\|_{\dot{H}^1}. \eeq To
handle the case that  $1<s\leq2$, we first use \eqref{app5} and then
Lemma \ref{L5} (iii) to  deduce \beq\label{app17}\begin{aligned}
\|u\circ\Phi\|_{\dot{H}^s}&\lesssim\|(\na_xu)\circ\Phi(I+\na_y\Psi)\|_{\dot{H}^{s-1}}\\
&\lesssim(1+\|\na_y\Psi\|_{L^\infty})\|(\na_xu)\circ\Phi\|_{\dot{H}^{s-1}}
+\|\na_y\Psi\|_{\dot{H}^s}\|(\na_xu)\circ\Phi\|_{L^2},
\end{aligned}\eeq
which along with \eqref{app6} and \eqref{app7} ensures
\beq\label{app8}
\|u\circ\Phi\|_{\dot{H}^s}\lesssim(1+\|\na_y\Psi\|_{L^\infty})^{s+1}\|u\|_{\dot{H}^s}
+\|\na_y\Psi\|_{\dot{H}^s}\|\na_xu\|_{L^2}. \eeq

For $k<s-1\leq k+1$ $(k\in\N)$, applying \eqref{app17} repeatedly,
we obtain \beq\label{app9a}\begin{aligned}
\|u\circ\Phi\|_{\dot{H}^{s-1}}
&\lesssim(1+\|\na_y\Psi\|_{L^\infty})^{k}\|(\na_x^{k}u)\circ\Phi\|_{\dot{H}^{s-(k+1)}}\\
&\quad+\sum_{j=2}^{k+1}
(1+\|\na_y\Psi\|_{L^\infty})^{j-2}\|\na_y\Psi\|_{\dot{H}^{s+1-j}}\|\na_x^{j-1}u\|_{L^2}.
\end{aligned}\eeq

On the other hand, thanks to \eqref{app5} and (i) of Lemma \ref{L5},
one has
\beno\|u\circ\Phi\|_{\dot{H}^s}\lesssim(1+\|\na_y\Psi\|_{L^\infty})\|(\na_xu)\circ\Phi\|_{\dot{H}^{s-1}}
+\|\na_y\Psi\|_{\dot{H}^{s-1}}\|\na_xu\|_{L^\infty}, \eeno this
combining with \eqref{app9a} yields for $k+1<s\leq k+2$ $(k\in\N)$
that \beq\label{app11}\begin{aligned} \|u\circ\Phi\|_{\dot{H}^{s}}
\lesssim&(1+\|\na_y\Psi\|_{L^\infty})^{k+1}\|(\na_x^{k+1}u)\circ\Phi\|_{\dot{H}^{s-(k+1)}}\\
&+\sum_{j=2}^{k+1}
(1+\|\na_y\Psi\|_{L^\infty})^{j-1}\|\na_y\Psi\|_{\dot{H}^{s+1-j}}\|\na_x^ju\|_{L^2}+\|\na_y\Psi\|_{\dot{H}^{s-1}}\|\na_xu\|_{L^\infty}\\
\lesssim&(1+\|\na_y\Psi\|_{L^\infty})^{k+1}\bigl(\|(\na_x^{k+1}u)\circ\Phi\|_{\dot{H}^{s-(k+1)}}
+\|\D_y\Psi\|_{H^{s-2}}\|\na_xu\|_{H^{s-1}}\bigr),
\end{aligned}\eeq
where we used the embedding inequality \beno
\|\na_xu\|_{L^\infty}\lesssim\|\na_xu\|_{H^{s-1}}\quad\text{for}\quad
s>2. \eeno

By \eqref{app6} and \eqref{app11}, we finally obtain
\beq\label{app16} \|u\circ\Phi\|_{\dot{H}^s}\lesssim
(1+\|\na_y\Psi\|_{L^\infty})^{s+1}(1+\|\D_y\Psi\|_{H^{s-2}})\|\na_xu\|_{H^{s-1}}.
\eeq This completes the proof of Lemma \ref{funct}.
\end{proof} \smallskip

\setcounter{equation}{0}
\section{The Proof of Lemmas \ref{L3} and \ref{L4}}\label{appendb}

\begin{proof}[Proof of Lemma \ref{L3}]
We first get by applying Bony's decomposition \eqref{bb5} and
\eqref{bb6} that \beq\label{bb20}\begin{aligned}
\p_2Y^1\p_1Y_t=&\bigl(TT^h+T\cR^h+\bar{T}T^h+\bar{T}\bar{T}^h
+\bar{T}R^h+RT^h+R\cR^h\bigr)(\p_2Y^1,\p_1Y_t).
\end{aligned}\eeq
Since \beno
\begin{split}
&\|\D_{j'}S_{k'+2}^h\p_1Y_t(t)\|_{L^2}\lesssim
d_{j',k'}(t)2^{-\f{3j'}2}2^{\f{k'}2}\|Y_t(t)\|_{\cB^{\f32,\f12}}\quad\mbox{and since}\\
&\|S_{j'-1}\D_{k'}^h\p_2Y^1(t)\|_{L^2_h(L^\infty_v)}\lesssim
d_{k'}(t)2^{-\f{k'}2}\|\p_2Y^1(t)\|_{\cB^{\f12,\f12}},
\end{split}
\eeno by applying Lemma \ref{L1} and Lemma \ref{le2.1}, we have
\beno\begin{aligned}
\|\D_j\D_k^h&(T\cR^h(\p_2Y^1,\p_1Y_t))(t)\|_{L^2}\\ &\lesssim
2^{\f{k}{2}}\sum_{{|j'-j|\leq4} \atop {j'+N_0\geq k'\geq
k-N_0}}\|S_{j'-1}\D_{k'}^h\p_2Y^1(t)\|_{L^2_h(L^\infty_v)}
\|\D_{j'}S_{k'+2}^h\p_1Y_t(t)\|_{L^2}\\
&\lesssim 2^{\f{k}{2}}\sum_{{|j'-j|\leq4} \atop { k'\geq k-N_0}}
d_{j',k'}(t)2^{-j'}2^{-\f{k'}{2}}\|\p_2Y^1(t)\|_{\cB^{\f{1}{2},\f{1}{2}}}
\|Y_t(t)\|_{\cB^{\f{3}{2},\f{1}{2}}}\\
&\lesssim
d_{j,k}(t)2^{-j}\|\p_2Y^1(t)\|_{\dot{B}^{1}_{2,1}}\|Y_t(t)\|_{\cB^{\f{3}{2},\f{1}{2}}},
\end{aligned}\eeno
The same estimate holds for
$\|\D_j\D_k^h(TT^h(\p_2Y^1,\p_1Y_t))(t)\|_{L^2}.$

Following the same line of the arguments, one has
\beno\begin{aligned}
\|\D_j\D_k^h&(R\cR^h(\p_2Y^1,\p_1Y_t))(t)\|_{L^2}\\
&\lesssim2^{\f{j}{2}}2^{\f{k}{2}}\sum_{{j'\geq j-N_0} \atop
{j'+N_0\geq k'\geq k-N_0}}
\|\D_{j'}\D_{k'}^h\p_2Y^1(t)\|_{L^2}\|\wt{\D}_{j'}S_{k'+2}^h\p_1Y_t(t)\|_{L^2}\\
&\lesssim2^{\f{j}{2}}2^{\f{k}{2}}\sum_{{j'\geq j-N_0} \atop { k'\geq
k-N_0}}d_{j',k'}(t)2^{-\f{3j'}{2}}2^{-\f{k'}{2}}\|\p_2Y^1(t)\|_{\dot{B}^1_{2,1}}
\|Y_t(t)\|_{\cB^{\f{3}{2},\f{1}{2}}}\\
&\lesssim
d_{j,k}(t)2^{-j}\|\p_2Y^1(t)\|_{\dot{B}^1_{2,1}}\|Y_t(t)\|_{\cB^{\f{3}{2},\f{1}{2}}}.
\end{aligned}\eeno
Similar estimate holds for
$\|\D_j\D_k^h(\bar{T}T^h(\p_2Y^1,\p_1Y_t))(t)\|_{L^2},$
$\|\D_j\D_k^h(\bar{T}R^h(\p_2Y^1,\p_1Y_t))(t)\|_{L^2},$ and
$\|\D_j\D_k^h(RT^h(\p_2Y^1,\p_1Y_t))(t)\|_{L^2}.$

Finally as
$$\|S_{j'-1}S_{k'-1}^h\p_1Y_t(t)\|_{L^\infty}\lesssim
d_{k'}(t)2^{k'}\|Y_t(t)\|_{\cB^{\f{1}{2},\f{1}{2}}},$$
 we get by applying Lemma \ref{le2.1}  that
\beno\begin{aligned}
\|\D_j\D_k^h(\bar{T}\bar{T}^h(\p_2Y^1,\p_1Y_t))(t)\|_{L^2}
&\lesssim\sum_{{|j'-j|\leq4} \atop
{|k'-k|\leq4}}\|\D_{j'}\D_{k'}^h\p_2Y^1(t)\|_{L^2}
\|S_{j'-1}S_{k'-1}^h\p_1Y_t(t)\|_{L^\infty}\\
 &\lesssim \sum_{{|j'-j|\leq4}\atop{
|k'-k|\leq4}}d_{j',k'}(t)2^{-j'}\|\p_1Y^1(t)\|_{\dot{B}^2_{2,1}}
\|Y_t(t)\|_{\cB^{\f{1}{2},\f{1}{2}}}\\
&\lesssim
d_{j,k}(t)2^{-j}\|\p_1Y^1(t)\|_{\dot{B}^2_{2,1}}\|Y_t(t)\|_{\cB^{\f{1}{2},\f{1}{2}}}.
\end{aligned}\eeno

Substituting the above estimates into \eqref{bb20} and integrating
the resulting inequality over $(0,T),$ we complete the proof of
Lemma \ref{L3}.
\end{proof}

\begin{proof}[Proof of Lemma \ref{L4}]
Thanks to Definition \ref{def2} and Lemmas \ref{L1}, \ref{L2}, we obtain that
\beno\begin{aligned}
&\|\p_1(\p_2Y^1\p_2Y_t)\|_{L^1_T(\cB^{0,0})}=\|\p_1\p_2Y^1\p_2Y_t+\p_2Y^1\p_1\p_2Y_t\|_{L^1_T(\cB^{0,0})}\\
&\lesssim\|\p_1\p_2Y^1\|_{L^\infty_T(\cB^{0,0})}\|\p_2Y_t\|_{L^1_T(\cB^{\f12,\f12})}+\|\p_2Y^1\|_{L^\infty_T(\cB^{\f12,\f12})}
\|\p_1\p_2Y_t\|_{L^1_T(\cB^{0,0})}\\
&\lesssim\|\p_2Y^1\|_{L^\infty_T(\dot{B}^1_{2,1})}\|Y_t\|_{L^1_T(\cB^{\f32,\f12})},
\end{aligned}
\eeno
which finishes the proof of Lemma \ref{L4}.

\end{proof}

\smallskip

\setcounter{equation}{0}
\section{The Proof of Lemma \ref{lemf1}}\label{appenda}

\begin{lem}\label{C.1}
{\sl Let  $s>2$ and $f\in H^s(\R^2)$ with $\Supp f(\cdot,x_2)\subset
[-K,K]$ for some positive constant $K.$ Let $\vv b=(b^1,b^2)^T$ be a
divergence free vector field with $\na\vv b\in H^{s-1}(\R^2)$ and
$b^1\geq\f12.$ We assume moreover that $f$ and $\vv b$ are
admissible on $\{0\}\times\R$ in the sense of Definition
\ref{def1.1ad}.
  Then \eqref{d1} has a solution
$\psi\in H^s(\R^2)$ so that there holds \beq\label{d2}
\|\psi\|_{H^s}\leq C\bigl(K,\|\na\vv b\|_{H^{s-1}}\bigr)\|f\|_{H^s}.
\eeq}
\end{lem}
\begin{proof}  Since $\na\vv b\in H^{s-1}(\R^2)\subset L^\infty(\R^2)$, \eqref{d3}  has a unique global solution on $\R$ so that
 for all $t\in\R,$ \beq \label{d4} \|\na
X(t,\cdot)\|_{L^\infty}\leq \exp\Bigl(\|\na\vv
b\|_{L^\infty}|t|\Bigr)\quad\andf\quad \det\Bigl(\f{\p X(t,x)}{\p
x}\Bigr)=1. \eeq While it follows from \eqref{d3} and \eqref{d1}
that \beno \frac{d}{dt}\psi(X(t,x))=f(X(t,x)), \eeno from which, we
define \beno \psi(x)= \left\{\begin{array}{l}
\displaystyle -\int_0^\infty f(X(t,x))\,dt\quad\mbox{if}\quad x_1\geq 0, \\
\displaystyle \ \ \int_{-\infty}^0 f(X(t,x))\,dt\quad\mbox{if}\quad
x_1\leq 0.
\end{array}\right. \eeno
Thanks to the assumption that $f$ and $\vv b$ are admissible on
$\{0\}\times\R$ in the sense of Definition \ref{def1.1ad}, the
values of $\psi(x)$ at $(0,x_2)$ are compatible. We remark  that the
equation \eqref{d1}, $ b^1\p_1\psi=-b^2\p_2\psi+f,$ and
$b^1\geq\f12$ implies that the derivatives of $\psi$ in the $x_2$
variable yields the derivatives of $\psi$ with respect to $x_1$
variable. Therefore, we do not require any admissible condition for
the derivatives of $f$ and $\vv b.$

On the other hand, as $b^1\geq \f12,$ we deduce from \eqref{d3} that
\beno && X^1(t,x)\geq
x_1+\f{t}2\geq K \quad\ \ \mbox{if}\quad t\geq 2K,\,\ \  x_1\geq 0,\\
 && X^1(t,x)\leq
x_1+\f{t}2\leq -K \quad \mbox{if}\quad t\leq -2K,\, x_1\leq 0,\eeno
which together with the assumption: $\Supp f(\cdot,x_2)\subset
[-K,K]$ for some positive constant $K,$ implies that \beq \label{d5}
\psi(x)= \left\{\begin{array}{l}
\displaystyle -\int_0^{2K} f(X(t,x))\,dt\quad\mbox{if}\quad x_1\geq 0, \\
\displaystyle \ \ \int_{-2K}^0 f(X(t,x))\,dt\quad\mbox{if}\quad
x_1\leq 0.
\end{array}\right. \eeq

It remains to prove \eqref{d2}. Indeed for any $s>2,$ we deduce from
\eqref{d3} and product laws in Sobolev spaces that \beno
\begin{aligned} \|\na_xX(t,\cdot)-I\|_{H^{s-1}}\lesssim &
\int_0^t\Bigl(\|\na\vv b\|_{L^\infty}\|\na_xX(t',\cdot)-I\|_{H^{s-1}}\\
&+\|(\na\vv
b)(X(t',\cdot))\|_{H^{s-1}}\bigl(1+\|\na_xX(t',\cdot)-I\|_{L^\infty}\bigr)\Bigr)\,dt',
\end{aligned}
\eeno from which, \eqref{d4} and Lemma \ref{funct}, we infer, for $
|t|\leq 2K,$ that \beno
\begin{aligned} \|\na_xX(t,\cdot)-I\|_{H^{s-1}}
\leq & C\bigl(K,\|\na\vv b\|_{H^{s-1}}\bigr)\Bigl(\|\na\vv
b\|_{H^{s-1}}+\int_0^t\|\na_xX(t',\cdot)-I\|_{H^{s-1}}\,dt'\Bigr).
\end{aligned}
\eeno Applying Gronwall's Lemma gives rise to
 \beq\label{d7} \max_{|t|\in
[0,2K]}\|\na_xX(t,\cdot)-I\|_{H^{s-1}}\leq C\bigl(K,\|\na\vv
b\|_{H^{s-1}}\bigr)\|\na\vv b\|_{H^{s-1}}. \eeq

By virtue of Lemma \ref{funct} and \eqref{d4}, we thus deduce from
\eqref{d5} that \beno
\|\psi\|_{H^s}\lesssim\int_{-2K}^{2K}C(\|\na_xX(t',\cdot)-I\|_{L^\infty})\bigl(1+\|\na_xX(t',\cdot)-I\|_{H^{s-1}}\bigr)\|f\|_{H^s}\,dt',
\eeno which along with \eqref{d4} and \eqref{d7} implies \eqref{d2}.
And Lemma \ref{C.1} is proved.
\end{proof}

We now turn to the proof of Lemma \ref{lemf1}.

\begin{proof}[{Proof of Lemma \ref{lemf1}}] We first deduce from
\eqref{a2} that $\tilde{\psi}_0$ satisfies \beq\label{f2}
(1+\p_{x_2}\psi_0)\p_{x_1}\tilde{\psi}_0-\p_{x_1}\psi_0\p_{x_2}\tilde{\psi}_0=\p_{x_2}\psi_0,\quad\mbox{on}\quad
\R^2.\eeq It is easy to observe that $\dive\bigl(1+\p_{x_2}\psi_0,
-\p_{x_1}\psi_0\bigr)^T=0,$ and $\|\p_{x_2}\psi_0\|_{L^\infty}\leq
Cc_0\leq \f12$ for $c_0$ sufficiently small in \eqref{1.5a}.
Applying Lemma \ref{C.1} then ensures that \eqref{f2} has a solution
$\tilde{\psi}_0$ which satisfies \eqref{f2qw}. This completes the
proof of Lemma \ref{lemf1}.
\end{proof}

\bigskip

\noindent {\bf Acknowledgments.}  We would like to thank Dr. Zhen
Lei for pointing out a mistake in the earlier version of Lemma
\ref{lemf1}. F. Lin is partially supported by
the NSF grants DMS 1065964 and DMS 1159313. Both L. XU and P. Zhang
are supported by innovation grant from National Center for
Mathematics and Interdisciplinary Sciences. L. Xu is partially
supported by NSF of China under Grant 11201455. P. Zhang is
partially supported by NSF of China under Grant 10421101 and
10931007, the one hundred talents' plan from Chinese Academy of
Sciences under Grant GJHZ200829.

\end{document}